\documentclass[11pt]{article}
\usepackage[margin=1in]{geometry}
\usepackage{amsmath,amsthm,amssymb,amsfonts}
\usepackage{algorithm} 
\usepackage{algpseudocode}
\usepackage[hidelinks]{hyperref}
\usepackage{tabularx}
\usepackage{booktabs}

\usepackage{mathtools}
\usepackage{tcolorbox}
\usepackage{bbm}
\usepackage{todonotes}
\usepackage{tensor}
\usepackage[font=small,labelfont=bf]{caption}
\usepackage{subcaption}
\usepackage{graphicx}
\usepackage{sectsty}

\usepackage[noadjust]{cite}

\newcommand{\RR}{\mathbb{R}}
\newcommand{\symMat}{\mathbb{S}}

\usepackage{accents}

\mathchardef\mhyphen="2D 
\DeclareMathOperator*{\argmin}{\arg\!\min}
\DeclareMathOperator*{\argmax}{\arg\!\max}

\newtheorem{theorem}{Theorem}[section]
\newtheorem{definition}{Definition}[section]

\newtheorem{remark}{Remark}[section]
\newtheorem{lemma}{Lemma}[section]

\newtheorem{proposition}{Proposition}[section]

\newtheorem{example}{Example}[section]

\newtheorem{assumption}{Assumption}[section]

\usepackage[nameinlink]{cleveref}
\crefname{equation}{}{}
\crefname{theorem}{Theorem}{Theorems}
\crefname{corollary}{Corollary}{Corollaries}
\crefname{example}{Example}{Examples}
\crefname{assumption}{Assumption}{Assumptions}
\crefname{lemma}{Lemma}{Lemmas}
\crefname{proposition}{Proposition}{Propositions}
\crefname{figure}{Figure}{Figures}
\crefname{table}{Table}{Tables}
\crefname{section}{Section}{Sections}
\crefname{appendix}{Appendix}{Appendices}
\Crefname{equation}{}{}
\Crefname{theorem}{Theorem}{Theorems}
\Crefname{corollary}{Corollary}{Corollaries}
\Crefname{example}{Example}{Examples}
\Crefname{lemma}{Lemma}{Lemma}
\Crefname{proposition}{Proposition}{Propositions}
\Crefname{figure}{Figure}{Figures}
\Crefname{table}{Table}{Tables}
\Crefname{section}{Section}{Sections}
\Crefname{appendix}{Appendix}{Appendices}

\usepackage{titlesec} 
\titleformat{\subsubsection}[runin]{\normalfont\bfseries}{\thesubsubsection.}{10pt}{}

\usepackage{adjustbox} 

\newcommand{\SBMP}{$(r_\mathrm{p},r_\mathrm{c})$\textsf{-SBMP}}
\newcommand{\SBMD}{$(r_\mathrm{p},r_\mathrm{c})$\textsf{-SBMD}}

\newcommand{\vectorize}[1]{\mathrm{vec}\left(#1\right)}
\newcommand{\Pstar}{\mathcal{P}^\star}
\newcommand{\Dstar}{\mathcal{D}^\star}
\newcommand{\pstar}{p^\star}
\newcommand{\dstar}{d^\star}
\newcommand{\Xstar}{X^\star}
\newcommand{\ystar}{y^\star}
\newcommand{\Zstar}{Z^\star}
\newcommand{\Amap}{\mathcal{A}}
\newcommand{\Ajmap}{\Amap^*}
\newcommand{\Nmap}{\mathcal{N}}
\newcommand{\Njmap}{\Nmap^*}
\newcommand{\DZstar}{\mathcal{D}_{Z_\star}}
\newcommand{\DXstar}{\mathcal{D}_{X_\star}}
\newcommand{\Dxo}{\mathcal{D}_{\Omega_0}}
\newcommand{\lambdamax}{\lambda_{\max}}
\newcommand{\Dist}{\mathrm{dist}}
\newcommand{\Wtstar}{W_t^\star}
\newcommand{\ytstar}{y_t^\star}

\newcommand{\Ft}{\bar{F}_t}

\newcommand{\BoundOnTp}{\bigO \left( \frac{16 \times 9\max\{\|C\|^2_{\mathrm{op}}, 1\}}{\mu \beta (1-\beta)^2 \min \{\alpha,\mu\} \delta^2 }\right)}

\newcommand{\BoundOnetaP}{4 \rho \max \biggl\{ \frac{144 \sup_{\Xstar \in \Pstar} \|\Xstar\|_{\mathrm{op}}}{\delta^2},\frac{9(8\sqrt{2}+16)}{\delta} \biggr\}}

\newcommand{\ep}{\epsilon_{\mathrm{p}}}
\newcommand{\ed}{\epsilon_{\mathrm{d}}}
\newcommand{\eg}{\epsilon_{\mathrm{g}}}

\newcommand{\ee}{\mathrm{e}}

\newcommand{\etac}{\eta_{\mathrm{c}}}

\newcommand{\ml}{m_{\mathrm{l}}}
\newcommand{\mr}{m_{\mathrm{r}}}

\newcommand{\dm}{n}

\newcommand{\inprod}[2]{\langle #1, #2 \rangle}

\DeclareMathOperator{\dist}{dist}

\newcommand{\Trace}{\mathop{\bf tr}}
\newcommand{\tr}{\mathsf{ T}}

\newcommand{\ynextposi}{\omega_{t+1}}
\newcommand{\ycurrposi}{\omega_{t}}
\newcommand{\yinitposi}{\omega_{0}}
\newcommand{\ycandidate}{y_{t+1}^\star}
\newcommand{\candidate}{X_{t+1}^\star}
\newcommand{\initposi}{\Omega_{0}}
\newcommand{\currposi}{\Omega_{t}}
\newcommand{\nextposi}{\Omega_{t+1}}
\newcommand{\xcandidate}{x_{t+1}^\star}

\newcommand{\rpast}{r_{\mathrm{p}}} 
\newcommand{\rcurrent}{r_{\mathrm{c}}}
\newcommand{\rnull}{\sigma_{\mathrm{p}}}
\newcommand{\rnulld}{\sigma_{\mathrm{d}}}
\newcommand{\bigO}{\mathcal{O}}

\newcommand{\faceplus}[1]{\mathcal{C}_{#1}^{+}}%
\newcommand{\myparagraph}[1]{\textbf{#1}.}

\let\originalleft\left
\let\originalright\right
\renewcommand{\left}{\mathopen{}\mathclose\bgroup\originalleft}
\renewcommand{\right}{\aftergroup\egroup\originalright}

\usepackage{multirow}

\ifpdf
\hypersetup{
  pdftitle={Spectral Bundle Methods for SDPs},
  pdfauthor={Feng-Yi Liao, Lijun Ding, and Yang Zheng}
}
\fi

\usepackage[utf8]{inputenc}
\usepackage[english]{babel}
\usepackage{xcolor}

\usepackage{caption}
\usepackage{subcaption}
\usepackage{authblk}

\begin{document}
\title{An Overview and Comparison of Spectral Bundle Methods for Primal and Dual Semidefinite Programs\thanks{This work is supported by NSF ECCS-2154650. Corresponding author: Yang Zheng (zhengy@eng.ucsd.edu).}}
\author[1]{Feng-Yi Liao}
\author[2]{Lijun Ding}
\author[1]{Yang Zheng}
\affil[1]{Department of Electrical and Computer Engineering, University of California, San Diego (fliao@ucsd.edu, zhengy@eng.ucsd.edu)}
\affil[2]{Wisconsin Institute for Discovery, University of Wisconsin–Madison, Madison (lding47@wisc.edu)}

\date{}
\maketitle
	
\begin{abstract}

  The spectral bundle method developed by Helmberg and Rendl is well-established for solving large-scale semidefinite programs (SDPs) in the dual form, especially when the SDPs admit \textit{low-rank primal solutions}. Under mild regularity conditions, a recent result by Ding and Grimmer has established fast linear convergence rates when the bundle method captures \textit{the rank of primal solutions}.   
  In this paper, we present an overview and comparison of spectral bundle methods for solving both \textit{primal} and \textit{dual} SDPs. In particular, we introduce a new family of spectral bundle methods for solving SDPs in the \textit{primal} form. The algorithm developments are parallel to those by Helmberg and Rendl, mirroring the elegant duality between primal and dual SDPs. The new family of spectral bundle methods also achieves linear convergence rates for primal feasibility, dual feasibility, and duality gap when the algorithm captures \textit{the rank of the {dual} solutions}. 
  Therefore, the original spectral bundle method by Helmberg and Rendl is well-suited for SDPs with \textit{low-rank primal solutions}, while on the other hand, our new spectral bundle method works well for SDPs with \textit{low-rank dual solutions}. These theoretical findings are supported by a range of large-scale numerical experiments. Finally, we demonstrate that our new spectral bundle method achieves state-of-the-art efficiency and scalability for solving polynomial optimization compared to a set of baseline solvers \textsf{SDPT3}, \textsf{MOSEK}, \textsf{CDCS}, and \textsf{SDPNAL+}. 
  
 \end{abstract}

    \section{Introduction}

Semidefinite programs (SDPs) are an important class of convex optimization problems that minimize a linear function in the space of positive semidefinite (PSD) matrices subject to linear equality constraints \cite{vandenberghe1996semidefinite}. Mathematically, the standard primal and dual SDPs are in the form of 
\begin{equation}
    \begin{aligned}
        \min_{X} \quad & \langle C, X\rangle \\
        \mathrm{subject~to} \quad & \langle A_i, X\rangle = b_i, \quad i = 1, \ldots, m,   \label{eq:SDP-primal}\\
        & X \in \mathbb{S}^n_+, 
\end{aligned}
\tag{P}
\end{equation}
and 
\begin{equation}
    \begin{aligned}
        \max_{y, Z} \quad & b^\tr y  \\
        \mathrm{subject~to} \quad & Z + \sum_{i=1}^m A_i y_i = C, \label{eq:SDP-dual} \\
        & Z \in \mathbb{S}^n_+, 
    \end{aligned}
    \tag{D}
\end{equation}
where $b \in \mathbb{R}^m, C, A_1, \ldots, A_m \in \mathbb{S}^n$ are the problem data, $\mathbb{S}^n_+$ denotes the set of $n \times n$ PSD matrices (we also write $X \succeq 0$ to denote $X \in \mathbb{S}^n_+$ when the dimension is clear from the context or not important), and $\langle \cdot,\cdot \rangle$ denotes the standard trace inner product on the space of symmetric matrices. 

SDPs offer a powerful mathematical framework that has gained significant attention for decades \cite{vandenberghe1996semidefinite,blekherman2012semidefinite} and still receives strong research interests today \cite{zheng2021chordal,vandenberghe2015chordal,majumdar2020recent}. Indeed, SDP provides a versatile and robust modeling and optimization approach for solving a wide range of problems in different fields. 
Undoubtedly, SDPs have become powerful tools in control theory \cite{boyd1994linear}, combinatorial optimization \cite{sotirov2012sdp}, polynomial optimization \cite{blekherman2012semidefinite}, machine learning \cite{lanckriet2004learning}, and beyond \cite{vandenberghe1999applications}. 
In theory, one can solve any SDP instance  up to arbitrary precision in polynomial time using second-order interior point methods (IPMs)\cite{vandenberghe1996semidefinite,wolkowicz2012handbook}.
At each iteration of second-order IPMs, one usually needs to solve a linear system with a coefficient matrix (i.e., the Schur complement matrix) being generally dense and ill-conditioned. Consequently, 
IPMs often suffer from both computational and memory issues when solving SDPs from large-scale practical applications. 

Improving the scalability of SDPs has gained significant attention in recent years; see \cite{majumdar2020recent,zheng2021chordal,ahmadi2017improving} for surveys. 
In particular, first-order methods (FOMs) are at the forefront of developing scalable algorithms for solving large-scale SDPs thanks to their low complexity per iteration. For instance, the alternating direction method of multipliers (ADMM) is used to solve large-scale SDPs in the dual form \cref{eq:SDP-dual} \cite{wen2010alternating}. The ADMM framework has been extended to solve the homogenous self-dual embedding of SDPs \cref{eq:SDP-primal,eq:SDP-dual} in \cite{ocpb:16}.
In \cite{garstka2021cosmo}, ADMM has been applied to solving SDPs with a quadratic cost function. 
It is known that augmented Lagrangian methods (ALMs) are also suitable for solving large-scale optimization problems. Some efficient ALM-based algorithms have recently been developed to solve large-scale SDPs. For example, a Newton-CG augmented Lagrangian method is proposed to solve SDPs with a large number of affine constraints. An enhanced version is developed in  \cite{zhao2010newton} to further tackle degenerate SDPs by employing a semi-smooth Newton-CG scheme coupled with a warm start strategy \cite{yang2015sdpnal}. The algorithms \cite{zhao2010newton, yang2015sdpnal} have been implemented in a MATLAB package, \textsf{SDPNAL+}, which has shown promising numerical performance.   
To tackle the storage issue, 
the sketching idea, approximating a large matrix $X$ without explicitly forming it, is exploited in the ALM framework together with a conditional gradient method for SDPs   
\cite{yurtsever2021scalable}. In \cite{ding2021optimal}, an optimal storage scheme  is developed to solve SDPs  by using a first-order method to solve the dual SDP \cref{eq:SDP-dual} and recovering the primal solution in \cref{eq:SDP-primal}. Finally, a class of efficient first-order spectral bundle methods has been developed to solve an equivalent eigenvalue problem when primal SDPs enjoy a constant trace property \cite{helmberg2000spectral,helmberg2002spectral}. 

Another important idea in designing efficient algorithms is to exploit the underlying sparsity and structures in  SDPs~\cite{zheng2021chordal,vandenberghe2015chordal,zheng2023sum}. 
When SDPs have an aggregate sparsity pattern, chordal decomposition \cite{vandenberghe2015chordal,zheng2021chordal} has been exploited to reduce the dimension of PSD constraints in the design of both IPMs~\cite{fukuda2001exploiting} and ADMM~\cite{zheng2020chordal}. In \cite{sun2014decomposition}, partially separable properties in conic programs (including chordal decomposition) have been investigated to design efficient first-order algorithms. On the other hand, when the SDPs have low-rank solutions, low-rank factorization decomposing a big PSD matrix $X \in \mathbb{S}^{n}_+$ into $VV^\tr$, where $V \in \mathbb{R}^{n \times r} $ with $r \ll n $, has been utilized reduce the searching space in  \cite{burer2003nonlinear}. This low-rank factorization leads to a nonconvex optimization problem, and there are significant efforts in advancing theoretical understanding of the factorization approach \cite{boumal2016non,ge2017no,waldspurger2020rank}. Similar to the low-rank factorization, one can approximate the PSD constraint $X \in \mathbb{S}^n_+$ by $X = V S V^\tr$ where $S \in \mathbb{S}^r_+$ and the factor $V \in \mathbb{R}^{n \times r}$ is fixed. This is one of the main ideas in the design of the spectral bundle method \cite{helmberg2000spectral,ding2020revisit} and the spectral Frank Wolf algorithm \cite{ding2020spectral}. The core strategy in \cite{helmberg2000spectral,ding2020revisit,ding2020spectral} is to iteratively search for the factor $V$ such that it spans the range space of the optimal solution in \Cref{eq:SDP-primal}. Another similar approximation strategy is the basis pursuit techniques in \cite{ahmadi2017sum,liao2022iterative}.

In this paper, we focus on the development of the low-rank approximation $X = V S V^\tr$ in spectral bundle methods \cite{helmberg2000spectral,helmberg2002spectral,helmberg2014spectral,ding2020revisit}. The spectral bundle method, originally developed by Helmberg and Rendl in \cite{helmberg2000spectral}, is well-established to solve large-scale SDPs, thanks to its low per-iteration complexity and fast practical convergence. Further developments of spectral bundle methods appear in \cite{helmberg2002spectral,helmberg2014spectral}. Very recently, Ding and Grimmer established sublinear convergence rates of the spectral bundle method in terms of primal feasibility, dual feasibility, and duality gap, and further proved a \textit{linear convergence} rate when the algorithm captures a rank condition \cite{ding2020revisit}. To the best of our knowledge, all existing spectral bundle methods \cite{ding2020revisit,helmberg2000spectral,apkarian2008trust,helmberg2002spectral,helmberg2014spectral} focus on solving dual SDPs in the form of \cref{eq:SDP-dual}. As shown in \cite{ding2020revisit}, these spectral bundle methods are more desirable when the primal SDP \cref{eq:SDP-primal} admits low-rank solutions in which it is easier to enforce the rank condition to guarantee linear convergence. On the other hand, when the dual SDP \cref{eq:SDP-dual} admits low-rank solutions, 
the existing spectral bundle methods may offer less benefit in terms of convergence and efficiency. Indeed, SDPs arising from moment/sum-of-squares (SOS) optimization problems and their applications \cite{parrilo2003semidefinite,lasserre2009moments} are likely to admit low-rank solutions in the \textit{dual} SDP \eqref{eq:SDP-dual} (this low-rank property is consistent with the flat extension theory on the moment side \cite[Theorem 3.7]{lasserre2009moments} when it is formulated as a dual SDP).  

In this work, we present an overview and comparison of spectral bundle methods for solving both \textit{primal} and \textit{dual} SDPs. In particular, we introduce a new family of spectral bundle methods for solving  SDPs in the \textit{primal} form \eqref{eq:SDP-primal}. Our algorithm developments are parallel to those by Helmberg and Rendl \cite{helmberg2000spectral} which focuses on solving dual SDPs \eqref{eq:SDP-dual}, mirroring the elegant duality between primal and dual SDPs. In particular, our contributions are as follows. 

\begin{itemize}
    \item We propose a new family of spectral bundle methods, called \SBMP, for solving primal SDPs \cref{eq:SDP-primal}, while all existing methods \cite{ding2020revisit,helmberg2000spectral,apkarian2008trust,helmberg2002spectral,helmberg2014spectral} focus on dual SDPs \eqref{eq:SDP-dual}. We first translate the primal SDP \eqref{eq:SDP-primal} into an eigenvalue optimization problem using the exact penalty method (\Cref{proposition-primal-exact-penalty}). Then, each iteration of \SBMP~solves a small subproblem formulated from $\rpast$ past eigenvectors and $\rcurrent$ current eigenvectors of the primal variable $X$ evaluated at the past and current iterates respectively (\Cref{proposition-primal-penalty}).   
    
    \item We show that any configuration of \SBMP~admits $\bigO(1/\epsilon^3)$ convergence rate in primal feasibility, dual feasibility, and duality gap. Similar to \cite{ding2020revisit}, \SBMP~has a faster convergence rate, $\bigO(1/\epsilon)$, when the SDPs \cref{eq:SDP-primal,eq:SDP-dual} satisfy strict complementarity (\Cref{thm: sublinearates-P-K}). In \Cref{thm: linear-convergence}, we further show \textit{linear convergence} of \SBMP~if $(\romannumeral 1) $ strict complementarity holds, $(\romannumeral 2) $ the number of eigenvectors $\rcurrent$ is larger than the rank of dual optimal solutions. 
    Our proofs largely follow the strategies in \cite[Section 3]{ding2020revisit}, \cite[Section 4]{du2017rate} and \cite[Section 7]{ruszczynski2011nonlinear}, and we complete some detailed calculations and handle the constrained case for the primal SDPs. As a byproduct, we revisit the results for generic bundle methods in \cite[Theorem 2.1 and 2.3]{diaz2023optimal} for constrained convex optimization (see \Cref{lemma-iterations-bound}).  

    \item We present a detailed comparison between the primal and dual formulations of the spectral bundle methods by showing the symmetry of the parameters and convergence behaviors from both sides. It becomes clear that the existing dual formulation in \cite{ding2020revisit,helmberg2000spectral,apkarian2008trust,helmberg2002spectral,helmberg2014spectral} is advantageous when \textit{the primal SDP \cref{eq:SDP-primal} admits low-rank solutions}. On the other hand, our primal formulation \SBMP~is more suitable when \textit{the dual SDP \cref{eq:SDP-dual} admits low-rank solutions}. These theoretical findings are supported by a range of large-scale numerical experiments. 
    
    \item Finally, we present an open-source implementation of the spectral bundle algorithms for both \cref{eq:SDP-primal} and \cref{eq:SDP-dual}, while the existing implementations of spectral bundle algorithms for \cref{eq:SDP-dual} are not open-source or not easily accessible.  We demonstrate that our new spectral bundle method \SBMP~achieves state-of-the-art efficiency and scalability for solving polynomial optimization compared to a set of baseline solvers \textsf{SDPT3} \cite{toh1999sdpt3}, \textsf{MOSEK} \cite{mosek}, \textsf{CDCS} \cite{CDCS}, and \textsf{SDPNAL+}~\cite{yang2015sdpnal}.
\end{itemize}

The rest of the paper is structured as follows. \cref{sec:Preliminaries} covers some preliminaries on SDPs and nonsmooth optimization. \Cref{sec:Penalized-nonsmooth-formulations-for-SDPs} presents the exact penalty formulations for primal and dual SDPs. This is followed by our new family of spectral bundle methods \SBMP~and~the~convergence results in \Cref{sec:SBM-Primal}. \Cref{section:SBM-dual} reviews the classical spectral bundle methods, called \SBMD, and the connections and differences between \SBMP~and~\SBMD~are clarified. Our open-source implementation and numerical experiments are presented in \Cref{sec:numerical-results}. \Cref{sec:Conclusion} concludes the paper. Some detailed calculations and technical proofs are postponed in the appendix.  

\vspace{2mm}

\noindent \myparagraph{Notation} We use $\langle \cdot,\cdot \rangle $ to denote the dot product and the trace inner product on the space of $\RR^n$ and $\mathbb{S}^n$ respectively.  For a symmetric matrix $A \in \mathbb{S}^n$, we denote its eigenvalues in the decreasing order as $\lambda_{\max}(A) = \lambda_{1}(A) \geq \cdots \geq \lambda_{n}(A)$. Given a vector on $\mathbb{R}^n$, we use $\|\cdot\|$ to denote its two norm. For a matrix $M \in \mathbb{R}^{m \times n}$, its Frobenius norm, operator two norm, and nuclear norm are denoted by $\|\cdot\|,\|\cdot\|_{\mathrm{op}},$ and $\|\cdot\|_{*}$ respectively. In the primal and dual SDPs \cref{eq:SDP-primal,eq:SDP-dual}, for notational simplicity, we will also denote a linear map $\Amap: \mathbb{S}^{n} \to \mathbb{R}^m$ as 
$
\Amap(X) := \begin{bmatrix} \langle A_1, X \rangle, \ldots, \langle A_m, X \rangle \end{bmatrix}^\tr, 
$ and its adjoint map $\Ajmap$ that is a linear mapping from $\mathbb{R}^m$ to
$\mathbb{S}^n$ as $\Ajmap(y) := \sum_{i=1}^m A_i y_i$. 
The optimal cost value of SDPs \cref{eq:SDP-primal,eq:SDP-dual} is denoted as $\pstar$ and $\dstar$ respectively.
Finally, given a closed set $\mathcal{C} \subset \mathbb{R}^n$ and a point $Y \in \mathbb{R}^n$, the distance of $Y$ to $\mathcal{C}$ is defined as $\Dist(Y,\mathcal{C}) = \inf_{X \in \mathcal{C}}  \|X-Y\|$. 
    \section{Preliminaries}
\label{sec:Preliminaries}
In this section, we first introduce standard assumptions and an important notion of strict complementarity for \cref{eq:SDP-primal,eq:SDP-dual}. We then briefly overview the exact penalization for constrained nonsmooth convex optimization and the generic bundle method.

\subsection{Strict complementarity of SDPs}
Throughout this paper, we make the following standard assumptions for well-behaved SDPs \cref{eq:SDP-primal} and \cref{eq:SDP-dual}.

\begin{assumption} \label{assumption:linearly-independence}
The matrices $A_i, i = 1, \ldots, m$ in \cref{eq:SDP-primal} and \cref{eq:SDP-dual} are linearly independent. 
\end{assumption}

\begin{assumption} \label{assumption-slater-condition}
The SDPs \cref{eq:SDP-primal} and \cref{eq:SDP-dual} satisfy Slater's constraint qualification, i.e., they are both strictly feasible.
\end{assumption}

\Cref{assumption:linearly-independence} allows us to uniquely determine $y$ from a given dual feasible $Z$, i.e., the feasible point $y$ is unique in $Z+\mathcal{A}^* (y) = C$ when giving a feasible $Z$. Under \Cref{assumption-slater-condition}, the strong duality holds for \cref{eq:SDP-primal} and \cref{eq:SDP-dual} (i.e., $p^\star = d^\star$), and both \cref{eq:SDP-primal} and \cref{eq:SDP-dual} are solvable (i.e., there exist at least a primal minimizer $\Xstar$ and a dual maximizer $(y^\star, Z^\star)$ that achieve the optimal cost) \cite[Section 5.2.3]{boyd2004convex}. 
We denote the set of primal optimal solutions to \cref{eq:SDP-primal} as $\Pstar$ and the set of dual optimal solutions to \cref{eq:SDP-dual} as $\Dstar$, i.e.,
\begin{subequations} \label{eq:optimal-solution-sets}
    \begin{align}
    \Pstar &= \left\{X \in \mathbb{S}^{n} \mid p^\star = \langle C, X\rangle, \mathcal{A}(X) = b, X \in \mathbb{S}^n_+\right\}, \label{eq:optimal-solution-sets-P} \\
    \Dstar &= \left\{(y,Z) \in \mathbb{R}^m \times \mathbb{S}^{n} \mid d^\star = b^\tr y, Z+\mathcal{A}^* (y) = C, Z \in \mathbb{S}^n_+\right\}. \label{eq:optimal-solution-sets-D} 
\end{align}
\end{subequations}
\Cref{assumption-slater-condition} ensures that $\Pstar\neq \emptyset$ and $\Dstar \neq \emptyset$. In addition, if \Cref{assumption:linearly-independence} holds, then the solution sets $\Pstar$ and $\Dstar$ are nonempty and compact \cite[Section 2]{ding2020revisit}. 

\begin{proposition}[{\cite[Section 2]{ding2020revisit}}] \label{assumption-compactness}
    Under \Cref{assumption:linearly-independence,assumption-slater-condition}, the mapping $\Amap$ is surjective and the optimal solution sets $\Pstar$ and $\Dstar$ are nonempty and compact.
\end{proposition}
It is clear that $\Pstar$ and $\Dstar$ are closed. Given any $(\ystar,\Zstar) \in \Dstar$ and a strict primal feasible point $\hat{X}$, we have a finite duality~gap $\langle C ,\hat{X}\rangle - b^\tr \ystar = \langle \hat{X} , \Zstar \rangle \geq 0$. If $\|\Zstar\| \to \infty$, then $\langle \hat{X} , \Zstar \rangle \to \infty$ since $\Zstar \in \mathbb{S}^n_+$ and $\hat{X}$ is positive definite. 
This is impossible due to the finite duality gap, and thus, any optimal $\Zstar$ is bounded.  \Cref{assumption:linearly-independence} ensures that~$\Amap$ is a surjective mapping, which means $\Ajmap$ is injective. Thus, any optimal $\ystar$ is bounded, and $\Dstar$ is bounded. Similarly, the existence of a strictly feasible point $(\hat{y}, \hat{Z})$ ensures the compactness of $\Pstar$. 

The following result is a version of the KKT optimality condition for the SDPs \cref{eq:SDP-primal} and \cref{eq:SDP-dual}.

\begin{lemma}[{\cite[Lemma 3]{alizadeh1997complementarity}}] \label{theorem: optimality condition}
    Given a pair of primal and dual feasible solutions $\Xstar$ and $(\ystar, \Zstar)$, they are optimal if and only if  
    there exists an orthonormal matrix $Q \in \mathbb{R}^{n \times n}$ with $Q^{\tr} Q = I$, such that 
    \begin{equation} \label{eq:complementarity}
        \begin{aligned}
            X^\star = Q \cdot \mathrm{diag}(\lambda_1, \ldots, \lambda_n) \cdot Q^\tr, \quad          Z^\star = Q \cdot \mathrm{diag} (w_1, \ldots, w_n) \cdot Q^\tr
        \end{aligned}
    \end{equation}
    and $\lambda_i w_i = 0, i = 1,\ldots,n$.   
\end{lemma}

\noindent Given a pair of optimal solutions $\Xstar$ and $(\ystar, \Zstar)$, the complementary slackness condition 
\eqref{eq:complementarity} 
is equivalent to $\Zstar \Xstar = 0$ ($\Xstar$ and $\Zstar$ commutes so they share a common set of eigenvectors as the columns of $Q$). This implies that 
\begin{align*} 
    \text{rank}(\Xstar) + \text{rank}(\Zstar) \leq n, \quad \text{range}(\Zstar) \subset \text{null}(\Xstar), \quad \text{and} \quad  \text{range}(\Xstar) \subset \text{null}(\Zstar).
\end{align*}
We now introduce the notion of strict complementarity for a pair of 
optimal solutions. 
\begin{definition} [{\cite[Definition 4]{alizadeh1997complementarity}}] 
    A pair of primal and dual optimal solutions $\Xstar \in \Pstar$ and $(\ystar,\Zstar) \in \Dstar$ satisfies strict complementarity if $\mathrm{rank}(\Xstar) + \mathrm{rank}(\Zstar) = n$ holds, i.e., exactly one of the two conditions $\lambda_i = 0$ and $w_i = 0$ is true in \eqref{eq:complementarity}.  
    \label{eq:strictly-complementarity}
\end{definition}
If such a pair $\Xstar$ and $(\ystar,\Zstar)$ exists, we also say that the SDPs \cref{eq:SDP-primal} and \cref{eq:SDP-dual} satisfy strict complementarity. We note that strict complementarity is not restrictive. It is a generic property of SDPs \cite[Theorem 15]{alizadeh1997complementarity}, and many structured SDPs from practical applications also satisfy strict complementarity 
 \cite{ding2021simplicity}.

\subsection{Exact penalization for constrained convex optimization}

Consider a constrained convex optimization problem of the form
\begin{equation} \label{eq:constrained_non-smooth-problem}
    \begin{aligned}
    \min &\quad f(x) \\
    \text{subject to}& \quad g_i(x) \leq 0, \quad i = 1, \ldots, m,\\
    &\quad x \in \mathcal{X}_0,
    \end{aligned}
\end{equation}
where $f: \mathbb{R}^n \to \mathbb{R}$ and $g_i: \mathbb{R}^n \to \mathbb{R}, i = 1, \ldots, m$ are (possibly nondifferentiable) convex functions, and $\mathcal{X}_0 \subseteq \mathbb{R}^n$ is a closed convex set (which are defined by some simple constraints). The idea of exact penalty methods is to reformulate the constrained optimization problem \cref{eq:constrained_non-smooth-problem} by a problem with simple constraints. In particular, upon defining an exact penalty function 
$$
 P(x) = \sum_{i=1}^m \; \max\{0, g_i(x)\},
$$
we consider a penalized problem 
\begin{equation} \label{eq:constrained-penalized}
    \begin{aligned}
    \min &\quad \Phi_\rho (x) := f(x) + \rho P(x) \\
    \text{subject to}& \quad x \in \mathcal{X}_0,
    \end{aligned}
\end{equation}
where $\rho > 0$ is a penalty parameter. When choosing $\rho$ large enough, problems \eqref{eq:constrained_non-smooth-problem} and \eqref{eq:constrained-penalized} are equivalent to each other in the sense that they have the same optimal value and solution set. 

\begin{theorem}[{\cite[Theorem 7.21]{ruszczynski2011nonlinear}}] \label{theorem:exact-penalization}
Suppose that problem \eqref{eq:constrained_non-smooth-problem} satisfies Slater's constraint qualification. There exists a constant $\rho_0 \geq 0$ such that for each $\rho > \rho_0$, a point $\hat{x}$ is an optimal solution of \eqref{eq:constrained_non-smooth-problem} if and only if it is an optimal solution of \eqref{eq:constrained-penalized}. In particular, we can choose $\rho_0 = \sup_{\lambda \in \Lambda} \|\lambda\|_\infty$, where $\Lambda \subset \mathbb{R}^m$ is the set of optimal Lagrange multipliers associated with $g_i(x) \leq 0, i = 1, \ldots, m$.
\end{theorem}

Therefore, we can transform some nonsmooth constraints that are hard to handle in~\eqref{eq:constrained_non-smooth-problem} into the nonsmooth cost function of~\eqref{eq:constrained-penalized}. Then, we can apply cutting plane or bundle methods to solve the nonsmooth optimization~\eqref{eq:constrained-penalized}. However, it should be noted that the resulting problem~\eqref{eq:constrained-penalized} may be difficult to solve if the penalty parameter $\rho$ is too large. As we will discuss in \Cref{sec:Penalized-nonsmooth-formulations-for-SDPs}, in some SDPs that arise from practical applications, the penalty parameter $\rho$ is known \textit{a priori} \cite{ding2020revisit,helmberg2000spectral,mai2023hierarchy}.    

\subsection{The cutting plane and bundle methods} \label{subsection:bundle-methods}
In this subsection, we briefly overview the generic bundle method; see \cite[Chapter 7]{ruszczynski2011nonlinear} for details. Consider a generic nonsmooth constrained convex optimization
\begin{equation} \label{eq:non-smooth-problem}
   f^\star =  \min_{x \in \mathcal{X}_0} \;\; f(x),
\end{equation}
where $f: \mathbb{R}^n \rightarrow \mathbb{R} $ is convex but not necessarily differentiable and $\mathcal{X}_0 \subseteq  \mathbb{R}^n$ is a closed convex set. It is clear that problem \cref{eq:non-smooth-problem} includes \cref{eq:constrained-penalized} as a special case. 

The simplest method for solving \cref{eq:non-smooth-problem} is arguably the \textit{subgradient method}, which constructs a sequence of points $x_t$ iteratively by updating
\begin{equation*} 
    \begin{aligned}
        x_{t+1} = \Pi_{\mathcal{X}_0} ( x_t - \tau_t g_t), \quad t = 1, 2, \ldots
    \end{aligned}
\end{equation*}
where $g_t \in \partial f(x_t)$ is a subgradient of $f(\cdot)$ at the current point $x_t$, and $ \tau_t \in \mathbb{R}$ is a step size, and $\Pi_{\mathcal{X}_0} (x)$ denotes the orthogonal projection of the point $x\in\mathbb{R}^n$ onto $\mathcal{X}_0$. Recall that for a convex function $f: \mathbb{R}^n \rightarrow \mathbb{R}$, a vector $g \in \mathbb{R}^n$ is called a subgradient of $f$ at $x$ if 
\begin{equation*}
f(y) \geq f(x) + \langle g, y - x \rangle, \;\;  \forall y \in \mathbb{R}^n.
\end{equation*}
The set of all subgradients of $f$ at $x$ is called the subdifferential, denoted by $\partial f(x)$. 

With mild assumptions and an appropriate choice of decreasing step sizes, the subgradient method is guaranteed to generate a converging sequence $\{x_t\}$ to an optimal solution of problem \eqref{eq:non-smooth-problem} (see \cite[Theorem 7.4]{ruszczynski2011nonlinear}). Despite the simplicity of subgradient methods, it is generally challenging to develop reliable and efficient step size rules for practical optimization instances. 

\subsubsection{The cutting plane method.} 
Another useful way to utilize subgradients is the idea of the \textit{cutting plane method}, which solves a lower approximation of the function $f(x)$ at every iteration. Here, we assume that $\mathcal{X}_0$ is compact in \cref{eq:non-smooth-problem}; otherwise, we consider minimizing $f(x)$ over $\mathcal{X}_0 \cap \mathcal{C}$ where $\mathcal{C}$ is a compact set containing an optimal solution.

The basic idea of the cutting plane method is to use the subgradient inequality to construct lower approximations of $f$. In particular, at iteration $t$, having points $x_1, x_2, \ldots, x_t$, function values $f(x_1), f(x_2), \ldots,$ $ f(x_t)$, and the corresponding subgradients $g_1, g_2, \ldots, g_t$, we construct a lower approximation using a piece-wise affine function
\begin{equation} \label{eq:piecewise-linear-LB} 
    \begin{aligned}
        \hat{f}_t (x) = \max_{i=1,\ldots,t} \;\; f(x_i) + \langle g_i,x-x_i \rangle.
    \end{aligned}
\end{equation} 
By definition of subgradients, it is clear that $ f(x) \geq \hat{f}_t (x), \forall x \in \mathbb{R}^n$. Starting from a triple $\{x_1 \in \mathcal{X}_0, f(x_1),g_1 \in \partial f(x_1)\}$, the cutting plane method solves the following master problem to generate the next point, 
\begin{equation*}
    \begin{aligned}
        x_{t+1} & \in  \argmin_{x \in \mathcal{X}_0} \;\; \hat{f}_t (x), \quad t = 1, 2, \ldots. 
    \end{aligned}
\end{equation*} 
When $\mathcal{X}_0$ is a convex set defined by simple constraints (e.g., a polyhedron), the above problem becomes a linear program (LP) for which very efficient algorithms exist. 
The sequence generated by the cutting plane method is guaranteed to satisfy that $\lim_{t \to \infty} f(x_t) = f^\star$ \cite[Theorem 7.7]{ruszczynski2011nonlinear}. However, the theoretical convergence rate is rather slow; it generally takes $\bigO(1/\epsilon^n)$ iterations to reach $f(x_t) - f^\star \leq \epsilon$ \cite[Section 5.2] {mutapcic2009cutting}. 
The practical convergence of the cutting plane method may be faster. 

\subsubsection{The bundle method.} \label{subsec:bundle method}
The \textit{bundle method} improves the convergence rate and numerical behavior of the cutting plane method by incorporating a regularization strategy. Unlike the cutting plane method that only considers the lower approximation function, the bundle method updates its iterates by solving a regularized master problem (i.e., a proximal step to the lower approximation model $\hat{f}_t$): 
\begin{equation} \label{eq:bundle-method-t}
    \begin{aligned}
        \xcandidate & \in \argmin_{x \in \mathcal{X}_0 } \;\;\hat{f}_t (x) + \frac{\alpha}{2 } \|x - \omega_t\|_2^2,
    \end{aligned}
\end{equation} 
where $\omega_t  \in  \mathbb{R}^n$ is the current reference point and $\alpha \! > \!  0$ penalizes the deviation from $\omega_t$. The bundle method only updates the iterate $\omega_t$ when the decrease of the objective value of $f$ is at least a fraction of the decrease that the approximated model $\hat{f}_t$ predicts. In particular, letting $0 \!< \!\beta \! < \!1$, if 
\begin{equation} \label{eq:sufficient-descent}
    \begin{aligned}
       \beta \left(f(\omega_t) - \hat{f}_t(\xcandidate)\right) \leq f(\omega_t) -  f(\xcandidate)
    \end{aligned}
\end{equation}
then we set $\omega_{t+1} = \xcandidate$ (\textit{descent step}); otherwise,  we set $\omega_{t+1} = \omega_{t}$ (\textit{null step}). In any case, the subgradient $g_{t+1} \in \partial f(\xcandidate)$ at the new point $\xcandidate$ is used to update the lower approximation $\hat{f}_{t+1}$, e.g., using \cref{eq:piecewise-linear-LB}. 

The seemingly subtle modifications above have a rather surprising consequence: the cost value $f(\omega_t)$ generated by the bundle method converges to $f^\star$ for any constant $\alpha >0$ with a rate of $\bigO(1/\epsilon^3)$ when the objective function is Lipschitz continuous \cite{kiwiel2000efficiency}. Faster convergence rates appear under different assumptions of $f$; see \cite[Table 1] {diaz2023optimal} for a detailed comparison. Note that subgradient methods rely on very carefully controlled decreasing stepsizes which might be inefficient and unreliable in practice, and the cutting plane method has a slow convergence rate theoretically. On the contrary, the bundle method appears more suitable to solve the nonsmooth problem \cref{eq:non-smooth-problem}.

\subsubsection{The bundle method with cut-aggregation.} \label{subsection:lower-approximation-bound}
The lower approximation model $\hat{f}_t (x)$ can be constructed using all past subgradients as in \cref{eq:piecewise-linear-LB}, but this leads to a growing number of cuts or constraints when solving the regularized master~problem \cref{eq:bundle-method-t}. Another useful cut-aggregation idea~\cite[Chapter 7.4.4]{ruszczynski2011nonlinear} allows us to simplify the collection of $t$ lower bounds used by \cref{eq:piecewise-linear-LB} into just two linear lower bounds. In particular, the convergence of the bundle method is guaranteed as long as the lower approximation model $\hat{f}_{t+1}$ satisfies the following three properties~\cite{diaz2023optimal}\footnote{The analysis in \cite{diaz2023optimal} focuses on unconstrained optimization where $\mathcal{X}_0 = \mathbb{R}^n$ in \eqref{eq:non-smooth-problem}. We have extended the analysis \cite{diaz2023optimal} for constrained optimization where $\mathcal{X}_0$ is a closed convex set; see \Cref{Appendix:assumptions-CBM}.}:
\begin{itemize}
    \begin{subequations}
    \item \textbf{Minorant}:  the function $\hat{f}_{t+1}$ is a lower bound on $f$, i.e.
    \begin{equation} \label{eq:bundle-method-property-1}
        \hat{f}_{t+1}(x) \leq f(x), \quad \forall x \in { \mathcal{X}_0 }.
    \end{equation}
    \item \textbf{Subgradient lowerbound}: $\hat{f}_{t+1}$ is lower bounded by the linearization given by some subgradient $g_{t+1} \in \partial f(\xcandidate) $ computed after \cref{eq:bundle-method-t}, i.e.
    \begin{equation} \label{eq:bundle-method-property-2}
         \hat{f}_{t+1} (x)  \geq f(\xcandidate) + \langle g_{t+1}, x-\xcandidate \rangle, \quad \forall x \in {\mathcal{X}_0 }.
    \end{equation}
    \item \textbf{Model subgradient lowerbound}: If step t is a null step, $\hat{f}_{t+1}$ is lower bounded by the linearization of the model $\hat{f}_{t}$ given by the subgradient $s_{t+1}:= \alpha(\omega_t- \xcandidate) \in \partial \hat{f}_t(\xcandidate) + {\mathcal{N}_{\mathcal{X}_0}(\xcandidate)}$, 
    i.e.
    \begin{equation} \label{eq:bundle-method-property-3}
         \hat{f}_{t+1} (x)  \geq \hat{f}_t(\xcandidate) + \langle s_{t+1}, x-\xcandidate \rangle, \quad \forall x \in { \mathcal{X}_0 },
    \end{equation}
    where $\mathcal{N}_{\mathcal{X}_0}(\xcandidate)$ denotes the normal cone of $\mathcal{X}_0$ at point $\xcandidate$, i.e. $\mathcal{N}_{\mathcal{X}_0}(\xcandidate) = \{ v \in \mathbb{R}^n \mid \langle v, x - \xcandidate \rangle \leq 0, \forall x \in \mathcal{X}_0\}$. Note that $s_{t+1}$ certifies the optimality of $\xcandidate$ for problem \cref{eq:bundle-method-t}. 
    \end{subequations}
\end{itemize}
The lower bound \cref{eq:bundle-method-property-3} serves as an aggregation of all previous subgradient lower bounds. Instead of~\eqref{eq:piecewise-linear-LB}, we can construct the lower approximation $\hat{f}_{t+1}$ as the maximum of two lower bounds as
\begin{equation}
    \begin{aligned}
         \hat{f}_{t+1} (x) = \max \left \{f(\xcandidate) + \langle g_{t+1}, x-\xcandidate \rangle, \hat{f}_t(\xcandidate) + \langle s_{t+1}, x-\xcandidate \rangle \right \}. 
    \end{aligned}
\end{equation}

\begin{algorithm}[t] 
    \caption{General Bundle Method} \label{alg:bundle-method} \label{algorithm:bundle}
    \begin{algorithmic}[1]
        \Require{Problem $f(x)$, $\mathcal{X}_0$; Parameters $\alpha > 0$, $\beta \in (0,1)$, 
        $t_{\max} \geq 1$. }
        \For{$t=1,2,\ldots,t_{\max}$}
            \State Solve \cref{eq:bundle-method-t} to obtain a candidate iterate $x_{t+1}^\star $.  \Comment{\textit{master problem}}
            \If {\cref{eq:sufficient-descent} is satisfied}
                \State Set $\omega_{t+1} = x_{t+1}^\star$. \Comment{\textit{descent step}}
            \Else
                \State Set $\omega_{t+1} = \omega_{t}$. \Comment{\textit{null step}}
            \EndIf
            \State Update approximation model $\hat{f}_{t+1} $ without violating \cref{eq:bundle-method-property-1,eq:bundle-method-property-2,eq:bundle-method-property-3}.
            \If {stopping criterion}
                \State Quit. 
                \Comment{\textit{termination}}
                \EndIf
        \EndFor
    \end{algorithmic}
\end{algorithm}

The overall process of the general bundle method is listed in \Cref{alg:bundle-method}. The convergence rates for \Cref{alg:bundle-method} have been recently revisited in {\cite[Theorem 2.1 and 2.3]{diaz2023optimal}}. The big-O notation below suppresses some universal constants. 

\begin{lemma}[{\cite[Theorem 2.1 and 2.3]{diaz2023optimal}}\footnote{The results in \cite[Theorem 2.1 and 2.3]{diaz2023optimal} focus on unconstrained problems \cref{eq:non-smooth-problem} with $\mathcal{X}_0 = \mathbb{R}^n$. Upon some adaptations, we can extend their results for constrained convex optimization with a closed convex set $\mathcal{X}_0 \subseteq \mathbb{R}^n$. Details are presented in \Cref{Appendix:assumptions-CBM}.}] 
\label{lemma-iterations-bound}
    Consider a convex and M-Lipschitz function $f$ in \cref{eq:non-smooth-problem}. Let $f^\star = \inf_{{x} \in \mathcal{X}_0} f(x)$
    and $\mathcal{C} =\{x \in \mathcal{X}_0  \mid f(x) = f^\star\} $. If $\mathcal{C}$ is nonempty, the number of steps for \Cref{alg:bundle-method}  before reaching an $\epsilon > 0$ optimality, i.e. $f(x)- f^\star\leq \epsilon$, is bounded by
    \begin{align*}
        t \leq \bigO \biggl( \frac{ 12 \alpha M^2 D^4}{\beta (1-\beta)^2 \epsilon^3} \biggr),
    \end{align*}   
    where $D = \sup_{k} \dist(x_k,\mathcal{C})  < \infty$. 
    If $f$ further satisfies the quadratic growth condition 
    \begin{equation} \label{eq:quadratic-growth-lemma}
        \begin{aligned}
            f(x) - f^\star \geq \mu \cdot \dist^2(x,\mathcal{C}), \quad \forall x \in \mathcal{X}_0,
        \end{aligned}
    \end{equation}
    where $\mu>0$ is a positive constant, then the number of steps for \Cref{alg:bundle-method} before reaching an $\epsilon$ optimality is bounded by 
    \begin{align*}
        t \leq \bigO \biggl( \frac{16M^2}{\beta (1-\beta)^2 \min \{\alpha,\mu\} \epsilon}\biggr).
    \end{align*}
\end{lemma}

When solving SDPs \cref{eq:SDP-primal}-\cref{eq:SDP-dual}, a specialized version called \textit{spectral bundle method} constructs a special lower approximation model satisfying \cref{eq:bundle-method-property-1}-\cref{eq:bundle-method-property-3}.
This idea was first proposed in \cite{helmberg2000spectral} for the dual SDP \eqref{eq:SDP-dual} with a constant trace property. In this paper, we will show spectral bundle methods can be developed for both general primal and dual SDPs \eqref{eq:SDP-primal} and \eqref{eq:SDP-dual}.  
We will present the details in \Cref{sec:SBM-Primal} and \Cref{section:SBM-dual}. 

    \section{Penalized nonsmooth formulations for SDPs}\label{sec:Penalized-nonsmooth-formulations-for-SDPs}
We here present an exact nonsmooth penalization of primal and dual SDPs \cref{eq:SDP-primal}-\cref{eq:SDP-dual} in the form of \cref{eq:non-smooth-problem}, which allows us to apply the bundle method in \Cref{sec:SBM-Primal} and \Cref{section:SBM-dual}. 

\subsection{Exact penalization of primal and dual SDPs} \label{subsec:exact-penalty-SDP}
The semidefinite constraints in \cref{eq:SDP-primal} and \cref{eq:SDP-dual} are nonsmooth and typically non-trivial to deal with for numerical algorithms. A useful method proposed in \cite{helmberg2000spectral} is to move nonsmooth semidefinite constraints into the cost function. In particular, for the primal SDP \cref{eq:SDP-primal}, we consider a penalized nonsmooth formulation
\begin{equation} 
    \begin{aligned}
        \min_{X} \quad & \langle C, X\rangle + \rho \max \{\lambda_{\max} (-X) ,0 \}\\
        \mathrm{subject~to} \quad & \langle A_i, X\rangle = b_i, \quad i = 1, \ldots, m,   
        \label{eq:SDP-primal-penalized}
    \end{aligned}
    \end{equation}
    and for the dual SDP \cref{eq:SDP-dual}, we consider the following penalized nonsmooth formulation
    \begin{equation}
     \begin{aligned}
        \min_{y} \;\; -b^\tr y + \rho  \max \left\{\lambda_{\max} \left(\sum_{i=1}^m A_i y_i  - C\right) ,0 \right\}. \label{eq:SDP-dual-penalized}
    \end{aligned}
\end{equation}

From \Cref{theorem:exact-penalization}, we expect that if the penalty parameter $\rho$ is large enough, \cref{eq:SDP-primal-penalized} and \cref{eq:SDP-dual-penalized} are equivalent to the primal and dual SDPs \cref{eq:SDP-primal} and \cref{eq:SDP-dual}, respectively. We have the following results (recall that $\mathcal{P}^\star$ and $\mathcal{D}^\star$ are the sets of primal and dual optimal solutions, respectively; see \cref{eq:optimal-solution-sets}).

\begin{proposition} \label{proposition-primal-exact-penalty}
Let
$\rho > \DZstar := \max_{(\ystar, \Zstar) \in \Dstar} \; {\Trace(\Zstar)}.$
A point $\hat{X}$ is an optimal solution of the primal SDP \cref{eq:SDP-primal} if and only if it is an optimal solution of \eqref{eq:SDP-primal-penalized}.
\end{proposition}

\begin{proposition}
\label{proposition-dual-exact-penalty}
Let
$
\rho > \DXstar := \max_{\Xstar \in \Pstar} \; {\Trace(\Xstar)}. 
$
A point $\hat{y}$ (with $\hat{Z} = C - \sum_{i=1}^m A_i \hat{y}_i$) is an optimal solution of the dual SDP \cref{eq:SDP-dual} if and only if it is an optimal solution of \eqref{eq:SDP-dual-penalized}.
\end{proposition}

Both \Cref{proposition-primal-exact-penalty,proposition-dual-exact-penalty}
 are direct consequences of \Cref{theorem:exact-penalization}. In particular, a proof of \Cref{proposition-dual-exact-penalty} appeared in \cite{ding2020revisit}. 
For completeness, we provide a proof of \Cref{proposition-primal-exact-penalty}~in~\Cref{subsection:proof-penalty}. In some applications, we may have prior information on the bounds of $\DZstar$ and $\DXstar$ (for example, one may have explicit trace constraints; see \Cref{subsection:SDP-with-trace-constraints}). In these cases, we can choose the penalty parameter $\rho$ \textit{a priori}.  

\begin{remark}[Exact penalization for primal SDPs] 
The exact penalization for dual SDPs \cref{eq:SDP-dual-penalized} is in the form of unconstrained eigenvalue minimization; see~\cite{overton1992large} for excellent discussions on eigenvalue optimization. To our best knowledge, all the existing results on the application of bundle methods for solving SDPs focus on the dual formulation \cref{eq:SDP-dual-penalized}. One of the early results in \cite{helmberg2000spectral} assumes a constant trace constraint $\Trace(X) = k>0$. This has been generalized to standard SDPs (i.e. \cref{proposition-dual-exact-penalty}) in \cite{ding2020revisit,ding2021optimal}. However, the exact penalization for primal SDPs \cref{eq:SDP-primal-penalized} has been less studied. We cannot find a formal statement of \Cref{proposition-primal-exact-penalty} in the literature. For completeness, we provide a proof of \Cref{proposition-primal-exact-penalty} in \Cref{subsection:proof-penalty}. 
\end{remark}

\subsection{SDPs with trace constraints} \label{subsection:SDP-with-trace-constraints}
Here, we show that the exact penalty formulations can be viewed as standard SDPs with an explicit trace constraint $\Trace(X) \leq \rho$ or $\Trace(Z) \leq \rho$. In particular, let us consider 
\begin{equation}
 \begin{aligned}
        \max_{y, Z} \quad & b^\tr y  \\
        \mathrm{subject~to} \quad & Z + \sum_{i=1}^m A_i y_i = C, \label{eq:SDP-dual-trace} \\
        & \Trace(Z) \leq \rho, Z \in \mathbb{S}^n_+, 
    \end{aligned}
\end{equation}
and 
\begin{equation}
    \begin{aligned}
        -\min_{X} \quad & \langle C, X\rangle \\
        \mathrm{subject~to} \quad & \langle A_i, X\rangle = b_i, \quad i = 1, \ldots, m,   \label{eq:SDP-primal-trace}\\
        & \Trace(X) \leq \rho, X \in \mathbb{S}^n_+. 
\end{aligned}
\end{equation}

\begin{proposition}  \label{proposition:trace-constraint}
The following statements hold:
\begin{itemize}
    \item For any $\rho >0$ such that \cref{eq:SDP-dual-trace} is strictly feasible, the exact penalization~\cref{eq:SDP-primal-penalized} and the modified SDP \cref{eq:SDP-dual-trace} have the same optimal cost value. 
    \item For any $\rho >0$ such that \cref{eq:SDP-primal-trace} is strictly feasible, the exact penalization~\cref{eq:SDP-dual-penalized} and the modified SDP \cref{eq:SDP-primal-trace} have the same optimal cost value. 
\end{itemize}
\end{proposition}

\begin{proof}
    The equivalence comes from the strong duality. We present simple arguments below. It is straightforward to verify that the Lagrange dual problem for \cref{eq:SDP-dual-trace} is 
    \begin{equation} \label{eq:SDP-dual-trace-s1}
     \begin{aligned}
            \min_{X,t,Q} \quad & \langle C,Q \rangle + \rho t  \\
            \mathrm{subject~to} \quad & \langle A_i,Q \rangle =  b_i,\quad i=1,\ldots,m,\\
            & Q + t I  = X,\; t \geq 0,\; X \in \mathbb{S}^n_+. 
        \end{aligned}
    \end{equation}
        Eliminating the variable $X$ leads to $Q + t I\in \mathbb{S}^n_+$, which is equivalent to $t \geq \lambda_{\max}(-Q)$. Since $ t \geq 0$, upon partially minimizing over $t$, the problem \cref{eq:SDP-dual-trace-s1} is equivalent to 
        \begin{equation*}
     \begin{aligned}
            \min_{Q} \quad & \langle C,Q \rangle + \rho \max\{\lambda_{\max}(-Q),0\}  \\
            \mathrm{subject~to} \quad & \langle A_i,Q \rangle =  b_i,\quad i=1,\ldots,m, 
        \end{aligned}
    \end{equation*}
    which is clearly equivalent to~\cref{eq:SDP-primal-penalized}. Since \cref{eq:SDP-dual-trace} is strictly feasible, strong duality holds for \cref{eq:SDP-dual-trace} and \cref{eq:SDP-dual-trace-s1}, which confirms that~\cref{eq:SDP-primal-penalized} and \cref{eq:SDP-dual-trace} have the same optimal cost value.  

    Similarly, the Lagrange dual problem for \cref{eq:SDP-primal-trace} is 
        \begin{equation}\label{eq:SDP-primal-trace-s1}
        \begin{aligned}
            -\max_{y,t} \quad & b^\tr y - t \rho \\
        \mathrm{subject~to} \quad & Z + \sum_{i=1}^m A_i y_i - t I = C ,\\
        & t \geq 0, Z \in \mathbb{S}^n_+.
        \end{aligned}
        \end{equation}
    Eliminating the variable $Z$ leads to~$C-\sum_{i=1}^m A_i y_i+ t I \in \mathbb{S}^n_+$, which is equivalent to $t \geq \lambda_{\max}( \sum_{i=1}^m A_i y_i -C )$. Combining this bound with the constraint $t \geq 0$ leads to $t \geq \max \{\lambda_{\max}(\sum_{i=1}^m A_i y_i -C ),0 \}$. Thus, the problem \cref{eq:SDP-primal-trace-s1} is equivalent to 
    \begin{align*}
        -\max_{y}  \;\; & b^\tr y - \rho  \max \Biggl\{{\lambda_{\max} \left(\sum_{i=1}^m A_i y_i- C\right ),0}\Biggr\},
    \end{align*}
    which is equivalent to \cref{eq:SDP-dual-penalized}. Since the strongly duality holds for \cref{eq:SDP-primal-trace} and \cref{eq:SDP-primal-trace-s1}, we know that \cref{eq:SDP-dual-penalized} and \cref{eq:SDP-primal-trace} have the same optimal cost value.
    \end{proof}
    
  We note that \Cref{proposition:trace-constraint} implies that when $\rho$ is large enough, i.e., satisfying the bounds in \Cref{proposition-primal-exact-penalty,proposition-dual-exact-penalty}, \Cref{eq:SDP-primal-trace} is equivalent to the primal SDP \Cref{eq:SDP-primal}, and \Cref{eq:SDP-dual-trace} is equivalent to the dual SDP  \Cref{eq:SDP-dual}. This result becomes obvious since the extra constraint trace constraint $\Trace(X) \leq \rho$ or $\Trace(Z) \leq \rho$ does not affect the optimal solutions.   

If the constraints $\langle A_i, X \rangle = b_i, i = 1, \ldots, m$ imply that $\Trace(X) = k$ for some $k>0$, i.e., the primal SDP \Cref{eq:SDP-primal} has an implicit constant trace constraint, the exact dual penalization \cref{eq:SDP-dual-penalized}
can be simplified as
\begin{equation}
     \begin{aligned}
        \min_{y} \;\; -b^\tr y + k \lambda_{\max} \left(\sum_{i=1}^m A_i y_i  - C\right). \label{eq:SDP-dual-penalized-trace}
    \end{aligned}
\end{equation}
This was first used to derive the original spectral bundle method in \cite{helmberg2000spectral}. Similarly, if $ Z + \sum_{i=1}^m A_i y_i = C$ imply that $\Trace(Z) = k$, i.e., $\Trace(A_i) = 0$, the exact primal penalization \cref{eq:SDP-primal-penalized} can be simplified as 
\begin{equation} 
    \begin{aligned}
        \min_{X} \quad & \langle C, X\rangle + k \lambda_{\max} (-X)\\
        \mathrm{subject~to} \quad & \langle A_i, X\rangle = b_i, \quad i = 1, \ldots, m. 
        \label{eq:SDP-primal-penalized-trace}
    \end{aligned}
    \end{equation}
For self-completeness, we provide short derivations for \Cref{eq:SDP-dual-penalized-trace} and \Cref{eq:SDP-primal-penalized-trace} in \Cref{sec:derivation-constant-trace-SDP}.

\begin{remark}[Applications with constant trace]
    We note that primal SDPs with a constant trace constraint are very common for semidefinite relaxations of binary combinatorial optimization problems, such as MaxCut \cite{goemans1995improved} and Lov$\acute{a}$sz theta number \cite{lovasz1979shannon}. Also, dual SDPs with a constant trace constraint appears in certain matrix completion problem \cite{candes2012exact} (see \Cref{subsec:constant-trace}) and the moment/sum-of-squares relaxation of polynomial optimizations \cite{mai2023hierarchy} (see \Cref{subsec-quartic-poly,subsec:Exaxt-penalty-SOS}). For these problems, the penalty parameter $\rho$ is thus known \textit{a priori}. 
\end{remark}
    \section{Spectral bundle methods for primal SDPs} \label{sec:SBM-Primal}

We can apply the standard bundle method in \Cref{subsection:bundle-methods} to solve the penalized primal formulation \cref{eq:SDP-primal-penalized}  or dual formulation \cref{eq:SDP-dual-penalized}. This idea was first proposed in \cite{helmberg2000spectral}, and further revised and developed in \cite{helmberg2002spectral,helmberg2014spectral,ding2020revisit,apkarian2008trust}. To our best knowledge, however, all previous studies \cite{ding2020revisit,helmberg2000spectral,apkarian2008trust,helmberg2002spectral,helmberg2014spectral} only consider the penalized dual formulation \cref{eq:SDP-dual-penalized}. The dual formulation is in the form of unconstrained eigenvalue optimization~\cite{overton1992large}, for which it seems more convenient to apply the bundle method. 

In this section, we apply the bundle method to solve the penalized primal formulation \cref{eq:SDP-primal-penalized}, which leads to a new family of spectral bundle algorithms. Differences and connections between our new algorithms and the existing spectral bundle algorithms will be clarified in \Cref{section:SBM-dual}.

\subsection{A new family of spectral bundle algorithms for primal SDPs}

For notational convenience, we denote the cost function in \cref{eq:SDP-primal-penalized} as 
\begin{equation} \label{eq:cost-function-primal-sdp}
    \begin{aligned}
        F(X) := \langle C,X \rangle  + \rho \max \{\lambda_{\max} (-X) ,0 \}.
    \end{aligned}
\end{equation}
Directly applying the bundle method in \cref{subsection:bundle-methods} to solve \cref{eq:SDP-primal-penalized} requires computing a subgradient of $F$ at every iteration $t$. It is known that for every $X_t \in \mathbb{S}^n$, a subgradient $g_t\in \partial F(X_t)$ is given by \cite[Example 2.89]{ruszczynski2011nonlinear}
\begin{equation*}
    \begin{aligned}
        g_t = \begin{cases}
            C - \rho v_t v_t^\tr ,& \text{ if } \lambda_{\max}(-X_t)>0, \\
            C,& \text{ otherwise},
        \end{cases}
    \end{aligned}
\end{equation*}
where $v_t \in \mathbb{R}^n$ is a normalized eigenvector associated with $\lambda_{\max}(-X_t)$. 

\subsubsection{Lower approximation models.} As discussed in \Cref{subsection:lower-approximation-bound}, one key step in the bundle method is to construct a valid lower approximation model of $F$ at each iteration $t$. Similar to \eqref{eq:piecewise-linear-LB}, one natural choice is to use a piece-wise affine function, 
\begin{equation*}
    \begin{aligned}
        \hat{F}_{t}(X) &  = \max_{i=1,\ldots,t}  \;\; \langle C,X_i \rangle  +\rho  \max \{\lambda_{\max} (-X_i) ,0 \}  +  \langle g_i , X-X_i \rangle,
    \end{aligned}
\end{equation*}
where $X_i, i = 1, \ldots, t$ are the past iterates, and $g_i \in \partial F(X_i), i = 1, \ldots, t$ are subgradients. 
Via simple derivations, each affine function corresponding to $X_i$ becomes  
\begin{equation} \label{eq:eigenvalue-lowerbound}
            \langle C,X_i \rangle  +\rho  \max \{\lambda_{\max} (-X_i) ,0 \}  +  \langle g_i , X-X_i \rangle \\
           \! = \!\begin{cases}
                \langle C,X\rangle + \rho \left \langle v_i v_i^\tr, -X \right \rangle , &  \text{if } \lambda_{\max} (-X_i) > 0, \\
                \langle C,X \rangle, & \text{otherwise},
            \end{cases} \\
\end{equation}
where $v_i \in \mathbb{R}^n$ is a normalized eigenvector corresponding to $\lambda_{\max}(-X_i)$.

For this special cost function, 
one key idea of the original spectral bundle method in \cite{helmberg2000spectral} is to improve the lower bound \cref{eq:eigenvalue-lowerbound} using infinitely many affine minorants. In particular, at iteration $t$, we compute a matrix $P_t \in \mathbb{R}^{n \times r}$ with some small value $1 \leq r<n$ and orthonormal columns (i.e., $P_t^\tr P_t = I \in \mathbb{S}^r$), and define a lower approximation function 
\begin{equation}\label{eq:F_P_t}
    \begin{aligned}
        \hat{F}_{P_t}(X) &= \langle C,X \rangle +  \rho \max_{S \in \mathbb{S}^r_+ , \Trace(S) \leq 1}  \left  \langle P_t S P_t^\tr,   -X \right \rangle.
    \end{aligned}
\end{equation}
It is clear that $F(X) \geq \hat{F}_{P_t}(X), \forall X \in \mathbb{S}^n$ thanks to the fact that 
\begin{equation*}
    \begin{aligned}
        \max \{\lambda_{\max}(-X) ,0 \} = \max_{S \in \mathbb{S}^{n}_+ , \Trace(S) \leq 1} \langle S,-X\rangle, \qquad \forall X \in \mathbb{S}^n,
    \end{aligned}
\end{equation*}
and 
\begin{equation*}
    \begin{aligned}
    \{P_t S P_t^\tr \in \mathbb{S}^n_+ \mid S \in \mathbb{S}^r_+, \Trace(S) \leq 1\} \subset \{S \in \mathbb{S}^n_+\mid \Trace(S) \leq 1\}.
    \end{aligned}
\end{equation*}
Meanwhile, it is not difficult to check that if $r = 1$, and $P_t = v_t$ with $v_t$ being the top eigenvector of $-X_t$, then $\hat{F}_{P_t}(X)$ defined in \cref{eq:F_P_t} is reduced to the approximation function in \cref{eq:eigenvalue-lowerbound} with $i = t$. Thus, when choosing $r > 1$ and selecting $P_t$ spanning $v_t$, we have a strictly better lower approximation using \cref{eq:F_P_t} than the simple linear function \cref{eq:eigenvalue-lowerbound} based on one subgradient. In principle, the columns of $P_t \in \mathbb{R}^{n \times r}$ should consist of both top eigenvectors of the current iterate $X_t$ and the accumulation of spectral information from past iterates $X_1, \ldots, X_{t-1}$ \cite{helmberg2000spectral,ding2020revisit}.

Therefore, by construction, $\hat{F}_{P_t}$ in \cref{eq:F_P_t}  is naturally a minorant satisfying \cref{eq:bundle-method-property-1}, and it also satisfies subgradient lowerbound \cref{eq:bundle-method-property-2}. However, a further refinement is needed for \cref{eq:F_P_t} to fulfill the model subgradient lower bound \cref{eq:bundle-method-property-3}. The spectral bundle method  \cite{helmberg2000spectral,ding2020revisit} maintains a carefully selected weight to capture past information. In particular, we introduce a matrix $\bar{W}_t \in \mathbb{S}^n_+$ and $\Trace (\bar{W}_t) = 1$, and then build the lower approximated model using $\bar{W}_t$ and $P_t \in \mathbb{R}^{n \times r}$,
\begin{equation}\label{eq:F_W_t_P_t}
    \begin{aligned}
        \hat{F}_{(\bar{W}_t,P_t)}(X) = \langle C,X \rangle +  \rho \max_{S \in \mathbb{S}^r_+ , \gamma \geq 0, \gamma + \Trace(S) \leq 1 }    \left \langle \gamma \bar{W}_t + P_tSP_t^\tr, -X \right\rangle.
    \end{aligned}
\end{equation}
It is clear that this lower approximation model \cref{eq:F_W_t_P_t} is an improved approximation than \cref{eq:F_P_t} (e.g., letting $\gamma = 0$ in \cref{eq:F_W_t_P_t} recovers the approximation model \cref{eq:F_P_t}). Thus, it satisfies the inequalities \cref{eq:bundle-method-property-1} and \cref{eq:bundle-method-property-2}. Upon careful construction of $\bar{W}_t$ at each iteration, we will show that \cref{eq:F_W_t_P_t} will also satisfy \cref{eq:bundle-method-property-3}. The construction details will be presented below. 

\subsubsection{Spectral bundle algorithms.}  
We are ready to introduce a new family of spectral bundle algorithms for primal SDPs, which we call \SBMP \footnote{This name is consistent with~\cite{ding2020revisit} which focuses on solving the dual penalization formulation \cref{eq:SDP-dual-penalized}.}. As we will detail below, the family of spectral bundle algorithms considers $r = \rpast + \rcurrent$ (where $\rpast\geq 0, \rcurrent \geq 1$) normalized eigenvectors to form the orthonormal matrix $P_t  \in  \mathbb{R}^{n \times r}$ that is used in \cref{eq:F_W_t_P_t}. The algorithms have the following steps. 

 \textit{Initialization:} \SBMP~starts with an initial guess $\initposi \in \mathbb{S}^n$, and $P_0 \in \mathbb{R}^{n \times {r}}$ formed by the $r = \rpast + \rcurrent$ top normalized eigenvectors of $-\initposi$, any weight matrix $\bar{W}_0 \in \mathbb{S}^n_+$ with $\Trace(\bar{W}_0) = 1$. Matrices $P_0$ and $\bar{W}_0$ are used to construct an initial lower approximation model $\hat{F}_{(\bar{W}_0,P_0)}$ in~\cref{eq:F_W_t_P_t}. 

 \textit{Solve the master problem:}  Similar to \cref{eq:bundle-method-t}, at iteration $t \geq 0$,  \SBMP~solves the following regularized master problem 
\begin{equation} \label{eq:SBM-subproblem}
    \begin{aligned}
     (\candidate, S_t^\star,\gamma^\star_t) =   \argmin_{X \in \mathcal{X}_0 } \;\; \hat{F}_{(\bar{W}_t,P_t)}(X)  + \frac{\alpha}{2 } \|X - \currposi\|^2,
    \end{aligned}
\end{equation} 
where $\currposi \in \mathbb{S}^n$ is the current reference point (proximal center), $\alpha > 0$ serves as a penalty for the deviation from $\currposi$, and 
$
\mathcal{X}_0 := \{ X \in \mathbb{S}^n \;|\; \mathcal{A}(X) = b \}.
$
Solving \cref{eq:SBM-subproblem} is the main computation in each iteration of \SBMP, and we provide its computational details in \Cref{subsection:regularized-master-problem}. 

 \textit{Update reference point:} Similar to \cref{eq:sufficient-descent}, \SBMP~updates the next reference point $\Omega_{t+1}$ as follows: given $\beta \in (0,1)$, if  
\begin{equation} \label{eq:null-or-serious-step-SBM}
    \begin{aligned}
       \beta \left(F(\currposi) - \hat{F}_{(\bar{W}_t,P_t)}(\candidate)\right) \leq F(\currposi) -  F(\candidate)
    \end{aligned}
\end{equation} 
holds true (i.e. at the candidate point $\candidate$, the decrease of the objective value of $F$ is at least $\beta$ fraction of the decrease in objective value that the model $\hat{F}_{(\bar{W}_t,P_t)}$ predicts), we set $\nextposi = \candidate$, which is called a \textit{descent step}. Otherwise, we let $\nextposi  = \currposi$, which is called a \textit{null step}. 

 \textit{Update the lower approximation model:} \SBMP~updates the spectral matrices $\bar{W}_{t+1}, P_{t+1}$ for the lower approximation model $\hat{F}_{(\bar{W}_{t+1},P_{t+1})}$ using a strategy similar to that in \cite{ding2020revisit,helmberg2000spectral}. We first compute the eigenvalue decomposition of the small $r \times r$ matrix $ S^\star_t $ as
 $$
 S^\star_t =  \begin{bmatrix} Q_1 & Q_2 \end{bmatrix} \begin{bmatrix} \Sigma_1 & 0 \\ 0 & \Sigma_2 \end{bmatrix}\begin{bmatrix} Q_1^\tr \\ Q_2^\tr \end{bmatrix},
 $$
 where $Q_1 \in \mathbb{R}^{r \times r_{\mathrm{p}}}$ consists of the top $r_{\mathrm{p}} \geq 0$ orthonormal eigenvectors, $\Sigma_1 \in \mathbb{S}^{r_{\mathrm{p}}}$ is a diagonal matrix formed by the top $r_{\mathrm{p}}$ eigenvalues of $S^\star_t$, and $Q_2 \in \mathbb{R}^{r \times  (r- r_{\mathrm{p}})}$ and $\Sigma_2 \in \mathbb{S}^{(r-r_{\mathrm{p}})}$ captures the remaining orthonormal eigenvectors and eigenvalues, respectively. 
 
\begin{itemize}
    \item The orthonormal matrix $P_{t+1}$: we compute $V_t \in \mathbb{R}^{n \times r_\mathrm{c}}$ with its columns being the top $r_{\mathrm{c} }\geq 1$ eigenvectors of $-\candidate$, which naturally contains the subgradient information of the true objective function $F$ at $\candidate$. Let the range space of  $P_{t+1}$ span $V_t$, which guarantees to improve the lower approximation model. We also let $P_{t+1}$ contain the past \textit{important} information $P_tQ_1$ according to the top $r_{\mathrm{p}}$ eigenvalues of $S^\star_t$. Therefore, we update $P_{t+1}$ as
    \begin{equation} \label{eq:update-P-t}
    P_{t+1}  = \mathrm{orth} \left(\begin{bmatrix} 
             V_t & P_tQ_1
        \end{bmatrix}\right), 
    \end{equation}
    where $\mathrm{orth} (\cdot)$ denotes an orthonormal process such that $P_{t+1}^\tr P_{t+1} = I_r$ with $r = r_{\mathrm{p}} + r_{\mathrm{c}}$.  
    \item The weight matrix $\bar{W}_{t+1}$: we keep the rest of the past information by updating $\bar{W}_{t+1}$ as
    \begin{equation} \label{eq:update-W-t}
        \bar{W}_{t+1} =  \frac{1}{\gamma^\star_t  + \Trace (\Sigma_2) } \left( \gamma^\star_t \bar{W}_t + P_t Q_2 \Sigma_2 Q_2^\tr P_t^\tr \right).
    \end{equation}
    Note that $\bar{W}_{t+1}$ has been normalized such that $\Trace (\bar{W}_{t+1}) = 1$. 
\end{itemize}
 
\noindent If $r_{\mathrm{p}} = 0$, the updates in \cref{eq:update-P-t} and \cref{eq:update-W-t} become 
$$ P_{t+1} = V_t , \qquad \text{and} \qquad \bar{W}_{t+1} =  \frac{1}{\Trace (\Wtstar) } \Wtstar,$$ where we denote the optimal solution of $\gamma \bar{W}_t + P_tSP_t^\tr$ in \cref{eq:SBM-subproblem} as 
\begin{equation} \label{eq:optimal-W-t}
    \Wtstar =  \gamma^\star_t \bar{W}_t + P_t S^\star_t P_t^\tr. 
\end{equation}

Overall, \SBMP~generates a sequence of points $\{\currposi, \Wtstar,  \ytstar\}$, where $\ytstar$ is the dual variable corresponding to the affine constraint in \cref{eq:SBM-subproblem}, and a sequence of monotonically decreasing cost values $\{F(\currposi)\}$. The detailed steps  of \SBMP~are listed in \Cref{alg:(r_p-r_c)-SBM-P}.
 
\begin{algorithm}[t] 
    \caption{\SBMP-- \textbf{S}pectral \textbf{B}undle \textbf{M}ethod for \textbf{P}rimal SDPs} \label{alg:(r_p-r_c)-SBM-P}
    \begin{algorithmic}[1]
        \Require{Problem data $A_1, \ldots,A_m, C \in \mathbb{S}^n$, $b\in \mathbb{R}^n$. Parameters $r_\mathrm{p} \geq 0, r_\mathrm{c} \geq 1$, $\alpha > 0$, $\beta \in (0,1)$, $\epsilon \geq 0 $, $t_{\max} \geq 1$. An initial point $ \initposi \in \mathbb{S}^n$.} \\
        \textbf{Initialization:} Let ${r} =r_\mathrm{p}+r_\mathrm{c}$. Initialize $  \bar{W}_0 \in \mathbb{S}^n_+$ with $\Trace(\bar{W}_0) = 1$. Compute $P_0  \in \mathbb{R}^{n \times {r}}$ with columns being the top ${r}$ orthonormal eigenvectors of $-\initposi$.
        \For{$t=0,\ldots,t_{\max}$}
            \State Solve \cref{eq:SBM-subproblem} to obtain $\candidate, \gamma^\star_t, S^\star_t$. 
            \Comment{\textit{master problem}}
            \State Form the iterate $\Wtstar$ in \cref{eq:optimal-W-t} and dual iterate $\ytstar$ in \cref{eq:SBP-equivalence}.
            \If {$t=0$ \textbf{and} $\Amap(\Omega_t) \neq b$ }
                \State Set primal iterate $  \nextposi = \candidate$.  \Comment{\textit{affine constraint}}
            \Else
                 \If {\cref{eq:null-or-serious-step-SBM} holds} 
                    \State Set primal iterate $  \nextposi = \candidate$. 
                    \Comment{\textit{descent step}}
                \Else 
                    \State Set primal iterate $  \nextposi = \currposi$. 
                    \Comment{\textit{null step}}
                \EndIf     
            \EndIf
            \State Compute $P_{t+1}$ as \cref{eq:update-P-t}, and $\bar{W}_{t+1}$ as \cref{eq:update-W-t}.  
            \Comment{\textit{update model}}
            \If {stopping criterion}
                \State Quit. 
                \Comment{\textit{termination}}
                \EndIf
        \EndFor
    \end{algorithmic}
\end{algorithm}

\subsection{Computational details} \label{subsection:regularized-master-problem}

At every iteration $t$, we need to solve the subproblem \cref{eq:SBM-subproblem}, which is the main computation in  \SBMP.
Therefore, it is crucial to solve the master problem  \cref{eq:SBM-subproblem} efficiently.
We summarize the computational details of solving \cref{eq:SBM-subproblem} in the following proposition.

\begin{proposition} \label{proposition-primal-penalty}
    The master problem \cref{eq:SBM-subproblem} is equivalent to the following problem
    \begin{align}
       \min_{W \in \hat{\mathcal{W}}_t,~y \in \mathbb{R}^m} \;\; \langle W-C,\currposi \rangle  - \langle b- \mathcal{A}(\currposi)  ,y \rangle+ \frac{1}{2 \alpha } \left\|W-C+\sum_{i=1}^{m}A_i y_i\right\|^2, \label{eq:SBP-equivalence}
    \end{align} 
where 
the constraint set is defined as
\begin{equation} \label{eq:constraint-W_t}
    \hat{\mathcal{W}}_t := \{  \gamma \bar{W}_t + P_t S P_t^\tr \in \mathbb{S}^n \;|\; S \in \mathbb{S}^r_+ , \gamma \geq 0, \gamma + \Trace(S) \leq \rho \}.
\end{equation}
The optimal $X$ in \cref{eq:SBM-subproblem} is recovered by
\begin{align} \label{eq:optimality-condition-Xt}
     \candidate   = \currposi + \frac{1}{\alpha} \left(\Wtstar-C+ \Ajmap(y_{t}^\star)\right),
\end{align}
where $(\Wtstar,y_{t}^\star)$ is a minimizer of \cref{eq:SBP-equivalence}.
\end{proposition}
\begin{proof}
     Our proof relies on the strong duality of convex optimization. Upon applying the definition of $\hat{F}_{(\bar{W}_t,P_t)}(\cdot)$ in \eqref{eq:F_W_t_P_t}, it is clear that \cref{eq:SBM-subproblem} becomes
    \begin{align*}
       & \min_{X \in \mathcal{X}_0}   \langle C,X \rangle +  \rho \max_{S \in \mathbb{S}^r_+ , \gamma \geq 0, \gamma + \Trace(S) \leq 1 }   \;\;   \langle \gamma \bar{W}_t + P_tSP_t^\tr, -X\rangle + \frac{\alpha}{2 } \|X - \currposi\|^2 \nonumber \\
       = & \min_{X \in \mathcal{X}_0} \max_{S \in \mathbb{S}^r_+ , \gamma \geq 0, \gamma + \Trace(S) \leq \rho }  \;\;  \langle C,X \rangle +\langle \gamma \bar{W}_t + P_tSP_t^\tr, -X\rangle + \frac{\alpha}{2 } \|X - \currposi\|^2 \nonumber \\
       = & \min_{X \in \mathcal{X}_0 }  \max_{W \in  \hat{\mathcal{W}}_t}  \;\; \langle C-W,X \rangle  + \frac{\alpha}{2 } \|X - \currposi\|^2,
    \end{align*}
    where the first equality brings the constant $\rho$ into the constraint, and the second equality applies a change of variables $W = \gamma \bar{W}_t + P_t S P_t^\tr$ and uses the set $ \hat{\mathcal{W}}_t$ defined as~\cref{eq:constraint-W_t}. Since $\hat{\mathcal{W}}_t$ is bounded, by strong duality \cite[Corollary 37.3.2]{rockafellar1970convex}, we can switch the min-max order
    and have the following equivalency
    \begin{equation*}
    \begin{aligned} 
          \min_{X \in \mathcal{X}_0 }  \max_{W \in  \hat{\mathcal{W}}_t}  \;\; \langle C-W,X \rangle  + \frac{\alpha}{2 } \|X - \currposi\|^2
       =   \max_{W \in  \hat{\mathcal{W}}_t} \min_{X \in \mathcal{X}_0 }   \;\;\langle C-W,X \rangle + \frac{\alpha}{2 } \|X - \currposi\|^2.
    \end{aligned}
    \end{equation*}
    
Note that the inner minimization is an equality-constrained quadratic program,
\begin{equation} \label{eq:inner-optimization}
    \min_{X \in \mathcal{X}_0 }   \;\;\langle C-W,X \rangle + \frac{\alpha}{2 } \|X - \currposi\|^2,
\end{equation}
which can be simplified by considering its dual formulation. Specifically, we introduce a dual variable $ y \in \mathbb{R}^{m}$ and construct the Lagrangian for \cref{eq:inner-optimization} as follows
    \begin{align*}
        L(X,y) & = \langle C-W,X \rangle  + \frac{\alpha}{2} \|X-\currposi\|^2 + y^\tr(b-\Amap(X)),
    \end{align*}
    which is strongly convex in $X$. The dual function for \cref{eq:inner-optimization} is given by $g(y) := \min_X L(X,y)$, where the unique minimizer $X$ is    
    \begin{align} \label{eq:dual-form-X}
              X  = \currposi + \frac{1}{\alpha} \left(W-C+\sum_{i=1}^{m}A_i y_i\right).
    \end{align}
Therefore, the dual function becomes 
$$
g(y) =  \langle C-W,\currposi \rangle + \langle b -\mathcal{A}(\currposi)  ,y \rangle- \frac{1}{2 \alpha} \left\|W-C+\sum_{i=1}^{m}A_i y_i\right\|^2.
$$
By strong duality of \cref{eq:inner-optimization}, we have  
$$
\begin{aligned}
&\max_{W \in  \hat{\mathcal{W}}_t} \min_{X \in \mathcal{X}_0 }   \;\; \langle C-W,X \rangle + \frac{\alpha}{2 } \|X - \currposi\|^2  \\
= & \max_{W \in  \hat{\mathcal{W}}_t} \max_{y \in \mathbb{R}^m}   \;\; \langle C-W,\currposi \rangle + \langle b -\mathcal{A}(\currposi)  ,y \rangle- \frac{1}{2 \alpha} \left\|W-C+\sum_{i=1}^{m}A_i y_i\right\|^2  \\
=& \max_{W \in  \hat{\mathcal{W}}_t, y \in \mathbb{R}^m }  \;\; \langle C-W,\currposi \rangle + \langle b -\mathcal{A}(\currposi)  ,y \rangle- \frac{1}{2 \alpha} \left\|W-C+\sum_{i=1}^{m}A_i y_i\right\|^2,
\end{aligned} 
$$
which is clearly equivalent to \cref{eq:SBP-equivalence}. Finally, the optimal $X$ in \cref{eq:optimality-condition-Xt}  is recovered in \cref{eq:dual-form-X} once we obtain the optimal dual variables $\ytstar$ and $\Wtstar$ from solving \cref{eq:SBP-equivalence}. This completes the proof. 
\end{proof}

After the first iteration, if $\Amap(\Omega_0) \neq b$, we update the proximal center 
$\Omega_{1} = X^\star_{1}$.
Then, the rest of the iterations are naturally feasible to the affine constraint, i.e., $A(\currposi) = b,~ \forall t > 0$.
Therefore, the master problem  \cref{eq:SBP-equivalence} can be further simplified as 
\begin{align*}  
         \min_{W \in \hat{\mathcal{W}}_t,~y \in \mathbb{R}^m} \;\; \langle W-C,\currposi \rangle + \frac{1}{2 \alpha } \left\|W-C+\sum_{i=1}^{m}A_i y_i\right\|^2.
\end{align*} 

\begin{remark}[Complexity of solving~\cref{eq:SBP-equivalence} in \SBMP] \label{remark:master-problem-primal}
The subproblem \cref{eq:SBP-equivalence} is a semidefinite program with a convex quadratic cost function. The dimension of the PSD constraint is $r = \rpast+\rcurrent$, which can be chosen to be very small (i.e., $r \ll n$). Thus,  \cref{eq:SBP-equivalence} can be efficiently solved using either standard conic solvers (such as SeDuMi \cite{sedumi} and Mosek \cite{mosek}) or customized interior-point algorithms \cite{helmberg2000spectral}. In addition, we note that the dual variable $y \in \mathbb{R}^m$ in \cref{eq:SBP-equivalence} admits an analytical closed-loop solution in terms of $W$, and thus \cref{eq:SBP-equivalence} can be further converted into the following form 
\begin{equation}
    \begin{aligned} \label{eq:SBM-ms-pb-QP-abstract}
    \min_{v \in \mathbb{R}^{1+r^2} }  \quad&
    v^\tr Q v + q^\tr v + c  \\
    \mathrm{subject~to} \quad & v =  \begin{bmatrix} \gamma & \vectorize{S}^\tr  \end{bmatrix}^\tr,\\
    & \gamma \geq 0,S \in \mathbb{S}^r_+, \gamma + \Trace(S) \leq \rho,       
    \end{aligned}
\end{equation}
where $c \in \mathbb{R}$ is a constant, and $Q \in \mathbb{S}^{1+r^2}$ and $q \in \mathbb{R}^{1+r^2}$ depend on the problem data and the matrices $P_t$ and $\bar{W}_t$ at step $t$. 
We present the details of transforming \cref{eq:SBP-equivalence} into \cref{eq:SBM-ms-pb-QP-abstract} 
in \Cref{sec:reform-ms-pb}. Therefore, at each iteration of \SBMP, we only need to solve \eqref{eq:SBM-ms-pb-QP-abstract}, which admits very efficient solutions since $r$ can be much smaller than the original dimension $n$. 
As we shall see in \Cref{appendix:reformulation}, the computational complexity of solving the regularized sub-problem \cref{eq:SBM-subproblem} in spectral bundle methods for primal SDPs is very similar to that in spectral bundle methods for dual SDPs. Furthermore, when $r = 1$ (i.e., $\rpast = 0, \rcurrent=1$), the problem \cref{eq:SBM-ms-pb-QP-abstract} admits an analytical closed-form solution, and no other solver is required; see \Cref{appendix: r=1}. 
\end{remark}

\subsection{Convergence results}
We here present two convergence guarantees \Cref{thm: sublinearates-P-K} and \Cref{thm: linear-convergence} for~\SBMP. 
Consistent with \Cref{lemma-iterations-bound}, when strong duality holds for \cref{eq:SDP-primal} and \cref{eq:SDP-dual}, 
\Cref{thm: sublinearates-P-K} below provides a convergence rate of $\bigO(1/\epsilon^3)$ in terms of cost value gap, approximate primal feasibility, approximate dual feasibility, and approximate primal-dual optimality gap. The convergence rate improves to $\bigO(1/\epsilon)$ under the condition of strict complementarity  (see \Cref{eq:strictly-complementarity}). 
\begin{theorem}\label{thm: sublinearates-P-K}
{Suppose \cref{assumption:linearly-independence,assumption-slater-condition} are satisfied.} 
 Given any $\beta\in(0,1)$, $\rcurrent \geq 1$, $\rpast \geq 0 $, $ \alpha>0$, $r= \rcurrent+\rpast$,
	$\rho > 2\DZstar$,
 and target accuracy $\epsilon > 0 $, the \SBMP~in \Cref{alg:(r_p-r_c)-SBM-P} 
	produces iterates $\currposi$, $\Wtstar$, and $\ytstar$  with
	$F(\currposi)-F(\Xstar)\leq \epsilon$ and
	\begin{subequations} \label{eq:feasibility}
	\begin{align}
 \text{approximate primal feasibility: }& \quad  \Amap(\currposi) -b = 0, \quad \lambda_{\min}(\currposi)\geq -\epsilon,  \label{eq:feasibility-b}\\
	\text{approximate dual feasibility: }& \quad\left\|\Wtstar+\Ajmap(\ytstar) - C\right\|^2 \leq \epsilon, \quad  \Wtstar\succeq 0, \label{eq:feasibility-a}\\
	\text{approximate primal-dual optimality: }&\quad  | \langle C, \currposi \rangle - \langle  b, \ytstar \rangle | \leq \sqrt{\epsilon} \label{eq:feasibility-c}
	\end{align}
	\end{subequations}
	by some iteration $t\leq \bigO(1/\epsilon^3)$.
	If additionally, strict complementarity holds for \cref{eq:SDP-primal} and \cref{eq:SDP-dual}, then these conditions \cref{eq:feasibility} are reached by some iteration $t\leq \bigO(1/\epsilon)$.
\end{theorem}

In addition to strict complementarity, an improved convergence rate can be established if the number of current eigenvectors at every iteration satisfies  
\begin{equation}\label{eq: rank-condition}
\rcurrent \geq \rnull :=\max_{\Xstar \in \Pstar }\dim(\text{null}(\Xstar))
\end{equation} 
with proper choice of $\alpha$ and $\beta$. Under these conditions, \Cref{thm: linear-convergence} ensures that \SBMP~converges linearly once the iterate $\currposi$ is close enough to the set of the primal optimal solutions $\Pstar$.

\begin{theorem}\label{thm: linear-convergence}
{Suppose \cref{assumption:linearly-independence,assumption-slater-condition} are satisfied and strict complementarity holds for \cref{eq:SDP-primal} and \cref{eq:SDP-dual}.}
        There exist constants $T_0>0$ and $\eta>0$. Under a proper selection of~$\alpha \geq \eta$, $\beta\in(0,\frac{1}{2}]$, $\rcurrent \geq \rnull $, $\rpast \geq 0 $, $r= \rcurrent + \rpast$,
	$\rho > 2\DZstar$,  
 and target accuracy $\epsilon > 0$, after at most ${T}_{0}$ iterations, the \SBMP~in \Cref{alg:(r_p-r_c)-SBM-P} only takes descent steps and converges linearly to an optimal solution. Consequently, \SBMP~produces iterates $\currposi$, $\Wtstar$, and $\ytstar$  satisfying $F(\currposi)-F(\Xstar)\leq \epsilon$ and \cref{eq:feasibility}
    by at most ${T}_{0} + \bigO(\log(1/\epsilon))$ iterations. 
\end{theorem}

The proof sketches are provided in \Cref{subsection:proof-sketches}. The constants $T_0$ and $\eta$ only depend on problem data and are independent of the sub-optimality $\epsilon$, and we provide some discussions in \Cref{subsection:proof-linear-convergence} (see \Cref{lemma:quadratic-close} and \Cref{subsec:estimation-T} in the appendix for further details).

\begin{remark}
    The convergence results in \Cref{thm: sublinearates-P-K,thm: linear-convergence} can be viewed as the counterparts of \cite[Theorems 3.1 and 3.2]{ding2020revisit} when applying the spectral bundle method for solving primal SDPs. Note that \cite[Theorems 3.1 and 3.2]{ding2020revisit} focus on solving dual SDPs only (we will review some details in \Cref{subsection:SBM-dual}). As~highlighted earlier, all existing studies \cite{helmberg2000spectral,helmberg2002spectral,helmberg2014spectral,ding2020revisit} only consider the penalized dual formulation \cref{eq:SDP-dual-penalized}. Here, we establish a class of spectral bundle methods, i.e., \SBMP~in \Cref{alg:(r_p-r_c)-SBM-P}, to directly solve primal SDPs with similar computational complexity and convergence behavior. However, we remark that the value of $r_{\mathrm{c}}$ in \eqref{eq: rank-condition} can be drastically different from that in \cite{ding2020revisit} for linear convergence. A detailed comparison will be presented in \Cref{sec:comparison-connection}.
\end{remark}

\subsection{Proof sketches} \label{subsection:proof-sketches}
Here, we provide some proof sketches for \Cref{thm: sublinearates-P-K,thm: linear-convergence}. The proofs largely follow the strategies in \cite[Section 3]{ding2020revisit}, \cite[Section 4]{du2017rate} and \cite[Section 7]{ruszczynski2011nonlinear}. We complete some detailed calculations, handle the constrained case for the penalized primal SDPs \eqref{eq:SDP-primal-penalized} (note that the penalized dual case is unconstrained), and fix minor typos in \cite{ding2020revisit,du2017rate,ruszczynski2011nonlinear}. We have also provided an extension in \Cref{lemma:linear-contraction} compared to \cite[Section 3]{ding2020revisit}. We do not claim main contributions for establishing the proofs, and we provide them for self-completeness and for the convenience of interested readers. 

\subsubsection{Proof of \cref{thm: sublinearates-P-K}} One main step in the proof of \cref{thm: sublinearates-P-K} is to characterize the improvements in terms of primal feasibility, dual feasibility, and primal-dual optimality at every descent step, which is summarized in the following lemma. Its proof is provided in \Cref{Appendix:technical-proofs}.

\begin{lemma} \label{lemma:Primal-Dual-Gap-Feasibility}
        In \SBMP, let $\beta\in(0,1)$, $\rcurrent \geq 1$, $\rpast \geq 0 $, $ \alpha>0$, $r= \rcurrent+\rpast$,
	$\rho > 2\DZstar$. Then, at every descent step $t>0$, the following results hold.
	\begin{enumerate}
	\item The approximate primal feasibility for $\Omega_{t+1}$ satisfies
	\begin{equation*}
        \begin{aligned}
	    \lambda_{\min} (\Omega_{t+1}) \geq \frac{-(F(\currposi)-F(\Xstar))}{ \DZstar}, \quad and \quad \mathcal{A}(\Omega_{t+1}) = b.
     \end{aligned}
	\end{equation*}
 \item The approximate dual feasibility for $(\Wtstar, \ytstar)$ satisfies 
	\begin{equation*}
	\begin{aligned}
     \Wtstar \succeq 0,\quad and \quad \|\Wtstar-C+ \Ajmap(\ytstar) \|^2 \leq \frac{2 \alpha}{\beta} (F(\currposi)-F(\Xstar)).
     \end{aligned} 
	\end{equation*}
	\item The approximate primal-dual optimality for $(\nextposi, \Wtstar, \ytstar)$ satisfies
    \begin{equation*}
	    \begin{aligned}
             \langle C, \nextposi \rangle - \langle b,\ytstar \rangle & \geq -\frac{\rho (F(\currposi)-F(\Xstar))}{ \DZstar  } -  \Dxo \sqrt{\frac{2 \alpha}{\beta} (F(\currposi)-F(\Xstar))}, \\  
            \langle C, \nextposi \rangle - \langle b,\ytstar \rangle & \leq \frac{1-\beta}{\beta} (F(\currposi)-F(\Xstar))  +  \Dxo \sqrt{\frac{2 \alpha}{\beta} (F(\currposi)-F(\Xstar))}, 
	    \end{aligned}
	\end{equation*}
    where $\Dxo := \max_{F(\Omega_t) \leq F(\Omega_0) }\|\Omega_t\|$
    is bounded due to the compactness of $\Pstar$(see \cref{lemma:compact-sublevel-set}).
    \end{enumerate}
\end{lemma}

\Cref{thm: sublinearates-P-K} is a direct consequence by combining \Cref{lemma:Primal-Dual-Gap-Feasibility} with the convergence results of the generic bundle method in \Cref{lemma-iterations-bound}. 

\vspace{1mm}

\noindent \textbf{Proof of \Cref{thm: sublinearates-P-K}:} Consider the generic convex optimization \cref{eq:non-smooth-problem}. For any lower approximation function $\hat{f}_{t+1}$ satisfying \cref{eq:bundle-method-property-1}-\cref{eq:bundle-method-property-3}, \Cref{lemma-iterations-bound} guarantees that the generic bundle method in \Cref{algorithm:bundle} generates iterates $\omega_{t}$ with the gap $f(\omega_{t}) - f(x^\star)$ converging to zero at a rate of $\bigO (1/\epsilon^3)$. 
This rate improves to $\bigO (1/\epsilon)$ whenever the quadratic growth condition \cref{eq:quadratic-growth-lemma} occurs. 

In our case of SDPs, the quadratic growth of the primal penalized cost function $F$  in \cref{eq:cost-function-primal-sdp} 
holds whenever strict complementarity holds (see \Cref{lemma:quadratic-growth}): fix $\epsilon > 0$ and define a $\epsilon$-sublevel set
\begin{align}\label{eq:sublevelset-main-text}
\mathcal{P}_\epsilon = \{X\in \mathbb{S}^n\;|\; \Amap(X)=b, F(X) - F(\Xstar) \leq  \epsilon\}, 
\end{align}
 then there exist a constant $\mu > 0$  such that 
\begin{equation} \label{eq:quadratic-growth-F(x)}
 F(X)-F(\Xstar) \geq \mu \cdot \Dist^2(X,\Pstar), \,\; \forall X \in \mathcal{P}_\epsilon.
\end{equation}
In other words, the square of the distance between a point $X$ and the set of primal optimal solutions  $\Pstar$ is bounded by the optimality of the objective value when strict complementarity holds.
Therefore, if the lower approximation model $\hat{F}_{(\bar{W}_t,P_t)}$ in \cref{eq:F_W_t_P_t} satisfies \cref{eq:bundle-method-property-1,eq:bundle-method-property-2,eq:bundle-method-property-3}, the gap $F(\currposi) - F(X^{\star})$ from \SBMP~converges to zero at the rate of $\bigO \left(1/\epsilon^3\right)$. This rate improves to $\bigO (1/\epsilon)$ whenever strict complementarity holds. 

By \cref{lemma:Primal-Dual-Gap-Feasibility}, the convergence results for approximate primal feasibility \cref{eq:feasibility-b}, approximate dual feasibility \cref{eq:feasibility-a}, and primal and dual gap \cref{eq:feasibility-c} are naturally established. It remains to verify \cref{eq:bundle-method-property-1,eq:bundle-method-property-2,eq:bundle-method-property-3} for the lower approximation $\hat{F}_{(\bar{W}_t,P_t)}$, which is provided in \Cref{appendix:verification}. This completes the proof. \hfill $\blacksquare$

\subsubsection{Proof of \Cref{thm: linear-convergence}}  \label{subsection:proof-linear-convergence}

Thanks to \Cref{thm: sublinearates-P-K}, the iterate  $\currposi$ will be sufficiently close to the set of optimal solutions $\Pstar$ after some finite number of iterations (which is independent of $\epsilon$). In particular, let ${T}_0$ be the number of iterations that ensures 
$$ \min_{\Xstar \in \Pstar} \|\currposi - \Xstar\|_{\mathrm{op}} \leq \frac{\delta}{3},$$ 
where $\delta$ is a constant eigenvalue gap parameter (see~\cref{lemma:quadratic-close} in \Cref{appendix:theorem-linearly-convergence}). Then, the lower approximation model $\hat{F}_{(\bar{W}_t,P_t)}$ in \cref{eq:F_W_t_P_t} becomes quadratically close to the true penalized cost function $F$, as summarized in \Cref{lemma:quadratic-closeness}. Thanks to the quadratic closeness of $\hat{F}_{(\bar{W}_t,P_t)}$,  \Cref{alg:(r_p-r_c)-SBM-P} will take only descent steps after $T_0$ iterations, and we have contractions in terms of the cost value gap and the distance to the set of primal optimal solutions (\Cref{lemma:linear-contraction}). 

\begin{lemma} \label{lemma:quadratic-closeness}
    Suppose \cref{assumption:linearly-independence,assumption-slater-condition} are satisfied and strict complementarity holds for \cref{eq:SDP-primal} and \cref{eq:SDP-dual}. Let $r_{\mathrm{c}} \geq \rnull$. There exists a constant $\eta > 0$ (independent of $\epsilon$) such that for all iteration $t \geq T_0$, it holds that
    \begin{equation*} 
        \hat{F}_{(\bar{W}_{t},P_{t})} (X) \leq F(X) \leq \hat{F}_{(\bar{W}_{t},P_{t})} (X) + \frac{\eta}{2}\|X-\currposi\|^2, \quad \forall X \in \mathcal{X}_0. 
    \end{equation*}
\end{lemma}

\begin{lemma} \label{lemma:linear-contraction}
Under the conditions in \cref{lemma:quadratic-closeness}, for any $\alpha \geq \eta$ and $t \geq {T}_0$,  \Cref{alg:(r_p-r_c)-SBM-P} with $\beta \in (0, \frac{1}{2}]$ takes only descent steps and guarantees two contractions\footnote{Note that \cite[Lemma 3.8]{ding2020revisit} only shows the linear convergence in terms of 
the distance to the set of optimal solutions. 
We further establish the linear convergence in terms of the cost value gap $F(\currposi) - F(\Xstar)$, which can be directly used in \Cref{lemma:Primal-Dual-Gap-Feasibility}.} 
    \begin{subequations}
	\begin{align}
        \dist(\nextposi,\Pstar) &\leq \sqrt{ \frac{{\alpha}/{2}}{\mu +{\alpha}/{2}} }\dist(\currposi,\Pstar), \label{eq:contraction-iterate}  \\
        F(\nextposi) - F(\Xstar) &\leq \left(1- \min\left\{\frac{\mu}{2 \alpha},\frac{1}{2} \right\} \beta\right) \left(F(\currposi) - F(\Xstar)\right),
        \label{eq:contraction-cost}
 \end{align}
 \end{subequations}
 where $\mu$ is the quadratic growth constant in \cref{eq:quadratic-growth-F(x)} for the initial sublevel set $\mathcal{P}_{F(\Omega_0)}$ defined in \cref{eq:sublevelset-main-text}.
\end{lemma}

\noindent \textbf{Proof of \Cref{thm: linear-convergence}:} \Cref{thm: linear-convergence} is a direct consequence by combining \Cref{lemma:linear-contraction} with \Cref{lemma:Primal-Dual-Gap-Feasibility}. After $T_0$ iterations, \eqref{eq:contraction-cost} guarantees the linear convergence of $F(\currposi) - F(\Xstar)$. Thus, together with \Cref{lemma:Primal-Dual-Gap-Feasibility}, this ensures the linear convergence of approximate primal feasibility, approximate dual feasibility, and approximate primal-dual optimality after $T_0$ iterations. \hfill $\blacksquare$

We finally note that $T_0$ is a constant depending on problem data only. The proofs of \Cref{lemma:quadratic-closeness,lemma:linear-contraction} are provided in \Cref{appendix:theorem-linearly-convergence}.  
    
\section{Spectral bundle methods for dual SDPs} \label{section:SBM-dual}

In this section, we first review the existing spectral bundle method for dual SDPs that was originally proposed in \cite{helmberg2000spectral} and further developed in \cite{helmberg2002spectral,helmberg2014spectral,ding2020revisit}. The (sub)linear convergence results have been recently established in \cite{ding2020revisit}. We then compare the spectral bundle method for dual SDPs with that for primal SDPs developed in \cref{sec:SBM-Primal}.

\subsection{Spectral bundle algorithms for dual SDPs} \label{subsection:SBM-dual}
The idea in \cite{helmberg2000spectral} applies the standard bundle method in \Cref{subsection:bundle-methods} to the penalized dual formulation \cref{eq:SDP-dual-penalized}. One key step is to construct an appropriate lower approximation model.  
For notational simplicity, let us denote the objective function in \cref{eq:SDP-dual-penalized} as
\begin{equation*}
    \begin{aligned}
        F_{\mathrm{d}} (y) :=  -b^\tr y + \rho  \max \left\{\lambda_{\max} \left(\Ajmap(y)  - C\right) ,0 \right\}.
    \end{aligned}
\end{equation*}
Similar to \cref{eq:F_W_t_P_t}, the family of spectral bundle methods for dual SDPs uses a positive semidefinite definite matrix $ \bar{W}_t \in \mathbb{S}^n_{+}$ with $\Trace (\bar{W}_t) = 1$ and an orthonormal $P_t \in \mathbb{R}^{n \times r}$ matrix at iteration $t$ to lower approximate $F_d$. Specifically, the lower approximation model is constructed as 
\begin{equation} \label{eq:F_d_X_t_V_t}
    \begin{aligned}
        \hat{F}_{\mathrm{d},(\bar{W}_t,V_t)}(y) =-b^\tr y  + \rho  \max_{S \in \mathbb{S}^r_+ , \gamma \geq 0, \gamma + \Trace(S) \leq 1 }    \left   \langle \gamma \bar{W}_t + P_tSP_t^\tr, \mathcal{A}^*(y) -C  \right \rangle.
    \end{aligned}
\end{equation}
It is clear to see that $\hat{F}_{\mathrm{d},(\bar{W}_t,V_t)}(y) \leq F_{\mathrm{d}} (y), \forall y \in \mathbb{R}^m$. 

Following \cite{ding2020revisit}, we present a family of spectral bundle algorithms for dual SDPs, called \SBMD,  
which has the following steps: 
\begin{itemize}
    \item \textit{Initialization:} \SBMD~starts with an initial guess $\yinitposi \in \mathbb{R}^m$, $P_0 \in \mathbb{R}^{n \times r}$ formed by top $r$ normalized eigenvectors of $\Ajmap(\yinitposi) - C$, and $\bar{W}_0 \in \mathbb{S}^n$ with $\Trace(\bar{W}_0) = 1$. The matrices $P_0$ and $\bar{W}_0$ are used to construct an initial lower approximation model $\hat{F}_{\mathrm{d},(\bar{W}_t,V_t)}$ in~\cref{eq:F_d_X_t_V_t}. 
    \item \textit{Solve the master problem:}  Similar to \cref{eq:bundle-method-t}, at iteration $t \geq 0$,  \SBMD~solves the following regularized master problem 
\begin{equation}
    \begin{aligned}\label{eq:SBMD-subproblem}
       \left(\ycandidate,S_t^\star,\gamma^\star_t\right) = & \argmin_{y \in \mathbb{R}^n  } \;\; \hat{F}_{\mathrm{d},(\bar{W}_t,P_t)}(y)  + \frac{\alpha}{2} \|y -  \ycurrposi\|^2,
    \end{aligned}
\end{equation}
where $\ycurrposi \in \mathbb{R}^m$ is the current reference position, and $\alpha > 0$ is as a penalty parameter.

\item  \textit{Update reference point:} Similar to \cref{eq:sufficient-descent}, \SBMD~updates the next reference point $\omega_{t+1}$ as follows: given $\beta \in (0,1)$, if  
\begin{equation} \label{eq:null-or-serious-step-SBM-dual}
    \begin{aligned}
        \beta \left(F_\mathrm{d}(\ycurrposi) - \hat{F}_{\mathrm{d},(\bar{W}_t,P_t)}(\ycandidate)\right) \leq F_\mathrm{d}(\ycurrposi)- F_\mathrm{d}(\ycandidate)
    \end{aligned}
\end{equation} 
holds true, we let $\ynextposi = \ycandidate$ (\textit{descent step}), otherwise $\ynextposi  = \ycurrposi$ (\textit{null step}).

\item  \textit{Update the lower approximation model:}  \SBMD~updates the matrices $\bar{W}_{t+1}, P_{t+1}$ for the lower approximation model $\hat{F}_{\mathrm{d},(\bar{W}_t,V_t)}$ using formulas in \cref{eq:update-P-t,eq:update-W-t} with one difference that the matrix $V_t \in \mathbb{R}^{n \times \rcurrent}$ is formed by the top $\rcurrent \geq 1$ eigenvectors of $\Ajmap(\ycandidate) - C$.
\end{itemize}

The overall process for~\SBMD~is listed in \Cref{alg:(r_p-r_c)-SBM-D}. \SBMD~generates a sequence of iterates $\{\ycurrposi, \Wtstar\}$ with monotonically decreasing cost values $\{F_d(\ycurrposi)\}$. 
Solving the subproblem \cref{eq:SBMD-subproblem} is the main computation in each iteration of \SBMD. Similar to \cref{proposition-primal-penalty}, we have the following result.
\begin{proposition}[{\cite[Section 2.3]{ding2020revisit}}]
The master problem \cref{eq:SBMD-subproblem} is equivalent to
    \begin{align} 
       \min_{W \in \hat{\mathcal{W}}_t} \;\; \langle b,\ycurrposi \rangle   + \langle W, C - \Ajmap(\ycurrposi) \rangle + \frac{1}{2 \alpha } \left\|b - \Amap W\right\|^2, \label{eq:SBD-equivalence}
    \end{align} 
where the constraint set is defined as
\begin{equation} \label{eq:SBMD-constraint-W_t}
    \hat{\mathcal{W}}_t := \{  \gamma \bar{W}_t + P_t S P_t^\tr \in \mathbb{S}^n \;|\; S \in \mathbb{S}^r_+ , \gamma \geq 0, \gamma + \Trace(S) \leq \rho \}.
\end{equation}
The optimal $y$ in \cref{eq:SBMD-subproblem} is recovered as $\ycandidate   = \ycurrposi + \frac{1}{\alpha} \bigl( b- \Amap (\Wtstar) \bigr),$ where $\Wtstar$ is a minimizer of \cref{eq:SBD-equivalence}.
\end{proposition}

Similar to \Cref{remark:master-problem-primal}, problem \cref{eq:SBD-equivalence} is a quadratic problem with a small semidefinite constraint, which can be reformulated into a problem of the form \cref{eq:SBM-ms-pb-QP-abstract}; see \cref{sec:reform-ms-pb-dual}. 
The convergence results of \SBMD~are summarized in \cref{thm: sublinearates,thm: linear convergence of Block SBM under the extra condition strict complementarity}.

\begin{algorithm}[t] 
    \caption{\SBMD-- \textbf{S}pectral \textbf{B}undle \textbf{M}ethod for \textbf{D}ual SDPs } \label{alg:(r_p-r_c)-SBM-D}
    \begin{algorithmic}[1]
        \Require{Problem data $A_1, \ldots,A_m, C \in \mathbb{S}^n$, $b\in \mathbb{R}^n$. }
        \Require{Parameters $r_\mathrm{p} \geq 0, r_\mathrm{c} \geq 1$, $\alpha > 0$, $\beta \in (0,1)$, $\epsilon \geq 0 $. An initial point $ \omega_0 \in \mathbb{R}^m$.}\\
        \textbf{Initialization: } Let ${r} =\rpast+\rcurrent $. Initialize $  \bar{W}_0 \in \mathbb{S}^n_+$ with $\Trace(\bar{W}_0) = 1$. Compute $P_0  \in \mathbb{R}^{n \times {r}}$ with columns being the top ${r}$ orthonormal eigenvectors of $\Ajmap(\omega_0)-C$.
        \For{$t=0,\ldots,t_{\max}$}
            \State Solve \cref{eq:SBMD-subproblem} to obtain $\ycandidate, \gamma^\star_t, S^\star_t$. 
            \Comment{\textit{master problem}}
            \State Form the iterate $\Wtstar$ as in \cref{eq:SBMD-constraint-W_t}.
             \If {\cref{eq:null-or-serious-step-SBM-dual} holds} 
                \State Set primal iterate $  \ynextposi = \ycandidate$. 
                \Comment{\textit{descent step}}
            \Else 
                \State Set primal iterate $  \ynextposi = \ycurrposi$. 
                \Comment{\textit{null step}}
            \EndIf               
                     \State Compute $P_{t+1}$ as \cref{eq:update-P-t}, and $\bar{W}_{t+1}$ as \cref{eq:update-W-t}.  
            \Comment{\textit{update model}}
             \If {stopping criterion are met}
            \State Quit. 
            \Comment{\textit{termination}}
            \EndIf
        \EndFor
    \end{algorithmic}
\end{algorithm}

\begin{theorem} [{\cite[Theorem 3.1]{ding2020revisit}}]  \label{thm: sublinearates}
    Suppose \cref{assumption:linearly-independence,assumption-slater-condition} are satisfied.
        Given any $\beta\in(0,1)$, $\rcurrent \geq 1$, $\rpast \geq 0 $, $ \alpha>0$, $r= \rcurrent+\rpast $,
	$\rho > 2\DXstar$, $P_0\in \RR^{ n \times r}$, $\omega_0 \in \mathbb{R}^m$, and accuracy $\epsilon > 0$, the \SBMD~in \Cref{alg:(r_p-r_c)-SBM-D}
	produces iterates $\ycurrposi$ and $\Wtstar$ with
	$F_\mathrm{d}(\ycurrposi)-F_\mathrm{d}(\ystar )\leq \epsilon$ and
\begin{subequations} \label{eq:feasibility-dual}
 \begin{align}
	\text{approximate primal feasibility: }& \quad\|b - \Amap (\Wtstar) \|^2 \leq \epsilon, \quad  \Wtstar \succeq 0,\\
	\text{approximate dual feasibility: }& \quad \lambda_{\min}(C - \Ajmap (\ycurrposi) )\geq -\epsilon,\\
	\text{approximate primal-dual optimality: }&\quad  |\langle C, \Wtstar\rangle- \langle b, \ycurrposi \rangle | \leq \sqrt{\epsilon}
	\end{align}
 \end{subequations}
	by some iteration $t\leq \bigO(1/\epsilon^3)$.
	If additionally, strict complementarity holds, then these conditions are reached by some iteration $t\leq \bigO(1/\epsilon)$.
\end{theorem}

Along with strict complementarity, a further improvement on the convergence rate can be shown if the number of selected current eigenvectors at every iteration satisfies 
\begin{equation}
\rcurrent \geq \rnulld :=\max_{ (y^\star, \Zstar) \in \Dstar }\dim(\text{null}( \Zstar)).
\label{eq:rank-condition-dual}
\end{equation}

\begin{theorem}[{\cite[Theorem 3.2]{ding2020revisit}}]\label{thm: linear convergence of Block SBM under the extra condition strict complementarity}
	Suppose \cref{assumption:linearly-independence,assumption-slater-condition} are satisfied and strict complementarity holds for \cref{eq:SDP-primal} and \cref{eq:SDP-dual}. There exist constants $T_0>0$ and $\eta>0$. Under a proper selection of~$\alpha$ and any $\beta\in(0,\frac{1}{2}]$, $\rcurrent \geq  \rnulld$, $\rpast \geq 0 $, $r= \rcurrent + \rpast$,
	$\rho >2 \DXstar $, $P_0\in \mathbb{R}^{ n \times r}$, $\omega_0 \in \mathbb{R}^m$, and accuracy $\epsilon > 0$, after at most ${T}_{0}$ iterations, the \SBMD~in \Cref{alg:(r_p-r_c)-SBM-D} only takes descent steps and converges linearly to an optimal solution. Consequently, \SBMD~produces iterates $\ycurrposi$ and $\Wtstar$ with
	$F_\mathrm{d}(\ycurrposi)-F_\mathrm{d}(\ystar )\leq \epsilon$ and \cref{eq:feasibility-dual}
	by at most $T_0 + \bigO(\log(1/\epsilon))$ iterations.
\end{theorem}

The constants $T_0>0$ and $\eta>0$ only depend on problem data and are independent of $\epsilon$. We refer the interested reader to \cite[Section 3.3.1]{ding2020revisit} for details. 

\begin{remark}
 Historically, the spectral bundle method was first introduced in \cite{helmberg2000spectral} to tackle SDP relaxations for large-scale combinatorial problems. The method in \cite{helmberg2000spectral} works for dual SDPs with an explicit constant trace constraint; see \cref{eq:SDP-dual-penalized-trace}. The algorithm \cite{helmberg2000spectral} requires one current eigenvalue, i.e. $\rcurrent = 1$, and allows different values of $\rpast$. This selection of parameters is different from the selection of parameters reviewed in \Cref{alg:(r_p-r_c)-SBM-D} (we follow the choice in \cite{ding2020revisit}). The convergence results in \Cref{thm: linear convergence of Block SBM under the extra condition strict complementarity} show that the parameter $\rpast$ can be chosen as zero which still guarantees the linear convergence whenever $\rcurrent \geq \rnulld $, $\alpha$ is chosen correctly, and the SDP satisfies strict complementarity. The result is based on one critical observation that $\gamma \bar{W}_t + P_tSP_t^\tr$ in \cref{eq:F_d_X_t_V_t} is an approximation of the optimal primal variable $\Xstar$ and $P_t$ is an approximation of the null space of $Z^\star$. A key result in the analysis is based on a novel eigenvalue approximation in \cite[Lemma 3.9]{ding2020revisit}. 
 One immediate benefit of selecting $\rcurrent$ instead of $\rpast$ is that \SBMD~becomes particularly efficient when the primal SDP admits low-rank solutions as $\rcurrent$ can be chosen small to fulfill the rank condition \cref{eq:rank-condition-dual}. In this case, the master problem \cref{eq:SBMD-subproblem} can be solved efficiently. The low-rank property indeed holds for many SDPs from combinatorial problems and phase retrieval \cite{candes2015phase}. 
\end{remark}

\subsection{Comparison and connections} \label{sec:comparison-connection}
In this subsection, we compare the differences and draw the connections between primal and dual formulations of the spectral bundle method. 

\begin{table*}[t]
\renewcommand{\arraystretch}{1.8}
    \centering
    \caption{Comparison between \SBMP~in \Cref{alg:(r_p-r_c)-SBM-P} and \SBMD~in \Cref{alg:(r_p-r_c)-SBM-D}. $\Wtstar$ denotes the optimizer of the master problem in constraint set $\hat{\mathcal{W}}_t$ in \cref{eq:constraint-W_t} for \SBMP~(or \cref{eq:SBMD-constraint-W_t} for \SBMD).
    The first block of rows (\textit{Comp.}) denotes computational details, and the second block of rows (\textit{Conv.}) shows convergence details.  
    } 
    \label{tb:comparison-SBMP-SBMD}
    \scalebox{0.89}{
    \begin{tabular}{cccc}
    \toprule
   & Descriptions & \SBMP & \SBMD \\
    \hline
   \multirow{7}{*}{\textit{Comp.}} & Domain & $\mathcal{X}_0 = \{X \in \mathbb{S}^n\;|\; \Amap(X) = b\}$ & $\RR^m$\\
   & Reference, candidate & $\currposi \in \mathbb{S}^n,\candidate\in \mathbb{S}^n $ & $\ycurrposi\in \mathbb{R}^m,\ycandidate\in \mathbb{R}^m$ \\
   &  Penalized parameter & $\rho > \sup_{(\ystar, \Zstar) \in \Dstar} \; {\Trace(\Zstar)}$  & $\rho > \sup_{\Xstar \in \Pstar} \; {\Trace(\Xstar)}$ \\
   & Nonsmooth objective & $ \langle C,X \rangle  + \rho \max \{\lambda_{\max} (-X) ,0 \}$ & $-b^\tr y + \rho  \max \left\{\lambda_{\max} \left( \Ajmap(y)   - C\right) ,0 \right\}$ \\
  & Approximation $(\hat{f}_t)$ &  $\langle C,X \rangle  + \max_{W \in \hat{\mathcal{W}}_t } \langle W, -X\rangle$ & $-b^\tr y +  \max_{ W \in \hat{\mathcal{W}}_t}   \left \langle W, \Ajmap(y)-C \right \rangle$ \\
  &  Master problem & $\min_{X \in \mathcal{X}_0 }\;\hat{f}_t(X)  + \frac{\alpha}{2 } \|X - \currposi\|^2$ & $\min_{y \in \RR^m }\;\hat{f}_t(y)  + \frac{\alpha}{2 } \|y - \ycurrposi\|^2$ \\
  &  Current information  & Top $\rcurrent$ eigenvectors of $-\candidate$  & Top $\rcurrent$ eigenvectors of $\Ajmap\left(\ycandidate\right) - C$ \\
    \hline
  \multirow{4}{*}{\textit{Conv.}} & Primal feasibility & $\Amap(\currposi) -b =0, \lambda_{\min}(\currposi)\geq -\epsilon$ & $\| \Amap (\Wtstar) -b\|^2 \leq \epsilon, \Wtstar \succeq 0$ \\ 
  & Dual feasibility & $\left\|\Wtstar+\Ajmap(\ytstar) - C\right\|^2 \leq \epsilon,  \Wtstar\succeq 0$ & $\lambda_{\min}\left(\Ajmap \left( \ycurrposi \right)  - C  \right)\geq -\epsilon$ \\
   & Duality gap & $\lvert \langle C, \currposi \rangle - \langle  b, \ytstar \rangle \rvert \leq \sqrt{\epsilon}$ & 
    $\lvert \langle C, \Wtstar\rangle- \langle b, \ycurrposi \rangle \rvert \leq \sqrt{\epsilon}$ \\
  &  Bound on $\rcurrent$  & $\max_{\Xstar \in \Pstar }\dim(\text{null}(\Xstar))$  & $\max_{(\ystar, \Zstar) \in \Dstar }\dim(\text{null}( \Zstar))$ \\
    \bottomrule 
    \end{tabular}
    }
\end{table*}

It is clear that both \SBMP~in \Cref{alg:(r_p-r_c)-SBM-P} (primal SDPs) and \SBMD~in \Cref{alg:(r_p-r_c)-SBM-D} (dual SDPs) follow the same framework of the generic bundle method (see \Cref{algorithm:bundle}). One main difference is that \SBMP~solves a constrained nonsmooth problem \cref{eq:SDP-primal-penalized} while \SBMD~deals with an unconstrained nonsmooth problem \cref{eq:SDP-dual-penalized}. The unconstrained nonsmooth problem \cref{eq:SDP-dual-penalized} is in the form of eigenvalue minimization which has extensive literature \cite{overton1992large}. \Cref{tb:comparison-SBMP-SBMD} presents a comparison of the computational details and convergence results between \SBMP~and \SBMD. 
In particular, we have the following observations:
\begin{itemize}
    \item 
\SBMP~generates iterates $(\nextposi,\Wtstar,\ytstar)$ that 1) \textit{exactly} satisfy the primal affine constraint $\Amap(\nextposi) = b$ and dual PSD constraint $\lambda_{\min}(\Wtstar) \geq 0$, but 2) \textit{do not exactly} satisfy the primal PSD constraint $\lambda_{\min}(\nextposi) \geq -\epsilon$ and the dual affine constraint $\left\|\Wtstar+\Ajmap(\ytstar) - C\right\|^2 \leq \epsilon$. 
\item \SBMD~outputs iterates $(\Wtstar,\ytstar)$ that 1) \textit{exactly} satisfy the dual affine constraint (by construction) and the primal PSD constraint $(\lambda_{\min}(\Wtstar) \geq 0)$, but 2) \textit{do not exactly }satisfy the primal affine constraint $\| \Amap (\Wtstar) -b\|^2 \leq \epsilon$ and the dual PSD constraint $\lambda_{\min}(C-\Ajmap(\ynextposi))\geq -\epsilon$.
\item Each iteration of \SBMP~and \SBMD~requires solving a small quadratic SDP in the form of \cref{eq:SBM-ms-pb-QP-abstract}. When choosing the same values for $\rpast$ and $\rcurrent$, the subproblems from \SBMP~and \SBMD~have the same dimension, and thus the computational complexity is very similar (although their constructions of \cref{eq:SBM-ms-pb-QP-abstract} are slightly different; see \Cref{appendix:reformulation}).

\end{itemize}

 Both \SBMP~and \SBMD~have very similar sublinear and linear convergence rates, as shown in \Cref{thm: sublinearates-P-K,thm: linear-convergence} and \Cref{thm: sublinearates,thm: linear convergence of Block SBM under the extra condition strict complementarity}, respectively. However, we here highlight a key difference in the condition for linear convergence.   
 When constructing the lower approximation model at each iteration, \SBMP~selects the top $\rcurrent$ eigenvectors of $-X^\star_{t+1}$ to approximate the null space of the primal optimal solutions $\Xstar$, while \SBMD~uses $\rcurrent$ eigenvectors of $\Ajmap(\ynextposi)-C$ to approximate the null space of the dual optimal solutions $\Zstar$. If the null space of the primal optimal solutions in $\Pstar$ has a smaller dimension than that of the dual optimal solutions in $ \Dstar$, i.e.,
\begin{subequations}
    \begin{align} \label{eq:low-rank-solution-dual}
\max_{\Xstar \in \Pstar }\dim(\text{null}(\Xstar)) \ll \max_{ (\ystar, \Zstar) \in \Dstar }\dim(\text{null}( \Zstar)),
    \end{align}
then the master problem in \SBMP~can select a smaller number of eigenvectors, leading to a smaller PSD constraint in \cref{eq:SBM-ms-pb-QP-abstract}. This is reflected in the linear convergence condition on $r_c$; see \cref{eq: rank-condition,eq:rank-condition-dual}. In this case, it will be more numerically beneficial to apply \SBMP~that solves the primal SDPs directly as it is easier to solve the subproblem \cref{eq:SBM-ms-pb-QP-abstract} at each iteration. If, on the other hand, we have the following relationship
 \begin{align} \label{eq:low-rank-solution-primal}
\max_{\Xstar \in \Pstar }\dim(\text{null}(\Xstar)) \gg \max_{ (\ystar,\Zstar) \in \Dstar }\dim(\text{null}( \Zstar)),
    \end{align}
then it will be more numerically beneficial to apply \SBMD~that solves the dual SDPs directly. 
\end{subequations}

\begin{remark}[Low-rank primal solutions versus low-rank dual solutions]
\label{remark:low-rank-discussion}
    
Note that \Cref{eq:low-rank-solution-dual} implies that the dual SDP \cref{eq:SDP-dual} admits low-rank optimal solutions (i.e., $\mathrm{rank}(Z^\star)$ is small), while \Cref{eq:low-rank-solution-primal} indicates that the primal SDP \cref{eq:SDP-primal} has low-rank optimal solutions (i.e., $\mathrm{rank}(X^\star)$ is small). Therefore, it is crucial to choose the appropriate algorithm to solve the problem if we have prior knowledge of the rank property of SDPs. For instance, SDP-based optimization problems from sum-of-square relaxation and their applications \cite{parrilo2003semidefinite,lasserre2009moments} are likely to admit low-rank dual optimal solutions in \cref{eq:SDP-dual}. On the other hand, SDP relaxation of combinatorial problems such as Max-Cut \cite{yurtsever2021scalable}, matrix completion \cite{candes2012exact}, and phase retrieval \cite{candes2015phase} are likely to admit low-rank primal solutions in \cref{eq:SDP-primal}. We provide numerical experiments of SOS optimization and Max-Cut in \Cref{sec:numerical-results}, which indeed validate the effect of the rank property on the convergence behavior of \Cref{alg:(r_p-r_c)-SBM-P,alg:(r_p-r_c)-SBM-D}. We finally note that the low-rank properties also depend on how SDPs are formulated in different applications; see the conversion between primal and dual SDPs in \Cref{appendix:conversion}.  
\end{remark}
    \section{Implementation and numerical experiments} \label{sec:numerical-results}
We have implemented \SBMP~and \SBMD~in \Cref{alg:(r_p-r_c)-SBM-P,alg:(r_p-r_c)-SBM-D} in an open-source MATLAB package, which is available at 
\begin{center}
\url{https://github.com/soc-ucsd/specBM}. 
\end{center}
In this section, we discuss some implementation details and present three sets of numerical experiments to validate the performance of \SBMP~and \SBMD, especially their linear convergence behaviors. The numerical results confirm the discussions in \Cref{sec:comparison-connection}: \SBMP~works better when the dual SDP \cref{eq:SDP-dual} admits low-rank solutions (i.e., \cref{eq:low-rank-solution-dual} holds). Similarly, \SBMD~is more beneficial when the primal SDP \cref{eq:SDP-primal} admits low-rank solutions (i.e., \cref{eq:low-rank-solution-primal} holds). 

In \Cref{sbusec:random-SDPs}, similar to \cite[Section 5.1]{ding2020revisit}, we consider SDPs with randomly generated problem data, which demonstrates that \SBMP~and \SBMD~admit sublinear convergence under different configurations of $(\rpast,\rcurrent)$ and that the algorithms have a linear convergence when the rank conditions in \Cref{eq: rank-condition,eq:rank-condition-dual} hold. 
 In \Cref{subsec:max-cut}, we consider a benchmark combinatorial problem -- Max-Cut and its standard SDP relaxation, which is likely to admit low-rank primal solutions. In this case, it is more desirable to apply
 \SBMD~than \SBMP. In \cref{subsec-quartic-poly}, we show that 
 \SBMP~is well-suited for SDP relaxation rising from sum-of-squares (SOS) optimization, which appears to admit low-rank dual solutions. In this case, a small number of current eigenvectors $\rcurrent$ is sufficient to fulfill the rank condition \Cref{eq: rank-condition} that ensures fast linear convergence. To highlight the efficiency of \SBMP, we further compare it with the state-of-the-art interior-point and first-order SDP solvers: we choose \textsf{SDPT3} \cite{toh1999sdpt3} and \textsf{MOSEK} \cite{mosek} as the interior-point solvers, and we select \textsf{CDCS} \cite{CDCS} and \textsf{SDPNAL+} \cite{yang2015sdpnal} as the first-order solvers.
 
 All numerical experiments are conducted on a PC with a 12-core Intel i7-12700K CPU$@$3.61GHz and 32GB RAM.

\subsection{Implementation details}

One major computation in each iteration of \SBMP~and \SBMD~is to solve the master problems \cref{eq:SBP-equivalence} and \cref{eq:SBD-equivalence}. As discussed in \Cref{remark:master-problem-primal}, we have implemented an automatic transformation from \cref{eq:SBP-equivalence} and \cref{eq:SBD-equivalence}  into standard conic form \cref{eq:SBM-ms-pb-QP-abstract}. Then, a subroutine is required to solve \cref{eq:SBM-ms-pb-QP-abstract}. In our current implementation, we used \textsf{MOSEK} \cite{mosek} to get an exact solution of \cref{eq:SBM-ms-pb-QP-abstract} when $r = \rpast +\rcurrent > 1$ and implemented the analytical solution in \Cref{appendix: r=1} when $\rpast = 0, \rcurrent = 1$. We note that customized algorithms can be developed to \cref{eq:SBM-ms-pb-QP-abstract} with varying accuracy which will further improve numerical efficiency at each iteration.  

To form the lower approximation models \cref{eq:F_W_t_P_t} and \cref{eq:F_d_X_t_V_t}, we computed eigenvalues and eigenvectors using the routine \texttt{eig} in MATLAB. Faster eigenvalue/eigenvector computations can be implemented using \textsf{dsyevx} in \textsf{LAPACK}, similar to \cite[Section 5]{yang2015sdpnal}. 
The orthogonalization process to update the matrix $P_{t+1}$ in \cref{eq:update-P-t} was implemented using the routine \texttt{orth} in MATLAB.

\subsubsection{Adaptive strategy on the regularization parameter.}

\Cref{alg:(r_p-r_c)-SBM-P,alg:(r_p-r_c)-SBM-D} are guaranteed to converge with any regularization parameter $\alpha > 0$ in the subproblems \cref{eq:SBM-subproblem,eq:SBMD-subproblem}. Yet, the value of $\alpha$ largely influences the practical convergence performance, as highlighted in \cite{kiwiel1990proximity} and \cite[Page 380]{ruszczynski2011nonlinear}. In our implementation of \Cref{alg:(r_p-r_c)-SBM-P}, the  parameter $\alpha$ uses the adaptive updating rule below\footnote{The implementation of \Cref{alg:(r_p-r_c)-SBM-D} uses the same updating rule and replaces  $\currposi$, $\candidate$, $F$, and $\hat{F}_{(\Bar{W}_t,P_t)}$ with $\ycurrposi$, $\ycandidate$, $F_\mathrm{d}$ and $\hat{F}_{\mathrm{d},(\Bar{W}_t,P_t)}$  respectively.}:
\begin{align*}
    \alpha = 
    \begin{cases}
    \min\{2\alpha,\alpha_{\max} \},~\text{if }\ml \left(F(\currposi) - \hat{F}_{(\bar{W}_t,P_t)}(\candidate)\right) \geq F(\currposi) -  F(\candidate) ~\text{and}~ N_{\mathrm{c}} \geq N_{\min}\\
         \max\{\alpha/2,\alpha_{\min} \},~\text{if }  \mr \left(F(\currposi) - \hat{F}_{(\bar{W}_t,P_t)}(\candidate)\right) \leq F(\currposi) -  F(\candidate), 
    \end{cases}
\end{align*}
where $\mr > \beta $ and $ 0 < \ml < \beta$ are two nonnegative parameters that indicate the effectiveness of the current approximation model $ \hat{F}_{(\bar{W}_t,P_t)}$ and the candidate point $\candidate$, $\alpha_{\min} $ and $ \alpha_{\max}$ are two nonnegative parameters that keep $\alpha$ staying in the interval $[\alpha_{\min},\alpha_{\max}]$, $N_{\mathrm{c}}$ counts the number of consecutive null steps, and $N_{\min}$ is the threshold that controls the frequency of increasing $\alpha$.

In our  implementation for
 \Cref{alg:(r_p-r_c)-SBM-P,alg:(r_p-r_c)-SBM-D},
the default parameters are chosen as $\ml = 0.001,  
  \alpha_{\min} = 10^{-5},\alpha_{\max} = 100, ~\text{and}~ N_{\min} = 10.$ 
  The parameter $\beta$ and $\mr$ were tuned slightly for different classes of instances in our experiments.
The initial points are chosen as $\Omega_0 = I$ and $\omega_0 = 0$ for
\Cref{alg:(r_p-r_c)-SBM-P} and \Cref{alg:(r_p-r_c)-SBM-D} respectively.

\subsubsection{Suboptimality measures.} 
For the pair of primal and dual SDPs \cref{eq:SDP-primal,eq:SDP-dual}, we measure the feasibility and optimality of a candidate solution $(X,y,Z) \in \mathbb{S}^n \times \mathbb{R}^m \times \mathbb{S}^n$ using 
\begin{subequations}\label{eq:measures-optimality}
    \begin{align}
    \eta_1 &= \frac{\|\mathcal{A}(X) - b\|}{1 + \|b\|}, \qquad \quad\eta_2 = \max\{0,\lambda_{\max}(-X) \}, \label{eq:suboptimality-primal}\\
    \eta_3 &= \frac{\|\mathcal{A}^*(y) + Z - C\|}{1 + \|C\|}, \quad \eta_4 = \max\{0,\lambda_{\max}(-Z) \}, \label{eq:suboptimality-dual}\\
    \eta_5 &= \frac{|\langle C, X \rangle - b^\tr y|}{1 + |\langle C, X \rangle| + |b^\tr y|}.  \label{eq:suboptimality-gap}
    \end{align}
\end{subequations}
In \cref{eq:suboptimality-primal}, $\eta_1$ and $\eta_2$ measure the violation of the affine constraint and the conic constraint in the primal SDP, respectively. The measures $\eta_3$/$\eta_4$ in \cref{eq:suboptimality-dual} quantify the violation of the affine/conic constraints in the dual SDP. The last index $\eta_5$ in \cref{eq:suboptimality-gap} measures the duality gap. 

As stated in \Cref{thm: sublinearates-P-K}, all iterates $\currposi$ (primal variable), $\ytstar$ and $\Wtstar$ (dual variables) from \SBMP~ensure $\eta_1 = 0$ and $\eta_4 = 0$ (up to machine accuracy). On the other hand, the iterates $\Wtstar$ (primal variable) and $\omega_t$ (dual variable) from \SBMD~guarantee $\eta_2 = 0$ and $\eta_3 = 0$ (see \Cref{thm: sublinearates}).
In \Cref{sbusec:random-SDPs,subsec:max-cut}, we ran \SBMP~and \SBMD~for a fixed number of iterations and report the measures $\eta_1, \ldots, \eta_5$ for the final iterate.  
In \Cref{subsec-quartic-poly}, for a given tolerance $\textsf{tol} \geq 0$, we terminate the algorithms when 
$$\max\{\eta_1, \ldots, \eta_5\} \leq \textsf{tol}.$$ The performance of \SBMP~was compared with the baseline solvers \textsf{SDPT3} \cite{toh1999sdpt3}, \textsf{MOSEK} \cite{mosek}, \textsf{CDCS} \cite{CDCS} and \textsf{SDPNAL+} \cite{yang2015sdpnal} in our last experiment. In \Cref{sbusec:random-SDPs,subsec:max-cut}, $f(\omega_t)$ denotes the value of the generic cost function \cref{eq:non-smooth-problem}, and it refers to the cost function in \cref{eq:SDP-primal-penalized} for primal SDPs and in \cref{eq:SDP-dual-penalized} for dual SDPs.   

\subsection{SDPs with randomly generated problem data}
\label{sbusec:random-SDPs}

Our first experiment demonstrates the (sub)linear convergence of \SBMP~and \SBMD\; under different configurations of $(\rpast,\rcurrent)$. Similar to \cite[Section 5.1]{ding2020revisit}, we randomly generated two SDPs, satisfying strict complementarity, in the form of \cref{eq:SDP-primal} and \cref{eq:SDP-dual}. Both SDP instances have a PSD constraint of dimension $n = 1000$ and affine constraints of size $m = 200$.  
 The first SDP admits a low-rank dual solution ($\text{rank}(\Zstar) = 3$) and the second SDP admits a low-rank primal solution ($\text{rank}(\Xstar) = 3$). Details of generating these SDP instances are discussed in \Cref{subsec:data-generation}. 

We consider two different configurations of the parameters $\rpast$ and $\rcurrent$: 
\begin{itemize}
    \item $\rpast = 0$, while changing $\rcurrent = r^\star -1,~\rcurrent = r^\star,~\text{and } \rcurrent = r^\star+1,$ 
    \item $\rcurrent = 1$, while changing $\rpast = r^\star-2,~\rpast = r^\star-1,~\text{and } \rpast = r^\star,$
\end{itemize}
where we set $r^\star = 3$ for both SDP instances (thus we have  $r^\star = \text{rank}(\Zstar)$ or $r^\star = \text{rank}(\Xstar)$). 
In the first setting, we do not consider any past information but only rely on the current information $\rcurrent$, while in the second setting, we keep the minimum amount of current information $\rcurrent$ and rely on the accumulated past information $\rpast$. 
As the discussion in \Cref{sec:comparison-connection}, \SBMP~and \SBMD~iteratively approximate two different spaces: the null space of the primal solutions $\Xstar$ and the null space of the dual solutions $\Zstar$, respectively. When an SDP admits low-rank dual solutions (i.e., the null space of the primal solutions $X^\star$ has a low dimension), it is computationally more beneficial to choose \SBMP~to solve the SDP. On the other hand, when an SDP admits low-rank primal solutions (i.e., the null space of the dual solutions $Z^\star$ has a low dimension), it is computationally more beneficial to choose \SBMD. 

We use the first SDP instance with a low-rank dual solution to demonstrate the benefits of \SBMP~and its fast convergence guarantees in \Cref{thm: linear-convergence}. For this SDP instance, we ran \SBMP~for both settings and \SBMD~for only the first setting. Then, we use the second SDP instance with a low-rank primal solution to highlight the benefits of \SBMD~and validate its fast convergence guarantees in \Cref{thm: linear convergence of Block SBM under the extra condition strict complementarity}. 
For the second SDP instance with a low-rank primal solution, we ran \SBMD~for both settings and \SBMP~for only the first setting. The penalty parameter $\rho$ is set $2\Trace(\Zstar)+2$ and $2\Trace(\Xstar)+2$ for \SBMP~and \SBMD~respectively. The step-size parameters $\beta$ and $\mr$ are chosen as $0.4$ and $0.7$ respectively.

\begin{figure}[t]
     \centering
     \begin{subfigure}[b]{0.49\textwidth}
         \centering
         \includegraphics[width=0.85\textwidth]{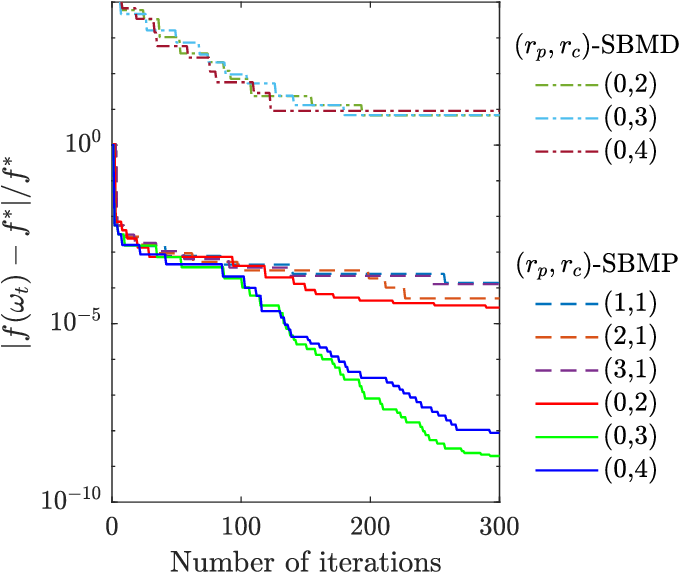}
         \caption{An SDP instance with $\mathrm{rank}(\Zstar)=3$}
     \end{subfigure}
     \hfill
     \begin{subfigure}[b]{0.49\textwidth}
         \centering
         \includegraphics[width=0.85\textwidth]{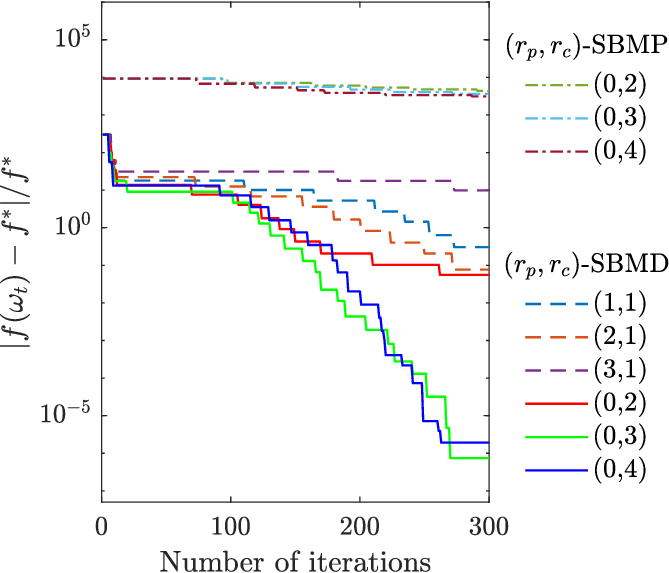}
         \caption{An SDP instance with $\mathrm{rank}(\Xstar)=3$}
     \end{subfigure}
    \caption{The relative optimality gap of different choices of $r_{\mathrm{p}}$ and $r_{\mathrm{c}}$ in \SBMP~and \SBMD~for solving two SDPs with $X \in \mathbb{S}^{1000}_+$: the left SDP instance admits a low-rank dual solution $\mathrm{rank}(Z^\star) = 3$ which suits well for~\SBMP, while the right SDP instance admits a low-rank primal solution $\mathrm{rank}(X^\star) = 3$ that suits well for~\SBMD. }
    \label{figure:experiment-randomSDPs}
\end{figure}

\begin{table}[t]
  \begin{center}
    \caption{Computational results of \SBMP~and\SBMD~for solving two random SDPs ($n=1000$ and $m = 200$) under different values of $\rpast$ and $\rcurrent$. ``Semi Feasi.'' denotes the violation of PSD constraints $\max\{\eta_2,\eta_4\}$,  ``Affine Feasi.'' denotes the violation of affine constraints $\max\{\eta_1,\eta_3\}$, ``Dual Gap'' denotes the duality gap $\eta_5$ in \Cref{eq:measures-optimality}, and ``Cost Opt.'' denotes the cost-value gap $\etac = {(f(\omega_t)-f^\star)}/{f^\star}$.}
    \label{tb:experiment-randomSDPs}
   { \small 
    \begin{tabular}{l*{7}{l}}
    \toprule
    SDP instance &  Algorithm & $(r_{\mathrm{p}}, r_{\mathrm{c}})$& Semi Feasi. & Affine Feasi. & Dual Gap & Cost Opt. \\
    \midrule
    \multirow{9}{*}{
            $
            \begin{aligned}
                & \text{Instance 1} \\
                & \text{rank}(\Zstar) = 3 
            \end{aligned}
           $
        } 
       &  \multirow{3}{*}{
            \SBMD
        } & \textsf{$(0,2)$} & $-7.53\ee{-1}$ & $1.55\ee{-4}$ & $4.28\ee{-5}$ & $13.1$
    \\ & &\textsf{$(0,3)$} & $-7.59\ee{-1}$ & $1.54\ee{-4}$ & $3.62\ee{-4}$ & $13.1$
    \\ & &\textsf{$(0,4)$} & $-9.96\ee{-1}$ & $3.31\ee{-4}$ & $1.10\ee{-4}$ & $28.6$
    \\\cline{2-7}  &  \multirow{6}{*}{
            \SBMP
        } & \textsf{$(1,1)$} & $-3.67\ee{-3}$ & $3.27\ee{-4}$ & $1.24\ee{-5}$ & $1.38\ee{-4}$
    \\ & &\textsf{$(2,1)$} & $-2.92\ee{-3}$ & $3.18\ee{-5}$ & $1.03\ee{-5}$ & $1.02\ee{-4}$
    \\ & &\textsf{$(3,1)$} & $-7.22\ee{-3}$ & $9.00\ee{-6}$ & $9.74\ee{-7}$ & $2.17\ee{-4}$
    \\ & &\textsf{$(0,2)$} & $-3.58\ee{-4}$ & $2.46\ee{-4}$ & $5.06\ee{-5}$ & $3.22\ee{-5}$
    \\ & &\textsf{$(0,3)$} & $-1.17\ee{-8}$ & $1.57\ee{-6}$ & $3.38\ee{-7}$ & $2.14\ee{-9}$
    \\ & &\textsf{$(0,4)$} & $-1.40\ee{-7}$ & $3.22\ee{-6}$ & $6.91\ee{-7}$ & $1.05\ee{-8}$
    \\ 
    \hline
    \multirow{9}{*}{
            $
            \begin{aligned}
                & \text{Instance 2} \\
                & \text{rank}(\Xstar) = 3
            \end{aligned}
           $
        } 
         &  \multirow{3}{*}{
            \SBMP
        } & \textsf{$(0,2)$} & $-144$  & $6.82\ee{-1}$ & $9.99\ee{-1}$  &  $4801$ 
    \\ & &\textsf{$(0,3)$} & $-121$  & $5.61\ee{-1}$ & $9.99\ee{-1}$ &  $4041$
    \\ & &\textsf{$(0,4)$} & $-106$  & $4.58\ee{-1}$ & $9.99\ee{-1}$  & $3369$
    \\\cline{2-7}  &  \multirow{6}{*}{
            \SBMD
        } & \textsf{$(1,1)$} & $-1.80$ &  $3.11\ee{-2}$ &   $8.21\ee{-3}$  &  $6.34\ee{-1}$
    \\ & &\textsf{$(2,1)$} & $-5.11\ee{-1}$ & $2.00\ee{-2}$ &  $4.78\ee{-3}$  &  $2.06\ee{-1}$
    \\ & &\textsf{$(3,1)$} & $-4.95\ee{-1}$ & $4.25\ee{-3}$ &  $4.62\ee{-3}$ &  $17.7$
    \\ & &\textsf{$(0,2)$} & $-3.04\ee{-1}$ & $5.47\ee{-3}$ &  $6.74\ee{-4}$ & $1.03\ee{-1}$
    \\ & &\textsf{$(0,3)$} & $-4.66\ee{-6}$ & $6.65\ee{-5}$ &  $4.91\ee{-5}$ & $4.74\ee{-6}$
    \\ & &\textsf{$(0,4)$} & $-9.30\ee{-6}$ & $6.65\ee{-5}$ &  $1.32\ee{-5}$ & $3.94\ee{-6}$
    \\ 
    \toprule
    \end{tabular}
    }
  \end{center}
\end{table}

In all cases, we ran \SBMP~and \SBMD~for 300 iterations. In this experiment, we also computed the cost value gap as 
$
\etac = {(f(\omega_t)-f^\star)}/{f^\star}. 
$
The convergence behaviors of the cost value gap are illustrated in \Cref{figure:experiment-randomSDPs}. In \Cref{tb:experiment-randomSDPs}, we list the suboptimality measures for the final iterates, where ``Semi Feasi.'' denotes the violation of PSD constraints $\max\{\eta_2,\eta_4\}$,  ``Affine Feasi.'' denotes the violation of affine constraints $\max\{\eta_1,\eta_3\}$, ``Dual Gap'' denotes the duality gap $\eta_5$ in \Cref{eq:measures-optimality}, and ``Cost Opt.'' denotes the cost-value gap $\etac$.  As expected, the value of $\rcurrent$ greatly affects the convergence performance for both \SBMP~and~\SBMD. 
In the first SDP with a low-rank dual solution, we observe that \SBMP~has a fast convergence behavior when choosing $\rcurrent \geq 3 = \text{dim}(\text{null}(\Xstar))$, while \SBMD~has a slow performance in all settings. On the other hand, in the second SDP with a low-rank primal solution, \SBMD~enjoys the fast convergence when $\rcurrent \geq 3 = \text{dim}(\text{null}(\Zstar))$, while \SBMP~converges poorly in all different configurations. The numerical results confirm the theoretical convergence results in \Cref{thm: linear-convergence,thm: linear convergence of Block SBM under the extra condition strict complementarity} and our discussions in \Cref{sec:comparison-connection}. 
\\
\subsection{Max-Cut}\label{subsec:max-cut}
\begin{figure}[t]
     \centering
     \begin{subfigure}[b]{0.49\textwidth}
         \centering
         \includegraphics[width=0.85\textwidth]{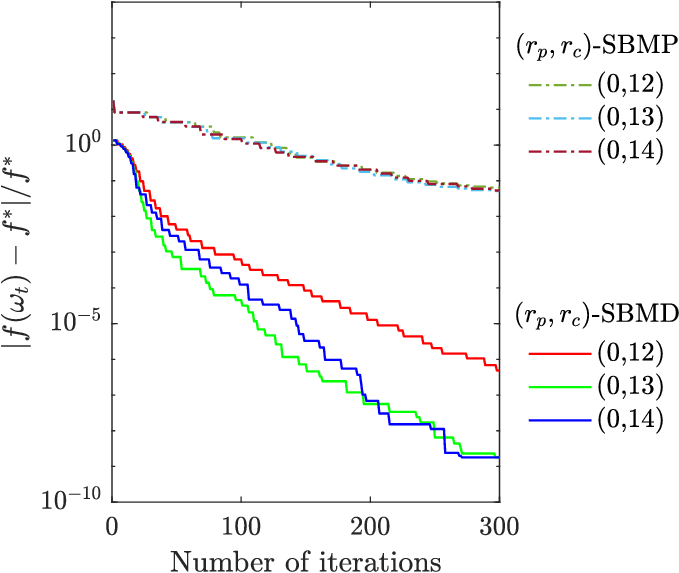}
         \caption{Max-Cut problem with Graph \textsf{G1}}
     \end{subfigure}
     \hfill
     \begin{subfigure}[b]{0.49\textwidth}
         \centering
         \includegraphics[width=0.85\textwidth]{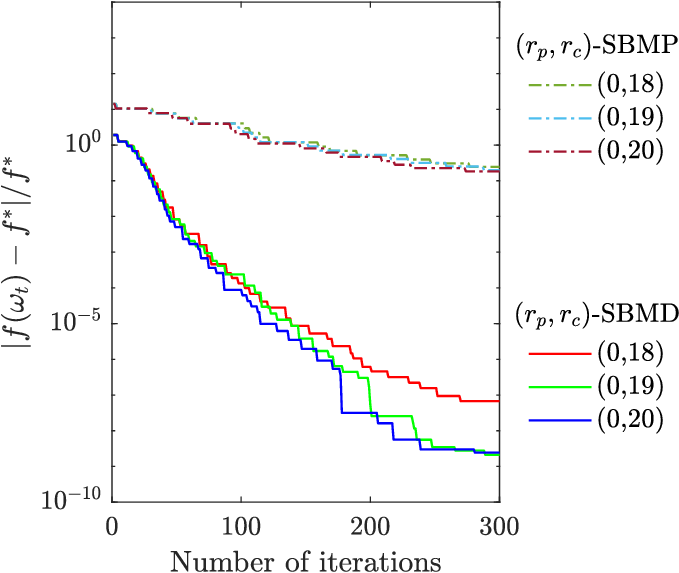}
         \caption{Max-Cut problem with Graph \textsf{G25}}
     \end{subfigure}
    \caption{The relative optimality gap of different choices of $r_{\mathrm{p}}$ and $r_{\mathrm{c}}$ in \SBMP~and \SBMD~for solving SDP relaxations of two Max-Cut problems.}
    \label{figure:experiment-maxcut}
\end{figure}

In this experiment, we consider the maximum cut problem, which is a benchmark combinatorial optimization problem. The SDP relaxation is likely to have low-rank primal solutions, for which \SBMD~is better suited. Consider an undirected graph $\mathcal{G}(\mathcal{V},\mathcal{E})$ defined by a set of vertices $\mathcal{V} = \{1,2,\ldots,n\}$ and a set of edges $\mathcal{E} \subseteq \mathcal{V} \times \mathcal{V}$, and each edge  $\{i,j\} \in \mathcal{E}$ has a weight $w_{ij} = w_{ji}$. The max-cut problem aims to find a maximum cut that separates the vertices into two different groups. This can be formulated as a binary quadratic program \cite{goemans1995improved} 
\begin{equation} \label{eq:max-cut}
         \min_{x_i^2 = 1, i = 1, \ldots, n} \;\; \frac{1}{4} x^\tr L x,  \\
\end{equation}
where $L \in \mathbb{S}^n$ is the Laplacian matrix of $\mathcal{G}(\mathcal{V},\mathcal{E})$ defined as $ L_{ij} = \sum_{j \neq i } w_{ij}$ if $i = j$, and $L_{ij} = -w_{ij}$ otherwise. 
A well-known semidefinite relaxation \cite{goemans1995improved} for the Max-Cut problem \cref{eq:max-cut} is   
\begin{equation}
    \begin{aligned}
       \min_{X} \quad & \frac{1}{4}  \langle L,X \rangle  \\
        \mathrm{subject~to} \quad & X_{ii} = 1, \ldots, n,   \\
        & X \in \mathbb{S}^n_+.
        \label{eq:Maxcut-sdp}
    \end{aligned}
\end{equation}
If the optimal solution of SDP relaxation \Cref{eq:Maxcut-sdp} satisfies $\mathrm{rank}(X^\star) = 1$, then the SDP relaxation \Cref{eq:Maxcut-sdp} is exact and one can recover a globally optimal solution to \Cref{eq:max-cut}. However, the rank-one solution may not exist. Instead, many max-cut instances admit low-rank optimal solutions: $1 < \mathrm{rank}(X^\star) \ll n$, as observed in \cite[Section 1.3.2]{yurtsever2021scalable}, \cite[Section 5.1]{ding2020revisit}. For these SDP instances, we expect that \SBMD~exhibits faster linear convergence when choosing a small value of  $\rcurrent$, while \SBMP~only has slower sublinear convergence for the same choice of $\rcurrent$.  

We run both \SBMP~and \SBMD~for a fixed number of 300 iterations for two Max-Cut instances \textsf{G1} and \textsf{G25} from \cite{davis2011university}. There are 800 nodes in graph \textsf{G1}, and 2000 nodes in graph \textsf{G25}, and so the PSD constraints in \cref{eq:Maxcut-sdp} have a dimension of 800 and 2000, respectively. Despite the large value of $n$, we observed a low rank primal solution $\mathrm{rank}(X^\star) = 12$ for \textsf{G1} and $\mathrm{rank}(X^\star) = 19$ for \textsf{G25}.
Note that the SDP \Cref{eq:Maxcut-sdp} has a constant trace property $\Trace(X) = n$. Thus, we chose the penalty parameter $\rho = 2n+2$ for \SBMD.~We can also estimate the penalty parameter for \SBMP~as\footnote{This estimate is due to the structure of Max-Cut problems \cref{eq:Maxcut-sdp} that $\Trace\left(\sum_{i=1}^m A_i y_i \right)= b^\tr y, \forall y \in \mathbb{R}^m$, and the fact that $\hat{y} = \left(\frac{n}{4}\min\{0,\lambda_{\min}(L)\}\right)\mathbf{1}$ is a dual feasible solution, where $\mathbf{1} \in \mathbb{R}^m$ is an all one vector. Hence, we have $\Trace(\Zstar) = \frac{1}{4}\Trace(L) - \sum_{i=1}^m \Trace(A_i) \ystar_i = \frac{1}{4}\Trace(L) - b^\tr \ystar \leq \frac{1}{4}\Trace(L) - b^\tr \hat{y}, \forall (\ystar, \Zstar) \in \Dstar.$} $\rho = 2\left(\frac{1}{4}\Trace(L)-\frac{n}{4}\min\{0,\lambda_{\min}(L)\}\right)+2$. We set $r_{\mathrm{p}} = 0$ since this parameter has little impact on convergence in this case. The other parameters were chosen as those in \Cref{sbusec:random-SDPs}.  

The numerical results are illustrated in \Cref{figure:experiment-maxcut} and \Cref{tb:experiment-MaxCut}. In all cases, compared with \SBMP, \SBMD~returns solutions with much higher accuracy within the same number of iterations. In particular, since the rank condition \Cref{eq:rank-condition-dual} is expected to hold, \SBMD~shows a faster linear convergence rate, while \SBMP~converges to an optimal solution slowly. Again, this is consistent with the theoretical expectation in \Cref{sec:comparison-connection}. 

\begin{table}[h]
  \begin{center}
    \caption{Computational results of \SBMP~and \SBMD~for solving max-cut SDPs \cref{eq:Maxcut-sdp} under different values of $\rpast$ and $\rcurrent$. ``Semi Feasi.'' denotes the violation of PSD constraints $\max\{\eta_2,\eta_4\}$,  ``Affine Feasi.'' denotes the violation of affine constraints $\max\{\eta_1,\eta_3\}$, ``Dual Gap'' denotes the duality gap $\eta_5$ in \Cref{eq:measures-optimality}, and ``Cost Opt.''  denotes the cost-value gap $\etac = {(f(\omega_t)-f^\star)}/{f^\star}$.}
    \label{tb:experiment-MaxCut}
    {\small
    \begin{tabular}{l*{6}{l}}
    \toprule
    Max-cut instance &   Algorithm & $(r_{\mathrm{p}}, r_{\mathrm{c}})$& Semi Feasi. & Affine Feasi. & Dual Gap & Cost Opt. \\
    \midrule
    \multirow{6}{*}{
            $
            \begin{aligned}
                & \text{G1} \\
                & \text{rank}(\Xstar) = 12 \\
                & n = 800 \\
                & m = 800
            \end{aligned}
           $
        } 
        &  \multirow{3}{*}{
            \SBMP
        }
       & \textsf{$(0,12)$} & $-1.39\ee{-2}$ & $7.28\ee{-2}$ & $1.68\ee{-2}$ & $6.85\ee{-2}$
    \\ & & \textsf{$(0,13)$} & $-1.39\ee{-2}$ & $6.94\ee{-2}$ & $1.77\ee{-2}$ & $5.49\ee{-2}$
    \\ & & \textsf{$(0,14)$} & $-2.36\ee{-2}$ & $6.63\ee{-2}$ & $1.47\ee{-2}$ & $5.97\ee{-2}$\\
           \cline{2-7} &  \multirow{3}{*}{
            \SBMD
        }
     & \textsf{$(0,12)$} & $-3.80\ee{-6}$ & $7.16\ee{-4}$ & $1.23\ee{-5}$ & $6.86\ee{-7}$
    \\ & & \textsf{$(0,13)$} & $-2.52\ee{-9}$ & $2.69\ee{-5}$ & $4.58\ee{-8}$ & $2.31\ee{-9}$
    \\ & & \textsf{$(0,14)$} & $-1.69\ee{-9}$ & $4.42\ee{-5}$ & $1.83\ee{-7}$ & $1.97\ee{-9}$
    \\ 
    \hline
    \multirow{6}{*}{
            $
            \begin{aligned}
                & \text{G25} \\
                & \text{rank}(\Xstar) = 19\\
                & n =2000 \\
                & m = 2000
            \end{aligned}
           $
        } 
    &  \multirow{3}{*}{
            \SBMP
        }
       & \textsf{$(0,18)$} & $-1.27\ee{-1}$ & $4.73\ee{-1}$ & $3.88\ee{-2}$ &  $3.07\ee{-1}$ 
    \\ & & \textsf{$(0,19)$} & $-9.24\ee{-2}$ & $1.18$        & $2.33\ee{-1}$ &  $2.51\ee{-1}$
    \\ & & \textsf{$(0,20)$} & $-8.27\ee{-2}$ & $6.62\ee{-1}$ & $1.30\ee{-1}$  & $2.24\ee{-1}$

    \\ \cline{2-7} & \multirow{3}{*}{
            \SBMD
        }
     & \textsf{$(0,18)$} & $-2.17\ee{-7}$ & $6.34\ee{-4}$ & $3.63\ee{-6}$ & $9.37\ee{-8}$
    \\ & & \textsf{$(0,19)$} & $-1.72\ee{-9}$ & $3.04\ee{-6}$ & $1.77\ee{-8}$ & $2.76\ee{-9}$
    \\ & & \textsf{$(0,20)$} & $-2.83\ee{-9}$ & $2.76\ee{-5}$ & $2.99\ee{-8}$ & $3.01\ee{-9}$
    \\
    \toprule
    \end{tabular}
    }
  \end{center}
\end{table}

\subsection{Quartic polynomial optimization on a sphere}
\label{subsec-quartic-poly}

In our last numerical experiments, we consider SOS relaxations for polynomial optimization, which are likely to admit low-rank dual solutions. We expect that \SBMP~is better suited than \SBMD. Our numerical results further show that \SBMP~outperforms 
a set of baseline solvers, including interior-point solvers \textsf{SDPT3} \cite{toh1999sdpt3}, \textsf{MOSEK} \cite{mosek}, and first-order solvers \textsf{CDCS} \cite{CDCS}, \textsf{SDPNAL+} \cite{yang2015sdpnal}. 

Consider a constrained polynomial optimization problem over a sphere
 $\min_{x \in \mathcal{S}^{n-1}} \; p_0(x),$  
where $p_0: \RR^n \rightarrow \RR$ is a polynomial and $\mathcal{S}^{n-1} = \{ x \in \RR^n \mid \|x\|^2 = 1\} $ is the $n$-dimensional unit sphere. This problem is in general NP-hard, but it can be approximated well using the moment/SOS relaxation \cite{parrilo2003semidefinite,lasserre2009moments}. 
In particular, the $k$th-order SOS relaxation  is  
\begin{equation}
    \begin{aligned}\label{eq:kth-SOS-hierachy}
    \max_{\gamma, \sigma_0, \psi_1} & \quad \gamma \\
     \mathrm{subject~to} & \quad  p_0(x) - \gamma -\psi_1 h(x) = \sigma_0 ,\\
                         & \quad \sigma_0 \in \Sigma[x]_{n,2k}, \;\psi_1 \in \RR[x]_{n,2(k-\lceil h \rceil)},
    \end{aligned}
\end{equation}
where $h(x) = \|x\|^2 -1$, $\lceil h \rceil = \lceil \mathrm{deg}(h)/2 \rceil$, $\RR[x]_{n,2(k-\lceil h \rceil)}$ denotes the real polynomial in $n$ variables and degree at most $2(k-\lceil h \rceil)$, and $\Sigma[x]_{n,2k}$ denote the cone of SOS polynomials in $ \RR[x]_{n,2k}$.  
It is well-known that \cref{eq:kth-SOS-hierachy} can be equivalently reformulated into the standard primal SDP in the form \cref{eq:SDP-primal} with some extra free variables; see e.g., \cite{parrilo2003semidefinite} for details. We observe that these SDPs are likely to admit low-rank dual solutions {(this observation is consistent with the flat extension theory on the moment side \cite[Theorem 3.7]{lasserre2009moments} when it is formulated as a dual SDP)}. 

Motivated by the benchmark problems \cite{zheng2022block}, we consider three instances of \cref{eq:kth-SOS-hierachy} in our numerical experiments:  
\begin{enumerate}
    \item Modified Broyden tridiagonal polynomial 
        \begin{align*}
            q_1(x) = & \left(\left(3-2 x_1\right) x_1-2 x_2+1\right)^2 + \\
        &\sum_{i=2}^{d-1}\left(\left(3-2 x_i\right) x_i-x_{i-1}-2 x_{i+1}+1\right)^2 +\left(\left(3-2 x_d\right) x_d-x_{d-1}+1\right)^2+\left(\sum_{i=1}^d x_i\right)^2.
        \end{align*}
    \item Modified Rosenbrock polynomial
       $  q_2(x) =  1 + \sum_{i=2}^{d} 100 \left( x_i - x_{i-1}^2 \right)^2 + \left( 1 - x_i \right)^2 + \left(\sum_{i=1}^d x_i\right)^2.$
    \item Random quartic polynomial
     $    q_3(x) =  \langle c_{d,4},[x]_{d,4}  \rangle ,$ 
    where $[x]_{d,4} $ is the standard monomial bases with $d$ variables and degree at most $4$, and $c_{d,4}$ is a randomly generated coefficient vector.
\end{enumerate}
We used the package \textsf{SOSTOOLS} \cite{sostools} to recast the SOS relaxation \Cref{eq:kth-SOS-hierachy} with $k = 2$ into a standard SDP in the form of \cref{eq:SDP-primal}.  
We tested \SBMP~for the above three polynomials with different dimensions (the performance of \SBMD~was very poor in our experiments, and we omitted it here). The dimension of the SDP relaxations ranges from $n = 496$ to $861$ and $m= 45,897$ to $134,889$.    
Recall in \cref{subsec:exact-penalty-SDP} that we need to choose the penalty parameter $\rho > \DZstar$ for \SBMP. For SDPs from \cref{eq:kth-SOS-hierachy}, we can show that any $\rho \geq (1+1)^2 = 4$ is a valid exact penalty parameter (see \Cref{lemma:const-trace-POP}) thanks to the unit sphere constraint. We thus chose $\rho =10$ for all cases. The parameters $(\rpast,\rcurrent)$ are chosen as $(0,3),(0,5),$ and $(0,7)$ for three different problems. In our experiments, we ran \SBMP~until it reached tolerance $1\ee{-4}$.  

\begin{table}
\centering
\caption{Computational results on the SOS relaxation \cref{eq:kth-SOS-hierachy} of quartic polynomial optimization on a sphere: $ \ep$ denotes the primal feasibility, $ \ed$ denotes the primal feasibility, and $\eg$ denotes the duality gap. Entries marked by {\sc oom} indicate ``out of memory'' runtime errors.  }
\label{tb:POPs} 
\adjustbox{max width=\textwidth}{
\centering
\small
\scalebox{1}{
    \begin{tabular}{|c|c|c|cccc||c|}
    \hline
     Dimension & Case & Metric & \textsf{SDPT3} \cite{toh1999sdpt3} & \textsf{MOSEK} \cite{mosek} & \textsf{CDCS} \cite{CDCS}  & \textsf{SDPNAL+} \cite{yang2015sdpnal} & \SBMP 
     \\
     \hline
     \hline
     \multirow{15}{*}{\Large$ \substack{d:\ 30 \\ \\  n:\ 496\  \\ \\ m:\ 45,879\ }$ } & \multirow{5}{*}{$q_1$} 
    							& $ \ep$  & \multirow{5}{*}{\textsc{oom}}  & $9.2\ee{-11}$ & $9.9\ee{-5}$  & $2.1\ee{-11}$ & $5.9\ee{-14}$\\
    							& & $\ed$ &  & $2.0\ee{-9}$ & $4.1\ee{-5}$  & $9.5\ee{-5}$ & $9.2\ee{-5}$\\
                                    & & $\eg$ &  & $5.9\ee{-12}$ & $9.0\ee{-6}$ & $1.0\ee{-2}$ & $6.8\ee{-6}$\\ 
                                    & & cost  &  & $-25.074$      & $-25.073$      & $-42.850$     & $-25.074$\\
    							& & time  &  & $887.1$       & $100.9$             & $1209$        & \textbf{17.5}\\
    \cline{2-8}
    							& \multirow{5}{*}{$q_2$} 
    							&   $\ep$ & \multirow{5}{*}{\textsc{oom}}  & $8.2\ee{-10}$ & $1.5\ee{-5}$   & $ 4.7\ee{-5} $ & $5.4 \ee{-14}$\\
    							& & $\ed$ &  & $1.9\ee{-8}$ & $9.3\ee{-5}$ & $8.1\ee{-5}$     & $8.9\ee{-5}$\\
    							& & $\eg$ &  & $2.9\ee{-10}$ & $3.9\ee{-5}$  & $6.7\ee{-3}$    & $5.3\ee{-5}$\\
                                    & & cost  &  & $-32.136$     & $-32.137$ & $-32.137$                          &$-32.135$    \\
    							& & time  &  & $583$      & $45.1$ & $17.3$ & \textbf{12.8} \\
    \cline{2-8}
    							& \multirow{5}{*}{$q_3$} 
    							&   $\ep$ & \multirow{5}{*}{\textsc{oom}}  & $9.1\ee-{9}$& $1.0\ee{-4}$  & $ 3.2\ee{-16}$ & $1.9 \ee{-14}$\\
    							& & $\ed$ & & $2.0\ee{-7}$ & $1.5\ee{-6}$ & $5.2\ee{-4}$  & $8.7\ee{-5}$\\
    							& & $\eg$ & & $4.9\ee{-9}$& $2.4\ee{-6}$  & $3.5\ee{-3}$  & $2.2\ee{-5}$\\
                                    & & cost  & & $2.023$     & $2.023$            & $2.021$      & $2.023$    \\
    							& & time  & & $645.2$      & $55.3$                  & $548.6$         & \textbf{16.9} \\
    \cline{2-8}
    \hline
    \hline
    \multirow{15}{*}{\Large$ \substack{d:\ 35 \\ \\  n:\ 666 \\ \\ m:\ 81,584}$ } & \multirow{5}{*}{$q_1$} 
    							& $\ep$   &\multirow{5}{*}{\textsc{oom}} & \multirow{5}{*}{\textsc{oom}}  & $9.6\ee{-5}$  & $1.1\ee{-15}$ & $9.9\ee{-14}$\\
    							& & $\ed$ &                    &                      & $1.9\ee{-5}$  & $6.8\ee{-5}$  & $9.9\ee{-5}  $\\
                                    & & $\eg$ &                    &                      & $2.1\ee{-5}$  & $2.7\ee{-4}$  & $9.4\ee{-6}  $\\ 
                                    & & cost  &                    &                      & $-30.053$        & $-30.070$     & $-30.050$\\
    							& & time  &                    &                      & $289.6$             & $3449.4$      & \textbf{41.4}\\
    \cline{2-8}
    							& \multirow{5}{*}{$q_2$} 
    							&   $\ep$ &\multirow{5}{*}{\textsc{oom}} & \multirow{5}{*}{\textsc{oom}}  & $2.2\ee{-5}$  & $ 4.2\ee{-5}$ & $1.3\ee{-13}$\\
    							& & $\ed$ &                    &  & $9.4\ee{-5}$  & $9.5\ee{-5}$    & $8.8\ee{-5}$\\
    							& & $\eg$ &                    &  & $7.7\ee{-5}$  & $2.4\ee{-3}$   & $2.0\ee{-5}$\\
                                    & & cost  &                    &  & $-37.100$  & $-37.221$      & $-37.091$  \\
    							& & time  &                    &  & $112.3$ & $100.5$ & \textbf{42.7}\\
    \cline{2-8}
    							& \multirow{5}{*}{$q_3$} 
    							&   $\ep$ &\multirow{5}{*}{\textsc{oom}} & \multirow{5}{*}{\textsc{oom}} & $1.0\ee{-4}$  & $ 9.3\ee{-13}$ & $1.9\ee{-14}$\\
    							& & $\ed$ &                    &  & $6.9\ee{-5}$  & $6.5\ee{-7}$  & $9.4\ee{-5}$\\
    							& & $\eg$ &                    &  & $3.5\ee{-5}$ & $6.2\ee{-3}$  & $2.5\ee{-6}$\\
                                    & & cost  &                    &  & $2.121$         & $2.039$      & $2.121$  \\
    							& & time  &                    &  & $170.1$           & $1208.4$         & \textbf{44.5}\\
    \cline{2-8}
    \hline
    \hline
    \multirow{15}{*}{\Large$ \substack{d:\ 40 \\ \\  n:\ 861 \\ \\ m:\ 134,889}$ } & \multirow{5}{*}{$q_1$} 
    							&   $\ep$ & \multirow{5}{*}{\textsc{oom}} & \multirow{5}{*}{\textsc{oom}} & $1.0\ee{-4}$ & $3.6\ee{-12}$ & $ 7.4\ee{-14}$\\
    							& & $\ed$ &                     &                     & $3.2\ee{-6}$ &  $1.0\ee{-4}$& $9.6\ee{-5} $\\
    							& & $\eg$ &                     &                     & $1.4\ee{-5}$ & $4.9\ee{-2}$ & $5.2\ee{-6}$\\
                                    & & cost  &                     &                     & $-35.037$    & $-47.792$ & $-35.034$ \\
    							& & time  &                     &                     & $738.9$      & $1881.1$ & \textbf{84.8}\\
    \cline{2-8}
    							& \multirow{5}{*}{$q_2$} 
    							&   $\ep$ & \multirow{5}{*}{\textsc{oom}} &\multirow{5}{*}{\textsc{oom}} & $1.7\ee{-5}$ & $3.3\ee{-14}$ & $1.3\ee{-13}$\\
    							& & $\ed$ &                     &                    & $9.3\ee{-5}$ & $9.9\ee{-5} $ & $5.5\ee{-5} $\\
    							& & $\eg$ &                     &                    & $5.8\ee{-5}$ & $2.8\ee{-5}$ & $1.6\ee{-5}$\\
                                    & & cost  &                     &                    & $-42.041$    & $-41.997$ & $-42.050$\\
    							& & time  &                     &                    & $254.9$      & $187.6$ & \textbf{82.1}\\
    \cline{2-8}
    							& \multirow{5}{*}{$q_3$} 
    							&   $\ep$ & \multirow{5}{*}{\textsc{oom}} &\multirow{5}{*}{\textsc{oom}} & $1.0\ee{-4}$ & $9.1\ee{-16}$ & $2.0\ee{-14}$\\
    							& & $\ed$ &                     &                    & $1.6\ee{-6}$ & $7.5\ee{-5} $ & $8.2\ee{-5} $\\
    							& & $\eg$ &                     &                    & $4.9\ee{-5}$ & $2.4\ee{-3}$ & $1.0\ee{-4}$\\
                                    & & cost  &                     &                    & $2.681$      & $2.695$ & $2.681$\\
    							& & time  &                     &                    & $144.5$      & $1339.6$ & \textbf{93.1}\\
    \cline{2-8}
    \hline
    \end{tabular}
    }
    }
    \end{table}

To further demonstrate the performance of \SBMP, we compare it with \textsf{SDPT3} \cite{toh1999sdpt3}, \textsf{MOSEK} \cite{mosek}, \textsf{CDCS} \cite{CDCS}, \textsf{SDPNAL+} \cite{yang2015sdpnal}. For \textsf{SDPT3} and \textsf{MOSEK}, we used their default parameters. For \textsf{CDCS}, we use the \textsf{sos} solver with a maximum 10,000 iterations with tolerance $1\ee{-4}$, and this \textsf{sos} option is customized for solving SDPs from SOS relaxations (by exploiting a property called \textit{partial orthogonality} \cite{zheng2018fast}). For \textsf{SDPNAL+}, we used $1\ee{-4}$ as the tolerance, turn off their stagnation detection, and run it using their default parameter with a maximum of 20,000 iterations and a maximum of 10,000 seconds runtime. 

The computational results are listed in \cref{tb:POPs}. To be consistent with other solvers, we report the cost value, time consumption, primal feasibility, dual feasibility, and duality gap (see \cref{eq:measures-optimality}) of the final outcome, i.e., 
$$ 
\ep = \eta_1, \; \ed = \eta_3, \; \eg = \eta_5. 
$$  
As we can see in \cref{tb:POPs}, our algorithm \SBMP~solves all SDP instances to the desired accuracy within a reasonable time, and it consistently outperforms the baseline solvers. For the interior-point solvers, $\sf{SDPT3}$ runs out of memory in all cases on our computer. $\sf{MOSEK}$ was able to return solutions of high accuracy for medium-size problems ($d=30$) but took more time consumption. $\sf{MOSEK}$ also encountered memory issues for larger instances ($d\geq 35$).  
For the first-order solvers, \textsf{CDCS} solved all tested problems with medium accuracy, but the runtime was worse than \SBMP~(indeed, our algorithm~\SBMP~was one order of magnitude faster than \textsf{CDCS} in some cases); 
the solver \textsf{SDPNAL+} solved all SDPs to the desired accuracy for the measures $\ep$ and $\ed$, while the duality gap $\eg$ remains unsatisfactory. We note that the design of \textsf{SDPNAL+} does not consider the duality gap $\eg$ as a stopping criterion. This partially explains the poor performance of the duality gap $\eg$ in the final iterate.

    \section{Conclusion}
\label{sec:Conclusion}

In this paper, we have presented an overview and comparison of spectral bundle methods for solving primal and dual SDPs. All the existing results focus on solving dual SDPs. We have established a family of spectral bundle methods for solving primal SDPs directly. The algorithm developments mirror the elegant duality between primal and dual SDPs. We have presented the sublinear convergence rates for this family of spectral bundle methods and shown that the algorithm enjoys linear convergence with proper parameter choice and low-rank dual solutions. The convergence behaviors and computational complexity of spectral bundle methods for both primal and dual SDPs are in general similar, but they have different features. It is clear that the existing spectral bundle methods are well-suited for SDPs with \textit{low-rank primal solutions}, and our new spectral bundle method works well for SDPs with \textit{low-rank dual solutions}. These theoretical findings are supported by a range of large-scale numerical experiments. We have further demonstrated that our new spectral bundle method achieves state-of-the-art efficiency and scalability when solving the SDP relaxations from polynomial optimization.  

Potential future directions include incorporating other types of constraints (such as nonnegative, second-order cone constraints, etc.), considering second-order information \cite{helmberg2014spectral} for the lower approximation, and analyzing the algorithm performance when the subproblem \cref{eq:SBM-subproblem} is solved inexactly \cite{rockafellar1976augmented}.  Finally, we remark that our current prototype implementation shows promising numerical performance, and it would also be very interesting to further develop reliable and efficient open-source implementations of these spectral bundle methods.

    \bibliographystyle{unsrt}
    \bibliography{reference}

    \newpage
    \appendix
    \numberwithin{equation}{section}
\noindent\textbf{\Large Appendix}
\vspace{5mm}

\small 

\sectionfont{\large}
\subsectionfont{\normalsize}
The appendix is divided into five parts:
\begin{enumerate}
    \item \Cref{Appendix:assumptions-CBM} extends the convergence results of generic bundle methods for unconstrained optimization in \cite{diaz2023optimal} to constrained convex optimization. We present the adaptions required to prove \Cref{lemma-iterations-bound};  
    \item \Cref{Appendix:proofs-in-penalty} presents some technical proofs in \Cref{sec:Penalized-nonsmooth-formulations-for-SDPs}, i.e., the exact penalization for primal and dual SDPs;
    \item \Cref{appendix:reformulation} presents some computation details in \SBMP~and \SBMD; 
    \item \Cref{Appendix:technical-proofs} completes the technical proofs for the convergence guarantees of \SBMP, i.e., \Cref{thm: sublinearates-P-K,thm: linear-convergence};
    \item \Cref{Appendix:data-generation-SDP-SOS} presents further details of our numerical experiments in \Cref{sec:numerical-results}, including the generation of random SDPs and the exact penalty parameter in SOS optimizations on a sphere.
\end{enumerate}

\section{Bundle methods for constrained convex optimization}
\label{Appendix:assumptions-CBM}

In this section, we show that the three conditions \Cref{eq:bundle-method-property-1,eq:bundle-method-property-2,eq:bundle-method-property-3} ensure the convergence of the bundle method in \Cref{alg:bundle-method} to solve the constrained convex optimization problem \Cref{eq:non-smooth-problem}. We restate \Cref{eq:non-smooth-problem} below for convenience,
\begin{equation} \label{eq:non-smooth-problem-appendix}
   f^\star =  \min_{x \in \mathcal{X}_0} \quad f(x),
\end{equation}
where $f: \mathbb{R}^n \rightarrow \mathbb{R} $ is convex but not necessarily differentiable and $\mathcal{X}_0 \subseteq  \mathbb{R}^n$ is a closed convex set. 

The analysis in \cite[Theorem 2.1 and 2.3]{diaz2023optimal} focuses on unconstrained convex optimization with $\mathcal{X}_0 = \mathbb{R}^n$ in \cref{eq:non-smooth-problem-appendix}. We here clarify the minor extension of their analysis to the constrained case with a closed convex set $\mathcal{X}_0 $. We present the adaptions to prove \Cref{lemma-iterations-bound} in the main text. Specifically, we only need to establish the constrained versions of \cite[Lemma 5.1, 5.2, and 5.3]{diaz2023optimal}, which are given in \Cref{lem:descent-prox-gap,lem:null-prox-gap,lem:prox-gap-bound} below. Then, the proof of \Cref{lemma-iterations-bound} follows the same arguments in \cite[Section 5]{diaz2023optimal}. Following the notations in \cite{diaz2023optimal}, at iteration $k$, we define the \textit{proximal gap} by 
$$ \Delta_k := f(\omega_k) - \left(f(\bar{x}_{k+1}) + \frac{\alpha}{2}\|\bar{x}_{k+1} - \omega_k\|^2\right),$$
where $\bar{x}_{k+1} = \argmin_{x\in \mathcal{X}_0 } \left\{f(x) + \frac{\alpha}{2}\|x-\omega_k\|^2\right\}$ and $\omega_k$ is the reference point. Note that $\bar{x}_{k+1}$ is obtained using the original $f$ rather than the lower approximation model $\hat{f}_{k}$.
\begin{lemma}[\bf Descent steps attain decrease proportionally to the proximal gap]\label{lem:descent-prox-gap}
A descent step at iteration $k$ satisfies
\begin{equation} \nonumber
  f(\omega_{k+1}) \leq f(\omega_k) - \beta\Delta_k.
\end{equation}
\end{lemma}
\begin{proof}
    The proof is the same as in \cite[Lemma 5.1]{diaz2023optimal}. We omit the details.
\end{proof}
\begin{lemma}[\bf The number of consecutive null steps is bounded by the proximal gap]\label{lem:null-prox-gap}
    A descent step, at iteration $k$, followed by $T$ consecutive null steps has at most
    \begin{equation}\nonumber
      T\leq \frac{8G_{k+1}^2}{(1-\beta)^2\alpha\Delta_{k+T}}, 
    \end{equation}
    where $G_{k+1}=\sup \{\|g_{t+1}\| \mid k\leq t \leq k+T\}$. If $f$ is $M$-Lipschitz, the condition simplifies to
    \begin{align*}
         T \leq 
      \dfrac{8M^2}{(1-\beta)^2\alpha\Delta_{k+T}}.
    \end{align*} 
\end{lemma}
\begin{proof}
 The proof is essentially the same that in \cite[Lemma 5.2]{diaz2023optimal}. 
 Consider some descent step $k$, followed by $T$ consecutive null step. Define the proximal subproblem gap at $k < t \leq k+T$ by 
    \begin{align*}
        \widetilde\Delta_t := f(\omega_{k+1}) - \left(\hat{f}_t( \xcandidate ) + \frac{\alpha}{2}\| \xcandidate - \omega_{k+1}\|^2\right).
    \end{align*}
    Note that every null step has the same reference point $\omega_{k+1}$. The core of this proof is to show the inequality
 \begin{align}
       \widetilde\Delta_{t+1} &\leq \widetilde\Delta_{t} -  \frac{(1-\beta)^2\alpha \widetilde\Delta_t^2}{8G_{k+1}^2}. \label{eq:condition}
    \end{align}
    Before proving this inequality, let us show how it completes the proof first. After $T$ consecutive null steps, the lower bound $\hat{f}_{k+T} (\cdot) \leq f(\cdot)$ in \Cref{eq:bundle-method-property-1}  ensures that $\widetilde\Delta_{k+T} \geq \Delta_{k+T}$. Thus, to bound $T$, it is sufficient to show that the reversed inequality $\widetilde\Delta_{k+T} < \Delta_{k+T}$ holds. Indeed, upon applying the result in \cite[Lemma A.1.]{diaz2023optimal} that bounds the number of steps for a recursive relation and setting the target accuracy $\epsilon = \Delta_{k+T}$ (note that $\Delta_{k+T} = \Delta_{k+1}$), we conclude the number of consecutive null steps $T$ is at most 
    \begin{align*}
        T\leq \frac{8G_{k+1}^2}{(1-\beta)^2\alpha\Delta_{k+T}}. 
    \end{align*}
    
    Now, let us focus on the derivation of \Cref{eq:condition}. Consider some null step $k < t \leq k+T$. We define the necessary lower bound given by \Cref{eq:bundle-method-property-2,eq:bundle-method-property-3} as 
    \begin{align*}
        \tilde f_{t+1}(x) := \max\left\{\hat{f}_{t}(\xcandidate) + \langle s_{t+1}, x-\xcandidate\rangle,\ f(\xcandidate) + \langle g_{t+1}, x-\xcandidate\rangle\right\} \leq \hat{f}_{t+1}(x), \quad \forall x \in \mathcal{X}_0,
    \end{align*}
and the solution of a proximal step made by $\tilde f_{t+1}(x)$ as $y_{t+2} = \argmin_{x \in \mathcal{X}_0} \left\{ \tilde f_{t+1}(x) + \frac{\alpha}{2} \|x - \omega_{k+1}\|^2 \right\}$. It admits the analytical solution 
\begin{align*}
\theta_{t+1} &= \min\left\{1,\frac{\alpha\left(f(\xcandidate)-\hat{f}_t(\xcandidate)\right)}{\|g_{t+1}-s_{t+1}\|^2}\right\},\\
  y_{t+2} &= \omega_{k+1} - \frac{1}{\alpha}\left(\theta_{t+1} g_{t+1} +(1-\theta_{t+1})s_{t+1} + \hat{h}_{t+1} \right),
\end{align*}
where $ \hat{h}_{t+1} \in \mathcal{N}_{\mathcal{X}_0}(y_{t+2})$.
Hence, the objective of the proximal subproblem can be lower bounded by 
\begin{equation*}
    \begin{aligned}
        &  f(x_{t+2}^\star) + \frac{\alpha}{2}\|x_{t+2}^\star - \omega_{k+1} \|^2 \\
         \geq  & \; \tilde f_{t+1}(x_{t+2}^\star) + \frac{\alpha}{2}\|x_{t+2}^\star - \omega_{k+1} \|^2 \\
         \geq & \; \tilde f_{t+1}(y_{t+2}) + \frac{\alpha}{2}\|y_{t+2} - \omega_{k+1} \|^2 \\
         \geq & \; \theta_{t+1} (\hat{f}_{t}(\xcandidate) + \langle s_{t+1}, y_{t+2}-\xcandidate\rangle) + (1-\theta_{t+1})(f(\xcandidate) + \langle g_{t+1}, y_{t+2}-\xcandidate\rangle) + \frac{\alpha}{2}\|y_{t+2} - \omega_{k+1} \|^2 \\
         = & \; f(\xcandidate) + \theta_{t+1} \left(f(\xcandidate)-\hat{f}_{t}(\xcandidate)\right) - \frac{\theta_{t+1}^2}{2\alpha}\|g_{t+1} -s_{t+1}\|^2 + \frac{\alpha}{2}\|\xcandidate-\omega_{k+1}\|^2 + \frac{1}{2\alpha}\left\|\hat{h}_{t+1}\right\|^2 \\
         \geq & \;  f(\xcandidate) + \theta_{t+1} \left (f(\xcandidate)-\hat{f}_{t}(\xcandidate) \right) - \frac{\theta_{t+1}^2}{2\alpha}\|g_{t+1} -s_{t+1}\|^2 + \frac{\alpha}{2}\|\xcandidate-\omega_{k+1}\|^2,
    \end{aligned}
\end{equation*}
where the equality uses the definition of $y_{t+2}$ and the fact that 
$$y_{t+2}  =\xcandidate -\frac{1}{\alpha}\left(\hat{h}_{t+1}+\theta_{t+1}(g_{t+1}-s_{t+1})\right) $$ since $\xcandidate = \omega_{k+1} - \frac{1}{\alpha}s_{t+1}$ by the optimality condition of the subproblem $\argmin_{x \in \mathcal{X}_0}  \left\{\hat{f}_t(x) + \frac{\alpha}{2}\|x - \omega_{k+1}\|^2 \right\}$ at iteration $t$. 
Thus, we have 
\begin{align*}
    \widetilde\Delta_{t+1} \leq  \widetilde\Delta_{t} - \left ( \theta_{t+1} \left(f(\xcandidate)-\hat{f}_{t}(\xcandidate)\right) - \frac{\theta_{t+1}^2}{2\alpha}\|g_{t+1} -s_{t+1}\|^2 \right).
\end{align*}
The amount of decrease above can be lower bounded as follows 
\begin{align*}
      &\theta_{t+1} \left(f(\xcandidate)-\hat{f}_{t}(\xcandidate)\right) - \frac{\theta_{t+1}^2}{2\alpha}\|g_{t+1} -s_{t+1}\|^2 \\
      & \; \geq \frac{1}{2} \min \left \{  f(\xcandidate)-\hat{f}_{t}(\xcandidate) , \frac{\alpha \left(f(\xcandidate)-\hat{f}_{t}(\xcandidate)\right)^2 }{\|g_{t+1} -s_{t+1}\|^2}  \right \} \\
      & \; \geq \frac{1}{2} \min \left \{  (1-\beta) \widetilde\Delta_{t} , \frac{\alpha (1-\beta)^2\widetilde\Delta_{t}^2 }{ \|g_{t+1}-s_{t+1}\|^2 }  \right \}\\ 
      & \; \geq \frac{1}{2} \min \left \{  (1-\beta) \widetilde\Delta_{t} , \frac{\alpha (1-\beta)^2\widetilde\Delta_{t}^2 }{ 2\|g_{t+1}\|^2 + 2\|s_{t+1}\|^2}  \right \},
\end{align*}
where the first inequality uses the definition of $\theta_{t+1}$, the second inequality uses the definition of the null step, and the third inequality uses Young's inequality.
It is clear that both components in the minimum above are non-negative and we have a weaker result that $\widetilde\Delta_{t+1}$ is non-increasing. Upon further utilizing the relation $\|s_{t+1}\|^2 \leq 2 \alpha \widetilde\Delta_{t} \leq G^2_{k+1}$ (which will be shown later), the decrease bound above is at least
\begin{align*}
    &= \frac{1}{2} \min \left \{  (1-\beta) \widetilde\Delta_{t} , \frac{\alpha (1-\beta)^2\widetilde\Delta_{t}^2 }{ 2\|g_{t+1}\|^2 + 2\|s_{t+1}\|^2}  \right \} \\
    &\geq \frac{1}{2} \min \left \{  2 \frac{\alpha(1-\beta) \widetilde\Delta_{t}^2}{G_{k+1}^2} , \frac{\alpha (1-\beta)^2\widetilde\Delta_{t}^2 }{ 4G_{k+1}^2}  \right \} \\
    &\geq \frac{\alpha (1-\beta)^2 \widetilde\Delta_{t}^2}{8G_{k+1}^2} \\
    & \geq \frac{\alpha (1-\beta)^2 \widetilde\Delta_{t}^2}{8M^2},
\end{align*}
where the last inequality applies the observation that the subgradient $\|G_{k+1}\|$ is uniformly upper bounded by the $M$ as $f$ is $M$-Lipschitz. Hence, we conclude that $\widetilde\Delta_{t+1} \leq  \widetilde\Delta_{t}  - \frac{\alpha (1-\beta)^2 \widetilde\Delta_{t}^2}{8M^2}.$

We then show the inequality $\|s_{t+1}\|^2 \leq 2 \alpha \widetilde\Delta_{t} \leq G^2_{k+1}$. By the facts that $\hat{f}_{t}(\cdot) + \frac{\alpha}{2} \|\cdot - \omega_{k+1}\|^2 $ is $\alpha$ strongly convex and that the minimizer $\xcandidate$ is unique, there exists a $h_{t} \in \mathcal{N}_{\mathcal{X}_0}(x_{t+1}^\star)$ and $v_{t} \in \partial \hat{f}_t(\xcandidate) + \alpha (\xcandidate - \omega_{k+1})$ such that $0 = v_{t} + h_{t}$. It follows from the first-order condition of a strongly convex function that
\begin{align*}
   \hat{f}_t(\omega_{k+1}) & \geq   \hat{f}_t(\xcandidate) + \frac{\alpha}{2} \|\xcandidate - \omega_{k+1}\|^2  - \langle h_{t} ,\omega_{k+1} - \xcandidate \rangle+\frac{\alpha}{2} \|\xcandidate- \omega_{k+1} \|^2 \\
   & \geq   \hat{f}_t(\xcandidate) + \frac{\alpha}{2} \|\xcandidate - \omega_{k+1}\|^2 + \frac{\alpha}{2} \|\xcandidate - \omega_{k+1}\|^2,
\end{align*}
where the second inequality uses the definition of normal cone $\langle  h_t, y -  x_{t+1}^\star\rangle \leq 0, \forall y \in \mathcal{X}_0.$ Using the lower bound property $\hat{f}_t(\omega_{k+1}) \leq f(\omega_{k+1})$, we have
\begin{align*}
    \frac{\alpha}{2} \|x_{t+1}^\star - \omega_{k+1}\|^2 \leq f(\omega_{k+1}) - \left(\hat{f}_t(\xcandidate)+ \frac{\alpha}{2} \|\xcandidate -  \omega_{k+1}\|^2\right) = \widetilde\Delta_{t} \leq \widetilde\Delta_{k+1} ,
\end{align*}
where the second inequality comes from the fact that $\widetilde\Delta_{t}$ is non-increasing. 

By the construction of the lower approximation model at step $k+1$ and Young's inequality, we know
\begin{align*}
\hat{f}_{k+1}(x^\star_{k+2}) & \geq f(\omega_{k+1}) + \langle g_{k+1},x^\star_{k+2}-\omega_{k+1}\rangle\\
                 & \geq f(\omega_{k+1}) -\frac{1}{2}\left( \frac{\|g_{k+1}\|^2}{\alpha} + \alpha\|x^\star_{k+2}-\omega_{k+1}\|^2\right),
\end{align*}
where $g_{k+1} \in \partial f(\omega_{k+1}).$ 
Combining the above inequality and the definition $\widetilde \Delta_{k+1}$, we get the relationship
\begin{align*}
     \frac{\alpha}{2} \|x_{t+1}^\star - \omega_{k+1}\|^2 \leq \widetilde\Delta_{t} \leq \widetilde \Delta_{k+1}   \leq \frac{1}{2\alpha}\|g_{k+1}\|^2,
\end{align*}
which implies $ \|s_{t+1}\|^2=\alpha^2 \|x_{t+1}^\star - \omega_{k+1}\|^2 \leq 2 \alpha \widetilde\Delta_{t}^2 \leq 2 \alpha \widetilde \Delta_{k+1} \leq \|g_{k+1}\|^2 \leq G_{k+1}^2$.

\end{proof}
\begin{lemma}[\bf Lower bound on the proximal gap] \label{lem:prox-gap-bound}
  Fix a minimizer $x^{\star}$ in \cref{eq:non-smooth-problem-appendix} and let $\omega_k\in\RR^n \backslash \{x^{\star}\}$, the proximal gap is lower bounded by
  \begin{equation*} \label{eq:prox-gap-bound}
    \Delta_k \geq \begin{cases} \dfrac{1}{2\alpha}\left(\dfrac{f(\omega_k)-f(x^*)}{\|\omega_k-x^*\|}\right)^2, & \text{ if } f(\omega_k)-f(x^*) \leq \alpha\|\omega_k-x^*\|^2,\\
      f(\omega_k)-f(x^*)-\dfrac{\alpha}{2}\|\omega_k-x^*\|^2, & \text{ otherwise.} \end{cases}
  \end{equation*}
\end{lemma}
\begin{proof}
    This follows \cite[Lemma 7.12]{ruszczynski2011nonlinear} exactly and we omit the proof here. 
\end{proof}

The proof of \Cref{lemma-iterations-bound} follows the same arguments in \cite[Section 5]{diaz2023optimal} via combining \Cref{lem:descent-prox-gap,lem:null-prox-gap,lem:prox-gap-bound}. Indeed, 
\begin{itemize}
    \item \Cref{lem:descent-prox-gap} gives the progress made in the descent step with the relation to the proximal gap; 
\item \Cref{lem:null-prox-gap} shows that the maximum number of null steps between two descent steps depends on the proximal gap; and \item \Cref{lem:prox-gap-bound} shows the lower bound on the proximal gap. 
\end{itemize}
As a result, by the lower bound on the proximal gap given in \Cref{lem:prox-gap-bound}, the number of null steps between two descent steps can be bounded using \Cref{lem:null-prox-gap}. Combining \Cref{lem:descent-prox-gap,lem:prox-gap-bound}, we can achieve the maximum number of descent steps to achieve the desired accuracy. Finally, considering both the maximum number of null steps and descent steps renders the maximum number of total steps to achieve the desired accuracy, which gives the convergence results in \Cref{lemma-iterations-bound}. We refer interested readers to \cite[Section 5]{diaz2023optimal} for more details. 
    \section{Technical proofs in \Cref{sec:Penalized-nonsmooth-formulations-for-SDPs}} 
\label{Appendix:proofs-in-penalty}
\subsection{Proof of \Cref{proposition-primal-exact-penalty}} \label{subsection:proof-penalty}

We first recall two technical lemmas: the first one is the KKT optimality condition for the primal and dual SDPs \cref{eq:SDP-primal} and \cref{eq:SDP-dual}, which is shown in \Cref{theorem: optimality condition}, and the second one is the computation of the subdifferential of the maximal eigenvalue of symmetric matrices. 

\begin{lemma}[{\cite[Theorem 2]{overton1992large}}] \label{lemma:subdifferential-eigenvalue}
Given a symmetric matrix $A \in \mathbb{S}^n$. Suppose its maximal eigenvalue $\lambda_{\max}(A)$ has multiplicity $t$. Then, we have 
$$
\partial \lambda_{\max}(A) = \{Q U Q^\tr \mid U \in \mathbb{S}^t_+, \Trace(U) = 1\},
$$
where the columns of $Q \in \mathbb{R}^{n \times t}$ forms an orthonormal set of the eigenvectors for $\lambda_{\max}(A)$. 
\end{lemma}

We are now ready to prove \Cref{proposition-primal-exact-penalty}.

\begin{proof}    
    Upon denoting $\mathcal{X}_0 = \{X \in \mathbb{S}^{n}\,|\, \langle A_i, X\rangle = b_i, \, i = 1, \ldots, m\}$, it is easy to see that the primal SDP \cref{eq:SDP-primal} is equivalent to 
    \begin{equation}\label{eq:SDP-primal-eigen}
    \begin{aligned}
        \min_{X} \quad & \langle C, X\rangle \\
        \mathrm{subject~to} \quad & \lambda_{\max}(-X) \leq 0,  \\
        &X \in \mathcal{X}_0,
\end{aligned}
\end{equation}
where we have applied the fact that $X \succeq 0 \Leftrightarrow -X  \preceq 0 \Leftrightarrow \lambda_{\max}(-X) \leq  0$. \Cref{theorem:exact-penalization} ensures that the penalty form \cref{eq:SDP-primal-penalized} is equivalent to \cref{eq:SDP-primal-eigen} if we choose $\rho > \rho_0$ with $\rho_0 = \sup_{\lambda \in \Lambda} |\lambda|$, where $\Lambda \subset \mathbb{R}$ is the set of Lagrange multipliers associated with the inequality $\lambda_{\max}(-X) \leq  0$. 

Denote $\tilde{\Lambda} = \{\Trace (\Zstar) \in \mathbb{R}\mid (\ystar,\Zstar) \in \Dstar\}$, where $\Dstar$ is the set of dual optimal solutions to \cref{eq:SDP-dual}. We only need to show that $\tilde{\Lambda} = \Lambda$.  To see this, we start with the KKT optimality condition for~\cref{eq:SDP-primal-eigen}, i.e.,
\begin{subequations}
\begin{align}
   &0   \in C + \alpha \partial (\lambda_{\max}(-\Xstar))+ N_{\mathcal{X}_0}(\Xstar) \;\;
    \Longleftrightarrow \;\; 0  \in C + \alpha \partial (\lambda_{\max}(-\Xstar))- \sum_{i=1}^{m}A_i y_i, \label{eq:mutliplier-a} \\
      &      \alpha \lambda_{\max}(-\Xstar) = 0,\quad  \alpha \geq0, \quad  \lambda_{\max}(-\Xstar) \leq 0,  \quad \Xstar \in \mathcal{X}_0,  \label{eq:mutliplier-b} 
\end{align}
\end{subequations}
where $\mathcal{N}_{\mathcal{X}_0}(\Xstar) :=\{ X \in \mathbb{S}^{n} \,|\, \langle X, Y-\Xstar \rangle \leq 0, \forall Y \in \mathcal{X}_0 \} $ denotes the normal cone to $\mathcal{X}_0$ at $\Xstar$, and we have used the fact that $\mathcal{N}_{\mathcal{X}_0}(\Xstar) =\{\sum_{i=1}^{m}A_i y_i \in \mathbb{S}^{n} \,|\, y \in \mathbb{R}^{m}  \}$ since $\mathcal{X}_0 = \{X \in \mathbb{S}^{n}\,|\, \langle A_i, X\rangle = b_i, \, i = 1, \ldots, m\}$ is an affine space. 
The set of Lagrange multipliers $\Lambda$ associated with the inequality $\lambda_{\max}(-X) \leq  0$  is all $\alpha$ satisfying \cref{eq:mutliplier-a} and \cref{eq:mutliplier-b}. 

Our proof is divided into two steps. 

\begin{itemize}
    \item We first prove that $\tilde{\Lambda} \subseteq \Lambda $. For any dual optimal solution $(\ystar, \Zstar) \in \mathcal{D}^\star$, let the corresponding primal optimal solution be $\Xstar$. It suffices to show that $\alpha = \Trace(\Zstar)$, $\ystar$, and $\Xstar$ satisfy \cref{eq:mutliplier-a} and \cref{eq:mutliplier-b}. 
    
        If $\Xstar$ is strictly positive definite, then $\Zstar = 0$ according to \Cref{theorem: optimality condition}. Then $\Xstar$, $\alpha = \Trace(\Zstar) = 0$, and $C - \sum_{i=1}^m A_i y_i^\star = 0$ naturally satisfy \Cref{eq:mutliplier-a}-\Cref{eq:mutliplier-b}.
    
    If $\Xstar$ is positive semidefinite with $\lambda_{\max}(-\Xstar)  = -\lambda_{\min}(\Xstar)= 0$, then $\alpha = \Trace(\Zstar) \geq 0$ naturally satisfies \Cref{eq:mutliplier-b}.  Denote the multiplicity of $\lambda_{\min}(\Xstar) = 0$ as $r$, and the corresponding set of orthonormal eigenvectors as the columns of $ P \in \mathbb{R}^{n \times r }$. Via a classical chain rule (see e.g.,\cite[Theorem XI.3.2.1]{hiriart1993convex}) and \Cref{lemma:subdifferential-eigenvalue}, the subdifferential of $\lambda_{\max}(-\Xstar)$ is  
    \begin{equation*}
        \begin{aligned}
            \partial \lambda_{\max}(-\Xstar) = \{ -P U P^\tr \in \mathbb{S}^{n} \,|\, U \in \mathbb{S}^{r}_+, \Trace(U) = 1\}.
        \end{aligned}
    \end{equation*}
    By \Cref{theorem: optimality condition}, it is not difficult to see that $\Zstar \in \{\Trace(\Zstar) \times P U P^\tr \,|\, U \in \mathbb{S}^{r}_+, \Trace(U) = 1\}$, which indicates
    $$
    C - \sum_{i=1}^m A_i y_i^\star =\Zstar \in - \Trace(\Zstar)\partial \lambda_{\max}(-\Xstar). 
    $$
    Thus, $\alpha = \Trace(\Zstar)$, $\ystar$, and $\Xstar$ also satisfy \cref{eq:mutliplier-a}.
        
    \item We then prove $\Lambda \subseteq \tilde{\Lambda}$. It suffices to show that for any $\alpha, y\in \mathbb{R}^m, \Xstar$ satisfying  \cref{eq:mutliplier-a}-\cref{eq:mutliplier-b}, the points $\Xstar, (y, Z = C - \sum_{i=1}^m A_i y_i)$ are a pair of primal and dual optimal solutions to the SDPs \cref{eq:SDP-primal} and \cref{eq:SDP-dual} (note that \cref{eq:mutliplier-a} implies that $\Trace(Z) = \Trace(C - \sum_{i=1}^m A_i y_i) = \alpha$).
    
    By definition, it is clear that $\Xstar, y, Z = C - \sum_{i=1}^m A_i y_i$ are primal and dual feasible, i.e.,
    $$
    \langle A_i, \Xstar \rangle = b_i, i = 1, \ldots, m, \quad \sum_{i=1}^m A_i y_i + Z = C,  \quad \Xstar \succeq 0, \quad Z \succeq 0.
    $$
   By \Cref{theorem: optimality condition}, we only need to show that the complementarity slackness  $\Xstar Z = 0$ also holds.  
   If $\lambda_{\max}(-\Xstar) < 0$, then $\alpha = 0$ by \Cref{eq:mutliplier-b}. Thus, we have $Z = C - \sum_{i=1}^m A_i y_i  \in - \alpha \partial \lambda_{\max}(-\Xstar) =\{0\},$ leading to $\Xstar Z = 0$.
   If $\lambda_{\max}(-\Xstar) = 0$ with multiplicity $r$, then $PUP^\tr \times \Xstar = 0, \forall U \in \mathbb{S}^r$, where the columns of $P \in \mathbb{R}^{n \times r}$ forms an orthogonal set of the eigenvectors for $\lambda_{\max}(-\Xstar)$. Combining this fact with the definition $Z = C - \sum_{i=1}^m A_i y_i  \in - \alpha \partial \lambda_{\max}(-\Xstar)$, we also have $\Xstar Z = 0$.  
   Therefore,  the points $\Xstar, y, Z = C - \sum_{i=1}^m A_i y_i$ satisfy the KKT condition for \cref{eq:SDP-primal} and \cref{eq:SDP-dual}. 
\end{itemize}

Therefore, we have established that $\tilde{\Lambda} = \Lambda$. Together with \Cref{theorem:exact-penalization}, this completes the proof. 
\end{proof}

\subsection{Constant trace property in primal and dual SDPs}\label{sec:derivation-constant-trace-SDP}
In this subsection, we give a detailed derivation for the penalized primal/dual problem when the feasible set of \cref{eq:SDP-primal,eq:SDP-dual} implies a constant trace constraint on the decision variables $X$ and $Z$, i.e. $\Trace(X) = k$ and $\Trace(Z) = k$ for some $k>0$.  

\subsubsection{Constant trace in primal SDPs: Derivation of \cref{eq:SDP-dual-penalized-trace}}

\label{subsubsec:const-trace-primal-eigen}
Without loss of generality, we can add the explicit constant trace constraint $\Trace(X) = k$ to \cref{eq:SDP-primal}, leading to 
\begin{equation*}
    \begin{aligned}
        \min_{X} \quad & \langle C, X\rangle \\
        \mathrm{subject~to} \quad & \langle A_i, X\rangle = b_i, \quad i = 1, \ldots, m,  \\
         \quad & \langle I,X \rangle = k, \\
        & X \in \mathbb{S}^n_+. 
    \end{aligned}
\end{equation*}
Then, the corresponding Lagrangian dual problem reads as
\begin{equation} \label{eq:dual-formulation-for-primal-constant-trace}
    \begin{aligned}
        \max_{y,t} \quad & b^\tr y + t k \\
        \mathrm{subject~to} \quad & Z + \sum_{i=1}^m A_i y_i + t I = C, \\
        & Z \in \mathbb{S}^n_+.
    \end{aligned}
\end{equation}
For every pair of optimal solutions $X^\star$ and $Z^\star$, we have $\Trace(\Xstar) = k>0$ and $\text{rank}(\Zstar) \leq n-1$ due to the rank condition in \cref{eq:complementarity}. This implies $\lambda_{\min}(\Zstar) = 0$. By eliminating the variable $Z$ in \cref{eq:dual-formulation-for-primal-constant-trace}, we get 
$$
\lambda_{\min}\left(C - \sum_{i=1}^m A_i y_i - t I\right) = 0 \quad \Rightarrow \quad t = \lambda_{\min}(C - \Ajmap(y)) =  -\lambda_{\max}(\Ajmap(y) -C).
$$
Therefore, the problem \cref{eq:dual-formulation-for-primal-constant-trace} can be equivalently written as 
\begin{equation*}
    \begin{aligned}
        \max_{y} \quad & b^\tr y - k \lambda_{\max}(\Ajmap(y) -C),
    \end{aligned}
\end{equation*}
which is also equivalent to \cref{eq:SDP-dual-penalized-trace}.

\subsubsection{Constant trace in dual SDPs: Derivation of \cref{eq:SDP-primal-penalized-trace}}
Similarly, we can add the explicit constant trace constraint, $\Trace(Z) = k$, to the dual SDP \cref{eq:SDP-dual}. Its Lagrange dual problem becomes
\begin{equation*} 
    \begin{aligned}
        \min_{X,t,Q} \quad & \langle C,Q \rangle + k t  \\
        \mathrm{subject~to} \quad & \langle A_i,Q \rangle =  b_i, \quad i=1,\ldots,m,\\
        & Q + t I  = X,\\
        & X \in \mathbb{S}^n_+. 
    \end{aligned}
\end{equation*}
Following the same argument in \Cref{subsubsec:const-trace-primal-eigen}, we know 
$$\lambda_{\min}(Q + t I) = 0 \quad \Rightarrow \quad t = -\lambda_{\min}(Q) = \lambda_{\max}(-Q),$$ 
which leads to
\begin{equation*} 
    \begin{aligned}
        \min_{Q} \quad & \langle C,Q \rangle + k \lambda_{\max}(-Q)  \\
        \mathrm{subject~to} \quad & \langle A_i,Q \rangle =  b_i,\quad i=1,\ldots,m.
    \end{aligned}
\end{equation*}
This is clearly equivalent to \cref{eq:SDP-primal-penalized-trace}.

\begin{example}[Constant trace in dual SDPs] \label{subsec:constant-trace}
Here, we point out a simple observation that a specific matrix completion problem admits a constant trace property in the dual variable $Z$. Matrix completion aims to recover a low-rank matrix $M \in \mathbb{R}^{s \times t}$ from its partially observed entries $\{M_{ij}\}_{(i,j)\in \Omega}$. A typical formulation with a nuclear norm regularization reads as \cite{candes2012exact}
\begin{equation*}
    \begin{aligned}
         \min_{X} \quad & \|X\|_{*} \\
        \mathrm{subject~to} \quad & X_{ij} = M_{ij},\; \forall (i,j) \in \Omega,
    \end{aligned}
\end{equation*}
where $\|X\|_{*}$ denotes the nuclear norm of $X$. The  problem above can be cast as a standard primal SDP
\begin{equation} \label{eq:SDP-matrix-completion}
    \begin{aligned}
         \min_{X} \quad & \Trace(X) \\
        \mathrm{subject~to} \quad & \left \langle \begin{bmatrix} 0 & E_{ij}^\tr \\
                                                    E_{ij} & 0\end{bmatrix},  X \right \rangle =  2M_{ij},\forall (i,j) \in \Omega, \\
        & X  = \begin{bmatrix}
                W_1 & U^\tr \\
                U & W_2
        \end{bmatrix}  \in \mathbb{S}^{s+t}_+,
    \end{aligned}
\end{equation}
where $E_{ij} \in \mathbb{R}^{s \times t}$ is zero everywhere except the $(i,j)$ entry being $1$. Letting $A_{ij} = \begin{bmatrix} 0 & E_{ij}^\tr \\ E_{ij} & 0\end{bmatrix}$ and re-indexing the matrices $\{A_{ij}\}$ by integers $s= 1, \ldots , |\Omega|$, the dual SDP for \cref{eq:SDP-matrix-completion} becomes
\begin{equation*}
    \begin{aligned}
        \max_{y,Z} \quad & b^\tr y \\
        \mathrm{subject~to} \quad & \sum_{s = 1}^{|\Omega|} A_{s}y_{s} + Z = I, \\
        \quad & Z \in \mathbb{S}^{s+t}_+.
    \end{aligned}
\end{equation*}
Note that $\Trace\left(A_{s}\right) = 0$ and any feasible variable $Z$ has the constant trace property $\Trace(Z) = \Trace(I) = s+t$. 
\end{example}
    \section{Computation and conversion in \SBMP~and\SBMD} \label{appendix:reformulation}

\subsection{Reformulation of master problem \cref{eq:SBP-equivalence} in \SBMP} \label{sec:reform-ms-pb}
As indicated in \Cref{proposition-primal-penalty}, \SBMP~needs to solve the following SDP with a quadratic cost function at each iteration (recall that the constraint set $\hat{\mathcal{W}}_t$ is defined in \cref{eq:constraint-W_t})
\begin{align*}
       \min_{W \in \hat{\mathcal{W}}_t,~y \in \mathbb{R}^m} \;\; \langle W-C,\Omega_t \rangle  - \langle b- \mathcal{A}(\Omega_t)  ,y \rangle+ \frac{1}{2 \alpha } \left\|W-C+\sum_{i=1}^{m}A_i y_i\right\|^2.
\end{align*} 
After dropping some constant terms and using the operator $\mathrm{vec}: \mathbb{S}^n \rightarrow \RR^{n^2}$ that stacks the columns of the input matrix on top of each other, the problem \cref{eq:SBP-equivalence} can be reformulated as
\begin{equation}
   \begin{aligned} \label{eq:ms-pb-reform-QP-1}
   \min_{\gamma, S, y} \quad & \begin{bmatrix} \gamma & \vectorize{S}^\tr & y^\tr \end{bmatrix}
        \begin{bmatrix} 
            Q_{11} & Q_{12} & Q_{13} \\
            Q_{12}^\tr & Q_{22} & Q_{23} \\
            Q_{13}^\tr & Q_{23}^\tr & Q_{33}
        \end{bmatrix}
        \begin{bmatrix} \gamma \\ \mathrm{vec}(S) \\ y\end{bmatrix} + \begin{bmatrix} q_1^\tr & q_2^\tr  &  q_3^\tr \end{bmatrix} \begin{bmatrix} \gamma \\ \vectorize{S} \\y\end{bmatrix} \\
        \mathrm{subject~to}\quad& \gamma \geq 0, S \in \mathbb{S}^r_+, \gamma + \Trace(S) \leq \rho, y \in \mathbb{R}^m,
    \end{aligned}
\end{equation}
where  $Q_{11}  = \langle \bar{W}_t,\bar{W}_t \rangle, Q_{22}  = I_{r^2},  Q_{33} = \begin{bmatrix} \Amap (A_1), \ldots, \Amap (A_m) \end{bmatrix}, Q_{12} = \vectorize{P_t^\tr \bar{W}_t P_t }^\tr, Q_{13}  = \Amap(\bar{W}_t)^\tr$ and
\begin{align*}
    Q_{23} & = \begin{bmatrix}\vectorize{P_t^\tr A_1 P_t} & \vectorize{P_t^\tr A_2 P_t} & \cdots &  \vectorize{P_t^\tr A_m P_t} \end{bmatrix}, \\
    q_{1} & = -2 \langle \bar{W}_t, C  \rangle + 2 \alpha \langle \bar{W}_t, \Omega_t  \rangle, \\
    q_{2} & = 2 \alpha\vectorize{P_t^\tr \Omega_t P_t} - 2 \vectorize{P_t^\tr C P_t}, \\
    q_{3} & = -(2 \alpha (b-\Amap(\Omega_t))+ 2\mathcal{A}(C)) ).
\end{align*}
Although the problem 
 \cref{eq:ms-pb-reform-QP-1} 
is already ready to be solved by standard conic solvers, the computation can be further simplified by eliminating the variable $y$. The optimality condition for $y$ follows
\begin{align*}
    y = Q_{33}^{-1} \Bigl(\frac{-q_3}{2} - Q_{13}^\tr \alpha - Q_{23}^\tr \vectorize{S}\Bigr),
\end{align*}
where $Q_{33}$ is invertible because of \cref{assumption:linearly-independence}.
Therefore, the problem becomes
\begin{equation}
   \begin{aligned}\label{eq:ms-pb-reform-QP-2}
   \argmin_{\gamma, S} \quad & \begin{bmatrix} \gamma & \vectorize{S}^\tr \end{bmatrix}
        \begin{bmatrix} 
            M_{11} & M_{12} \\
            M_{12}^\tr & M_{22}\\    
        \end{bmatrix}
        \begin{bmatrix} \gamma \\ \mathrm{vec}(S) \end{bmatrix} + \begin{bmatrix} m_1^\tr & m_2^\tr \end{bmatrix} \begin{bmatrix} \gamma \\ \vectorize{S} \end{bmatrix} \\
        \mathrm{subject~to}\quad& \gamma \geq 0, S \in \mathbb{S}^r_+, \gamma + \Trace(S) \leq \rho,
    \end{aligned}
\end{equation}
where $M_{11} = Q_{11} - Q_{13} Q_{33}^{-1} Q_{13}^\tr, M_{22} = Q_{22} - Q_{23} Q_{33}^{-1}  Q_{23}^\tr, M_{12} = Q_{12} - Q_{13} Q_{33}^{-1}  Q_{23}^\tr, m_1 = q_1  - Q_{13} Q_{33}^{-1} q_3$ and $m_2  =q_2 - Q_{23}Q_{33}^{-1} q_3$.
We note that \cref{eq:ms-pb-reform-QP-2} has only one non-negative variable, one semidefinite variable, and one inequality constraint. The computational complexity for solving  \cref{eq:ms-pb-reform-QP-2} is low 
when the dimension of the semidefinite variable is small.

\begin{remark}[Eliminating the equality constraints in \cref{eq:inner-optimization}] \label{remark:eliminate-equality}
Aside from deriving the dual problem for \cref{eq:inner-optimization}, another approach to solving equality-constrained problems is to eliminate the affine constraints by finding an explicit representation of the feasible set
\cite[Section 10.1.2]{boyd2004convex}. Specifically, \cref{eq:inner-optimization} can be reformulated as 
    \begin{align*}
        \min_{x \in \mathbb{R}^{p}} \quad & \Biggl\langle C-W, X_0 + \sum_{i=1}^{p} N_i x_i  \Biggr\rangle + \frac{\alpha}{2 } \left\|X_0 + \sum_{i=1}^{p} N_i x_i - \currposi \right \|^2,
    \end{align*}
    where $X_0 \in \mathbb{S}^n $ is a particular solution for the affine constraint, i.e., $\Amap(X_0) = b$, and $\{N_1, \ldots, N_p\}$ is a set of the orthonormal basis of the null space of $\Amap$. We note that the above problem
    is an unconstrained quadratic program in $x$, and thus we have the optimality condition
    \begin{align*}
        x =  -K^{-1} \Biggl(\Nmap\biggl(\frac{1}{\alpha}(C-W) + X_0-\currposi \biggr)\Biggr),
    \end{align*}
    where $\mathcal{N}(\cdot)$ is a linear map $ \mathbb{S}^{n} \to \mathbb{R}^p$ as 
$
\mathcal{N}(X) := \begin{bmatrix} \langle N_1, X \rangle, \ldots , \langle N_p, X \rangle \end{bmatrix}^\tr$, and $K = \begin{bmatrix} \Nmap(N_1) \ldots \Nmap(N_p) \end{bmatrix} \in \mathbb{S}^p_{++}$.
   Therefore, another equivalent problem of \cref{eq:SBM-subproblem} becomes (after dropping some constant terms) 
    \begin{align*}
        \argmax_{W \in  \hat{\mathcal{W}}_t} \quad   - \langle X_0 - \Njmap (K\Nmap(X_0-\currposi)) ,W \rangle -\frac{1}{2\alpha}\left \|(K^{-1})^{\frac{1}{2}} \Nmap(C-W) \right\|^2,
    \end{align*}
    where $\Njmap(\cdot):\mathbb{R}^p \to \mathbb{S}^{n} $ is the adjoint operation of $\Nmap$. 
\end{remark}

\subsection{Reformulation of master problem \eqref{eq:SBD-equivalence} in \SBMD}
\label{sec:reform-ms-pb-dual}
Similar to \SBMP, \SBMD~solves the following quadratic SDP at every iteration
\begin{align*} 
   \min_{W \in \hat{\mathcal{W}}_t} \;\; \langle b,\ycurrposi \rangle   + \langle W, C - \Ajmap(\ycurrposi) \rangle + \frac{1}{2 \alpha } \left\|b - \Amap (W)\right\|^2, 
\end{align*} 
where the constraint set is defined as
\begin{equation*} 
    \hat{\mathcal{W}}_t := \{  \gamma \bar{W}_t + P_t S P_t^\tr \in \mathbb{S}^n \;|\; S \in \mathbb{S}^r_+ , \gamma \geq 0, \gamma + \Trace(S) \leq \rho \}.
\end{equation*}
By removing some constants and performing a scaling, the problem can be reformulated as 
\begin{align}\label{eq:SBMD-ms-pb-reform-QP}
    \min_{\gamma, S} \quad & \begin{bmatrix} \gamma & \vectorize{S}^\tr \end{bmatrix}
        \begin{bmatrix} 
            M_{11} & M_{12} \\
            M_{21} & M_{22}
        \end{bmatrix} 
        \begin{bmatrix} \gamma \\ \mathrm{vec}(S)\end{bmatrix}
        + \begin{bmatrix} m_1^\tr & m_2^\tr \end{bmatrix}\begin{bmatrix} \gamma \\ \vectorize{S}\end{bmatrix} \\
    \mathrm{subject~to}\quad& \gamma \geq 0, S \in \mathbb{S}^r_+, \gamma + \Trace(S) \leq \rho, y \in \mathbb{R}^m \nonumber,
\end{align}
where the problem data are
\begin{align*}
    M_{11} & =  \langle \Amap(\bar{W}_t),\Amap(\bar{W}_t) \rangle  , &
    M_{22} & =  \sum_{i=1}^m  \vectorize{P_t^\tr A_i P_t} \vectorize{P_t^\tr A_i P_t}^\tr,\\
    M_{12} & = P_t^\tr \Ajmap(\Amap(\bar{W}_t))P_t, &
    m_1 & = \langle -2\Ajmap(b) + 2 \alpha G, \bar{W}_t \rangle ,\\
    m_2 & = P_t^\tr (-2\Ajmap(b)+ 2 \alpha G) P_t, &
    G  &  = C - \Ajmap(\ycurrposi).
\end{align*}
It is clear that \cref{eq:ms-pb-reform-QP-2,eq:SBMD-ms-pb-reform-QP} are in the same form as \cref{eq:SBM-ms-pb-QP-abstract} and have decision variables of the same dimension. Thus, solving \cref{eq:ms-pb-reform-QP-2,eq:SBMD-ms-pb-reform-QP} has the same computational complexity. However, the problem data $M, m_1, m_2$ in \cref{eq:ms-pb-reform-QP-2,eq:SBMD-ms-pb-reform-QP} are different, and their constructions take different time consumption. 

\subsection{Analytical solution when $\rpast = 0$ and $\rcurrent = 1$} \label{appendix: r=1}
The subproblem \cref{eq:SBP-equivalence} or \Cref{eq:SBD-equivalence} can be reformulated into a conic problem of the form  \cref{eq:ms-pb-reform-QP-2}. In general, the subproblem \cref{eq:ms-pb-reform-QP-2} does not admit an analytical solution and needs to be solved by another algorithm (e.g., MOSEK \cite{mosek} or SeDuMi \cite{sedumi}).  
Here, we highlight that an analytical solution to \cref{eq:ms-pb-reform-QP-2} exists when parameters $\rpast = 0$ and $\rcurrent = 1$. In this case, the conic problem \cref{eq:ms-pb-reform-QP-2} is reduced to 
\begin{equation}
   \begin{aligned}\label{eq:ms-pb-reform-QP-simple}
   \min_{\gamma, s} \quad & \begin{bmatrix} \gamma & s \end{bmatrix}
        M \begin{bmatrix} \gamma \\ s \end{bmatrix} + m^\tr \begin{bmatrix} \gamma \\ s \end{bmatrix} \\
        \mathrm{subject~to}\quad& \gamma \geq 0, s \geq 0, \gamma + s \leq \rho,
    \end{aligned}
\end{equation}
where $M \in \mathbb{S}^{2}_{+}$ and $m \in \mathbb{R}^{2}$ are problem data. 

There are only two scalar decision variables in \Cref{eq:ms-pb-reform-QP-simple}, and the feasible region forms a triangle in the first quadrant.
Therefore, the solution to \cref{eq:ms-pb-reform-QP-simple} depends on the location of the minimum of the quadratic objective function. Precisely, the minimum of the objective function happens at 
$$
    \begin{bmatrix} \gamma_{\mathrm{obj}}  \\ s_{\mathrm{obj}}  \end{bmatrix} = -\frac{M^{\dag}m}{2},
$$
where $M^{\dag}$ is the Moore–Penrose Pseudoinverse of $M$. If $\gamma_{\mathrm{obj}}$ and $ s_{\mathrm{obj}}$ are feasible in \cref{eq:ms-pb-reform-QP-simple}, they are also the optimal solution to \cref{eq:ms-pb-reform-QP-simple}. If, however, $\gamma_{\mathrm{obj}}$ and $ s_{\mathrm{obj}}$ are not feasible, the solution to \cref{eq:ms-pb-reform-QP-simple} will be on one of the three line segments defined by 
\begin{align*}
     l_1 &=\left \{ (\gamma,0) \in \mathbb{R}^2 \mid 0 \leq \gamma \leq \rho\right \}, \\
     l_2 &= \left\{(0,s) \in \mathbb{R}^2 \mid 0 \leq s \leq \rho\right \}, \\
     l_3 &= \left\{ (\gamma , s) \in  \mathbb{R}^2 \mid \gamma  +  s = \rho, \gamma \geq 0, s \geq 0 \right \}.
\end{align*}
Each line segment is one-dimensional and thus the minimizer of a quadratic function on each line segment can be calculated analytically:

\begin{itemize}
    \item 
 For the line segment $l_1$, the minimizer is 
\begin{align*}
    v_1 = \begin{bmatrix} \hat{\gamma}  \\ 0 \end{bmatrix} \text{ and } \hat{\gamma} = \begin{cases}     0 &\text{if} \quad \displaystyle-{m_1}/({2M_{11}}) < 0, \\
     \displaystyle -\frac{m_1}{2M_{11}} &\text{if} \quad 0 \leq   \displaystyle -{m_1}/(2M_{11}) \leq \rho, \\
    \rho &\text{if}\quad  \displaystyle- {m_1}/({2M_{11}}) > \rho.
    \end{cases} 
\end{align*}
\item For the line segment  $l_2$, the minimizer  is
\begin{align*}
     v_2 = 
     \begin{bmatrix} 0 \\ \hat{s} \end{bmatrix} \text{ and } 
     \hat{s} = \begin{cases} 
    0 &\text{if} \quad \displaystyle  -{m_2}/({2M_{22}}) < 0, \\
     \displaystyle -\frac{m_2}{2M_{22}} & \text{if}\quad 0 \leq  \displaystyle  -{m_2}/({2M_{22}}) \leq \rho, \\
    \rho &\text{if} \quad  \displaystyle  -{m_2}/({2M_{22}}) > \rho.
    \end{cases}    
\end{align*}
\item For the line segment  $l_3$, the minimizer  is 
\begin{align*}
      v_3 = \begin{bmatrix} \hat{\gamma} \\ \rho - \hat{\gamma} \end{bmatrix} \text{ and } \hat{\gamma} = \begin{cases}  
    0 &\text{if} \quad  \phi < 0, \\
    \phi &\text{if} \quad 0 \leq \phi \leq \rho, \\
    \rho &\text{if} \quad  \phi> \rho,
    \end{cases} 
\end{align*}
where $\phi = \frac{2\rho M_{22} - 2 \rho M_{12}  - m_1 + m_2}{2M_{11} + 2M_{22} - 4M_{12}} $.

\end{itemize}
Let $f(v) = v^{\tr} M v + m^{\tr} v$, where $v \in \mathbb{R}^2$, be the same objective function in \cref{eq:ms-pb-reform-QP-simple}. When $\gamma_{\mathrm{obj}}$ and $s_{\mathrm{obj}}$ are not feasible,  the solution to \cref{eq:ms-pb-reform-QP-simple} can be computed as
$\argmin_{v \in \{v_1,v_2,v_3\}} f(v).$

\subsection{Conversion between primal and dual formulations.} \label{appendix:conversion} 

Here, we show that the primal and dual algorithms \SBMP~and \SBMD~can be converted from each other through a reformulation of the SDPs  \cref{eq:SDP-primal,eq:SDP-dual}. Accordingly, \SBMP~is able to solve the dual SDP \cref{eq:SDP-dual} in a different form, and \SBMD~can solve the primal SDP \cref{eq:SDP-primal}. The idea is to make a conversion between the equality-form SDP \cref{eq:SDP-primal} and inequality-form SDP \cref{eq:SDP-dual}.

Consider the primal SDP \cref{eq:SDP-primal}. Any feasible point satisfying $\Amap(X) = b$ can be represented as
\begin{equation*}
    \begin{aligned}
        X = X_{\mathrm{p}} + \Njmap y,
    \end{aligned}
\end{equation*}
where $X_{\mathrm{p}} \in \mathbb{S}^n$ is a particular solution of $\Amap(X) = b$, $y \in \mathbb{R}^p$ with $p$ denoting the dimension of the null space of $\Amap$, and $\mathcal{N}(\cdot)$ is a linear map $ \mathbb{S}^{n} \to \mathbb{R}^p$ as $\mathcal{N}(X) := \begin{bmatrix} \langle N_1, X \rangle, \ldots , \langle N_p, X \rangle \end{bmatrix}^\tr$ such that its adjoint $\Njmap y$ represents the null space of $\Amap$.  
This allows us to equivalently rewrite \cref{eq:SDP-primal-penalized} as 
\begin{equation*}  
    \begin{aligned}
        &\min_{y} \quad \langle C, X_{\mathrm{p}} + \Njmap y \rangle    + \rho \max \{\lambda_{\max} (- X_{\mathrm{p}} - \Njmap y) ,0 \} \\
      = & \min_{y} \quad \langle \Nmap (C), y \rangle   + \rho \max \{\lambda_{\max} (- X_{\mathrm{p}} - \Njmap y) ,0 \},
    \end{aligned}
\end{equation*} 
where we drop a constant term and use the property $\langle C,\Njmap y \rangle = \langle \Nmap(C), y \rangle$.
It is clear that the problem above is in the same form as the penalized dual SDPs \cref{eq:SDP-dual-penalized} with problem data  
 $b = -\Nmap(C), \, \Amap = -\Nmap, \, C = X_{\mathrm{p}}.$
Therefore, we can apply \SBMD~in \Cref{alg:(r_p-r_c)-SBM-D} to solve it and convergence guarantees follow \Cref{thm: sublinearates,thm: linear convergence of Block SBM under the extra condition strict complementarity}. 

On the other hand, the dual SDP \Cref{eq:SDP-dual} can also be solved by \SBMP~in \Cref{alg:(r_p-r_c)-SBM-P}. Indeed, we can equivalently reformulate \cref{eq:SDP-dual} or \cref{eq:SDP-dual-penalized} into the form of \cref{eq:SDP-primal} or \cref{eq:SDP-primal-penalized}: 
\begin{itemize}
    \item We first find $G_1, \ldots,G_l \in \mathbb{S}^{n}$ and $h \in \mathbb{R}^l$, such that $\{ Z \in \mathbb{S}^n ~|~ \mathcal{G} (Z) = h \} = \{C - \Ajmap y ~|~ y\in \mathbb{R}^m \}$, where $\mathcal{G}$ is a linear map $: \mathbb{S}^n \rightarrow \mathbb{R}^l$ as $\mathcal{G}(X) = \begin{bmatrix} \langle G_1,X \rangle, \ldots,\langle G_m ,X\rangle \end{bmatrix}^\tr$. In particular,   $\{G_1, \ldots, G_l\}$ is a set of the basis of the null space of $\mathcal{A}$ and $h = \mathcal{G}(C)$.
    \item We then find a feasible point $X_\mathrm{f}$ satisfying the affine constraint in \cref{eq:SDP-primal}, i.e., $\mathcal{A}(X_\mathrm{f})=b$. 
\end{itemize}
Consequently, we can equivalently rewrite the dual penalized nonsmooth formulation \cref{eq:SDP-dual-penalized} in the form of
\begin{equation*}
    \begin{aligned}
        \min_{ Z} \quad & \langle X_{\mathrm{f}}, Z \rangle + \rho \max \{\lambda_{\max}(-Z),0 \} \\
        \mathrm{subject~to} \quad & \mathcal{G}(Z) = h,
    \end{aligned}
\end{equation*}
where we have replaced $b$ by $\Amap(X_\mathrm{f})$, introduced a constant $\langle C, X_\mathrm{f} \rangle$, and defined a new variable $Z = C - \Ajmap y$.
It is clear that the problem above is in the form of \cref{eq:SDP-primal-penalized}, which is ready to be solved by \SBMP~in \Cref{alg:(r_p-r_c)-SBM-P}. The convergence results follow \cref{thm: sublinearates-P-K,thm: linear-convergence}.

    \section{Technical proofs in \Cref{sec:SBM-Primal}} \label{Appendix:technical-proofs}

In this section, we complete the technical proofs in \Cref{sec:SBM-Primal}. We first verify that the lower approximation model \cref{eq:F_W_t_P_t} satisfies \cref{eq:bundle-method-property-1,eq:bundle-method-property-2,eq:bundle-method-property-3} in \Cref{appendix:verification}, and then prove \Cref{lemma:Primal-Dual-Gap-Feasibility} in \Cref{section:proof-lemma-appendix}. These two results complete the proof of \Cref{thm: sublinearates-P-K}. Finally, in \Cref{appendix:theorem-linearly-convergence}, we complete the proofs for \Cref{thm: linear-convergence}.

\subsection{Verification of the lower approximation model \cref{eq:F_W_t_P_t}} \label{appendix:verification}

 First, the constructions of $P_t$ in \cref{eq:update-P-t} and $\bar{W}_t $ in \cref{eq:update-W-t} guarantee that $P_t^\tr P_t = I_r, \bar{W}_t \succeq 0$ and $\Trace(\bar{W}_t) =1$. Then, we have 
$$
    \hat{\mathcal{W}}_t:=\{\gamma \bar{W}_t + P_tSP_t^\tr \mid S \in \mathbb{S}^r_+ , \gamma \geq 0,\gamma + \Trace(S) \leq \rho\} \subset \{ W \in \mathbb{S}^n_{+} \mid \Trace(W) \leq \rho\},
$$
which implies that
$$
\max_{W \in \hat{\mathcal{W}}_t }    \langle W, -X\rangle \leq \max_{\Trace(W) \leq \rho,  W \in \mathbb{S}^n_{+} }    \langle W, -X\rangle = \rho \max\{\lambda_{\max}(-X),0\}, \quad \forall X \in \mathbb{S}^n.
$$
Thus, $\hat{F}_{(\bar{W}_t,P_t)}(\cdot) $ 
is a global lower approximation of $F(\cdot)$
, satisfying \cref{eq:bundle-method-property-1}. 

 Second, the satisfaction of \cref{eq:bundle-method-property-2} is due to the fact that a subgradient of $F(X) $ at $\candidate$ at iteration $t$ is contained in the feasible set $\hat{\mathcal{W}}_{t+1}$ for the next model, when constructing $P_{t+1}$ in \cref{eq:update-P-t}. In particular,  the column space of $P_{t+1}$  in \cref{eq:update-P-t} spans the top $r_{\mathrm{c}}$ eigenvectors of $-\candidate$. Then, there exists a unit vector $s \in \mathbb{R}^r$ 
    such that $P_{t+1}s = v$, where $v$ is a top normalized eigenvector of $-\candidate$.  
Letting $\gamma = 0$, it is easy to verify that $\rho vv^\tr = P_{t+1} (\rho s s^\tr) P_{t+1}^\tr \in \hat{\mathcal{W}}_{t+1}$ since $\rho s s^\tr \in \mathbb{S}^r_+$ and $\Trace(\rho s s^\tr) \leq \rho$. Therefore, if $\lambdamax(-\candidate) > 0$, we have
    \begin{align*}
         \hat{F}_{(\bar{W}_{t+1},P_{t+1})} (X)   =  
         \langle C,X \rangle + \max_{W \in \hat{\mathcal{W}}_{t+1}} \langle W, -X \rangle 
    & \geq \langle C,X \rangle + \rho \langle vv^\tr, -X \rangle  \\
    & = F(\candidate) + \langle g_{t} ,  X- \candidate   \rangle, \quad \forall X \in \mathbb{S}^n,
    \end{align*}
    where $g_{t} =C-\rho vv^\tr \in \partial F(\candidate) $. On the other hand, if $\lambdamax(-\candidate) \leq 0$, since $0 \in \hat{\mathcal{W}}_{t+1} $, it also follows that  
    \begin{align*}
        \hat{F}_{(\bar{W}_{t+1},P_{t+1})} (X) \geq  \langle C , X \rangle 
         = F(\candidate) + \langle g_{t} , X- \candidate \rangle, \quad \forall X \in \mathbb{S}^n, 
    \end{align*}
    where $g_{t} = C \in \partial F(\candidate)$. 
    Hence, we have verified the subgradient lower-bound in \cref{eq:bundle-method-property-2} for $\hat{F}_{(\bar{W}_t,P_t)}(\cdot)$.
   
    Finally, the satisfaction of \cref{eq:bundle-method-property-3} is due to the fact that the optimal solution $\Wtstar$ at iteration $t$ is contained the feasible set $\hat{\mathcal{W}}_{t+1}$ at the next iteration $t+1$, thanks to the construction \cref{eq:update-P-t} and \cref{eq:update-W-t}. In particular, since the column space of $P_{t+1}$ spans the past information information $P_tQ_1$ as in \cref{eq:update-P-t}, there exists an orthonormal matrix $\bar{Q} \in \mathbb{R}^{ r \times r_{\mathrm{p}} }$ such that $P_{t+1}\bar{Q} = P_{t}Q_1$. Let $\gamma = \gamma_t^\star + \Trace(\Sigma_2)$ and $S =Q_1 \Sigma_1 Q_1^\tr$, then $W_{t}^\star = \gamma \bar{W}_{t+1} + P_{t+1}SP_{t+1}^\tr \in \hat{\mathcal{W}}_{t+1}$. Therefore, $\forall X \in \mathbb{S}^n$ such that $\Amap(X) = b$, we have
    \begin{align*}
         \hat{F}_{(\bar{W}_{t+1},P_{t+1})} (X) & =  
         \langle C,X \rangle + \max_{W \in \hat{\mathcal{W}}_{t+1}} \langle W, -X \rangle  \\
    & \geq \langle C,X \rangle +\langle W_{t}^\star, -X \rangle  \\
    & = \langle C- W_{t}^\star,\candidate \rangle +\langle W_{t}^\star, \candidate -X \rangle - \langle C,\candidate-X \rangle +  \langle \Ajmap(\ytstar),\candidate -X \rangle\\
    & = \hat{F}_{(\bar{W}_{t},P_{t})} (\candidate) + \langle -W_{t}^\star + C - \Ajmap(\ytstar) , X - \candidate \rangle \\
    & = \hat{F}_{(\bar{W}_{t},P_{t})} (\candidate) + \langle \alpha(\currposi-\candidate) , X - \candidate \rangle \\
    & = \hat{F}_{(\bar{W}_{t},P_{t})} (\candidate) + \langle s_{t+1}, X- \candidate \rangle,
    \end{align*}
    where the second equality uses $\langle \Ajmap(\ytstar), \candidate - X \rangle = \langle \Amap (\candidate - X \rangle) ,\ytstar\rangle = 0$ since $\candidate$ and $X$ are both feasible, the fourth equality uses the optimal condition \cref{eq:optimality-condition-Xt}, and the fifth equality sets $s_{t+1} =\alpha(\currposi-\candidate)  \in \partial \hat{F}_{(\bar{W}_{t},P_{t})} (\candidate) + \mathcal{N}_{\mathcal{X}_0}(\candidate)$ since $\mathcal{N}_{\mathcal{X}_0}(\candidate) = \{ \Ajmap y \mid \forall y \in \mathbb{R}^m \}$.  Therefore, we have verified the model subgradient lowerbound in \cref{eq:bundle-method-property-3}.
    \vspace{-2mm}

\subsection{Proof of \Cref{lemma:Primal-Dual-Gap-Feasibility}} \label{section:proof-lemma-appendix}

Before presenting the proof, we first draw a technical lemma that ensures the compactness of the sub-level set $\Dxo$ in \cref{lemma:Primal-Dual-Gap-Feasibility} (recall that $\Dxo := \sup_{F(\Omega_t) \leq F(\Omega_0) }\|\Omega_t\|$). 
\begin{lemma}\label{lemma:compact-sublevel-set}
    Given a closed convex set $A \subseteq \mathbb{R}^n $ and a convex function $f: A \rightarrow \mathbb{R}$,
    if the optimal solution set of $\min_{x \in A}f(x)$ is compact, then the sublevel set $C_{\epsilon} = \{x \in A \;|\; f(x^\star) \leq f(x^\star) + \epsilon \}$, $x^\star$ is a minimizer, is also compact for all $\epsilon > 0$.
\end{lemma}
\begin{proof}
    Let $\mathcal{X}$ be the optimal solution set of $\min_{x \in A}f(x)$. We first note that $f$ is closed as $f$ is a continuous function and the domain of $f$ is closed \cite[A.3.3]{boyd2004convex}. Thus, both $\mathcal{X}$  and $C_{\epsilon}$ are closed as the sublevel set of a closed function is closed \cite[A.3.3]{boyd2004convex}.
    We only need to prove $C_{\epsilon}$ is bounded. For this, it suffices to prove that the unboundedness of $C_{\epsilon}$ implies the unboundedness of $\mathcal{X}$. 
    
    We prove this by contradiction. For any $\epsilon > 0$, suppose $C_{\epsilon}$ is unbounded. Since $C_{\epsilon}$ is convex and closed, there exists a direction $d \in \mathbb{R}^n$ such that $\lim_{t \rightarrow \infty} \|x^\star + t d\|_2 = \infty$, where $x^\star$ is an optimal solution, and $x^\star + t d \in C_{\epsilon}, \forall t \geq 0$ \cite[Theorem 6.12]{soltanlectures}. If there exists a point $\hat{x}$ on this half-line such that $\hat{x} = x^\star + \hat{t} d, \hat{t} > 0$,  and  $f(\hat{x}) > f(x^\star)$, then according to \cite[Equation 2.30]{ruszczynski2011nonlinear}, we have 
    $$ 
    \frac{f(x^\star + td) - f(x^\star)}{t} \geq \frac{f\left(x^\star + \hat{t} d\right) - f(x^\star)}{\hat{t}} =:k > 0, \quad \forall t \geq \hat{t},
    $$
    which implies $f(x^\star + td) \geq f(x^\star) + t k, \forall t \geq \hat{t}$. Thus, we have
    $$
    \lim_{t \to \infty} f(x^\star + td) = \infty.
    $$
    However, this contradicts the assumption that $x^\star + t d \in C_{\epsilon}$ for all $ t \geq 0$. This indicates that $f(x^\star + td) = f(x^\star), \forall t \geq 0$, and consequently $\mathcal{X}$ is unbounded, which completes the proof. 
\end{proof}

To lighten the notation, for the rest of the section, we use $\Ft:= \hat{F}_{(\bar{W}_{t},P_{t})}$ to denote the approximate model at step $t$. For convenience, we restate \Cref{lemma:Primal-Dual-Gap-Feasibility} below. 
\begin{lemma}
        In \SBMP, let $\beta\in(0,1)$, $\rcurrent \geq 1$, $\rpast \geq 0 $, $ \alpha>0$, $r= \rcurrent+\rpast$,
	$\rho > 2\DZstar$. Then, at every descent step $t>0$, the following results hold.
	\begin{enumerate}
	\item The approximate primal feasibility for $\Omega_{t+1}$ satisfies
	\begin{equation*}
        \begin{aligned}
    	    \lambda_{\min} (\Omega_{t+1}) \geq \frac{-(F(\currposi)-F(\Xstar))}{ \DZstar }, \quad and \quad \mathcal{A}(\Omega_{t+1}) = b.
        \end{aligned}
	\end{equation*}
 \item The approximate dual feasibility for $(\Wtstar, \ytstar)$ satisfies 
	\begin{equation*}
            \begin{aligned}
	       \Wtstar \succeq 0,\quad and \quad \|\Wtstar-C+ \Ajmap(\ytstar) \|^2 \leq \frac{2 \alpha}{\beta} (F(\currposi)-F(\Xstar)).
            \end{aligned}
	\end{equation*}
	\item The approximate primal-dual optimality for $(\nextposi, \Wtstar, \ytstar)$ satisfies
    \begin{subequations}
	    \begin{align*}
             \langle C, \nextposi \rangle - \langle b,\ytstar \rangle & \geq -\frac{\rho (F(\currposi)-F(\Xstar))}{ \DZstar } -  \Dxo \sqrt{\frac{2 \alpha}{\beta} (F(\currposi)-F(\Xstar))},  \\
            \langle C, \nextposi \rangle - \langle b,\ytstar \rangle & \leq \frac{1-\beta}{\beta} (F(\currposi)-F(\Xstar))  +  \Dxo \sqrt{\frac{2 \alpha}{\beta} (F(\currposi)-F(\Xstar))},
	    \end{align*}
	\end{subequations}
    where $\Dxo := \sup_{F(\Omega_t) \leq F(\Omega_0) }\|\Omega_t\|$
    is bounded due to the compactness of $\Pstar$(see \cref{lemma:compact-sublevel-set}).
    \end{enumerate}
\end{lemma}

\begin{proof}
    
     First, by the construction of \cref{eq:optimal-W-t,eq:constraint-W_t}, we have $W_t^\star = \gamma^\star_t \bar{W}_t + P_t S_t^\star P_t$. Since $\bar{W}_t \in \mathbb{S}_+^n, \gamma^\star_t \geq 0$ and $S_t^\star \in \mathbb{S}^r_+$, $W_t^\star$ is positive definite. The inequality can be established by observing
    \begin{equation*}
        \begin{aligned}
            \frac{F(\currposi)-F(X_\star)}{\beta}& \geq \frac{F(\currposi)-F(\nextposi)}{\beta} \\ 
            & \geq F(\currposi) - \Ft(\nextposi) \\ 
            &\geq \frac{\alpha}{2} \|\nextposi-\currposi\|^2~\\
            &= \frac{\alpha}{2}\frac{1}{\alpha^2} \|W_t^\star-C+ \Ajmap(y_t^\star)\|^2 ~ \\
            &= \frac{1}{2\alpha} \|W_t^\star-C+ \Ajmap(y_t^\star)\|^2, 
        \end{aligned}
\end{equation*} 
where the second inequality comes from the definition of serious step in \cref{eq:null-or-serious-step-SBM}, the third inequality uses the fact that $\nextposi$ is the minimizer of $\Ft (X) + \frac{\alpha }{2}\|X-\currposi\|^2 $ and $\Ft (X) \leq F(X), \forall X \in \mathbb{S}^n$, and the first equality comes from the optimality condition in \cref{eq:optimality-condition-Xt}.

    Second, the feasibility of $\nextposi$ comes naturally from the problem construction \cref{eq:SBM-subproblem}. Further, the definition of serious step \cref{eq:null-or-serious-step-SBM} implies a cost drop. Hence, it follows  
    \begin{align*}
            F(\currposi)-F(\Xstar) & \geq F(\nextposi)-F(\Xstar) \\
            & = \langle C,\nextposi-\Xstar \rangle + \rho \max\{\lambda_{\max}(-\nextposi),0\} \\
            & = \langle C,\nextposi-\Xstar \rangle - \rho \min\{\lambda_{\min}(\nextposi),0\}.
    \end{align*}
    To further lower bound the first term, we note that \cref{assumption-slater-condition} ensures that \cref{eq:SDP-primal,eq:SDP-dual} are both feasible, and for any pair of optimal primal and dual solution $(\Xstar,\Zstar)$, complementary slackness holds, i.e., $\langle \Xstar, \Zstar\rangle = 0$. Therefore,
    \begin{align*}
            \langle C,\nextposi -\Xstar \rangle & = \langle \Zstar+\mathcal{A}^* \ystar, \nextposi-\Xstar \rangle \\
        & = \langle \Zstar, \nextposi\rangle+\langle \Zstar,-\Xstar\rangle+\langle \mathcal{A}^* \ystar, \nextposi-\Xstar\rangle \\
        & = \langle \Zstar, \nextposi\rangle + \langle \mathcal{A}^* \ystar, \nextposi-\Xstar\rangle \quad  \\
        & = \langle \Zstar, \nextposi\rangle \\
        & \geq \|\Zstar\|_* \min\{\lambda_{\min}(\nextposi),0\}\\
        & \geq \DZstar \min\{\lambda_{\min}(\nextposi),0\},
    \end{align*}
    where the third equality is due to the complementary slackness, the fourth equality uses the definition of the adjoint operator and the fact that both $\nextposi$ and $\Xstar$ are feasible. As a result, we obtain 
    \begin{align*}
        F(\currposi)-F(\Xstar) & \geq \DZstar \min\{\lambda_{\min}(\nextposi),0\}  - \rho \min\{\lambda_{\min}(\nextposi),0\} \\  & = (\DZstar-\rho)\min\{\lambda_{\min}(\nextposi),0\} \\
        & \geq -\DZstar \min\{\lambda_{\min}(\nextposi),0\},
    \end{align*}
    where the last inequality uses the assumption $\rho >  2\DZstar+1$. This completes the proof for the approximate primal feasibility.
    
    Third, by the feasibility of $\nextposi$, the duality gap follows
            \begin{equation*}
                \begin{aligned}
                    \langle C, \nextposi \rangle - \langle b,\ytstar \rangle  & =\langle C,\nextposi \rangle - \langle  \mathcal{A}(\nextposi),\ytstar \rangle   \\
                    & = \langle C,\nextposi \rangle -  \langle \Ajmap \ytstar,\nextposi \rangle \\
                    & = \langle \Wtstar, \nextposi\rangle - \langle \Ajmap \ytstar-C+\Wtstar,\nextposi \rangle.
                \end{aligned}
            \end{equation*}
        We first bound the second term using Cauchy inequality
        \begin{equation*}
            \begin{aligned}
                |\langle \mathcal{A}^* \ytstar-C+\Wtstar,\nextposi \rangle| & \leq \|\nextposi\| \|\mathcal{A}^* \ytstar-C+\Wtstar\| \\
                & \leq \|\nextposi\| \sqrt{\frac{2 \alpha}{\beta} (F(\currposi)-F(X_\star))}\\
                & \leq \Dxo \sqrt{\frac{2 \alpha}{\beta} (F(\currposi)-F(X_\star))},
            \end{aligned}
        \end{equation*}
        where the first inequality is due to 
        the approximate dual feasibility.
        The lower bound on the first term follows that
        \begin{equation*}
            \begin{aligned}
                  \langle \Wtstar, \nextposi\rangle & \geq \|\Wtstar\|_* \lambda_{\min}( \nextposi) \\
                 & \geq -\frac{\|\Wtstar\|_* (F(\currposi)-F(X_\star))}{\DZstar} \\
                 & \geq -\frac{\rho (F(\currposi)-F(X_\star))}{\DZstar},\\
            \end{aligned}
        \end{equation*}
        where the second inequality comes from 
        the approximate primal feasibility,
        and the last inequality is by the construction $\|\Wtstar\|_* \leq \rho$ in \cref{eq:constraint-W_t}. Therefore, the duality gap is lower bounded by 
        $$
           \langle C, \nextposi \rangle - \langle b,\ytstar \rangle \geq -\frac{\rho (F(\currposi)-F(X_\star))}{ \DZstar } -  \Dxo \sqrt{\frac{2 \alpha}{\beta} (F(\currposi)-F(X_\star))}.
        $$
        Similarly, by a reformulation of descent step \cref{eq:null-or-serious-step-SBM}, we obtain
        $$F(\nextposi) - \beta \Ft(\nextposi) \leq (1-\beta) F(\currposi).$$
        Adding $(\beta -1)F(\nextposi)$ from both sides and performing a simple algebra leads to 
        \begin{equation*}
            \begin{aligned}
                -\frac{1-\beta}{\beta} (F(\currposi)-F(\nextposi)) & \leq - F(\nextposi) + \Ft(\nextposi) \\
                & = -\rho \max \{\lambda_{\max}(-\nextposi),0\}  + \langle W_t^\star,-\nextposi\rangle \\
                & \leq -\langle \Wtstar, \nextposi\rangle,
            \end{aligned}
        \end{equation*}
        which implies $$\langle \Wtstar, \nextposi\rangle  \leq \frac{1-\beta}{\beta} (F(\currposi)-F(\nextposi)) \leq \frac{1-\beta}{\beta} (F(\currposi)-F(\Xstar)).$$
        Combining the last inequality with the lower bound from Cauchy inequality, we get 
        $$\langle C, \nextposi \rangle - \langle b,\ytstar \rangle \leq  \frac{1-\beta}{\beta} (F(\currposi)-F(\Xstar))  +  \Dxo \sqrt{\frac{2 \alpha}{\beta} (F(\currposi)-F(X_\star))}.$$
\end{proof}

\subsection{Proof of \cref{thm: linear-convergence}} \label{appendix:theorem-linearly-convergence}
The proof sketch of \cref{thm: linear-convergence} is presented in \Cref{subsection:proof-linear-convergence}, which relies on the results in \Cref{lemma:quadratic-closeness,lemma:linear-contraction}. 
Before proving \Cref{lemma:quadratic-closeness,lemma:linear-contraction}, we draw a few technical results. 

The first technical result, shown in \cref{lemma:quadratic-growth},  presents a powerful error bound that relates the distance of a point $X$ to the optimal solution set $\mathcal{P}^\star$ with its optimality for $F(X)$ in \cref{eq:cost-function-primal-sdp}. The proof is built on the results in \cite[Theorem 4.5.1]{drusvyatskiy2017many} and \cite[Section 4]{sturm2000error}. Given a $\epsilon >0$, we define the sublevel set for the objective function $F$ in \cref{eq:cost-function-primal-sdp} as
\begin{align}
    \mathcal{P}_\epsilon:=\{X \in \mathbb{S}^n\;|\; F(X) \leq F(\Xstar) +  \epsilon, \Amap(X) = b \}.
    \label{eq:sublevel-set-apx}
\end{align}

\begin{lemma}[Quadratic growth] \label{lemma:quadratic-growth}
    Under \Cref{assumption:linearly-independence,assumption-slater-condition},  choosing $\rho$ as in \cref{thm: linear-convergence}, there exist some constants $\zeta \geq 1$ and $\mu>0$ such that
    \begin{align} \label{eq:quadratic-growth-lemma-app}
        F(X)-F(\Xstar) \geq \mu \cdot \dist^{\zeta}(X,\Pstar), \;\; \forall X \in \mathcal{P}_\epsilon.
    \end{align}
    Furthermore, if the strict complementarity holds for \Cref{eq:SDP-primal,eq:SDP-dual}, the exponent term can be chosen as $\zeta = 2$. 
\end{lemma} 
\begin{proof}
    By \cref{proposition-primal-exact-penalty}, we know the optimal solution set of  $\min_{X\in \mathcal{X}_0} F(X) = \langle C, X\rangle + \rho \max \{\lambda_{\max} (-X) ,0 \}$ is the same as the optimal solution set $\Pstar$. ~\cref{assumption-compactness} further ensures that $\Pstar$ is compact. By \cref{lemma:compact-sublevel-set}, we know for any $\epsilon > 0$, the sub-level set $\mathcal{P}_\epsilon$ is also compact. 
    The optimal solution set $\Pstar$ can be rewritten as $\mathbb{S}^n_+ \cap \mathcal{L},$ where $\mathcal{L}= \{X \;|\;\Amap(X) = b,\; \langle C,X \rangle = \pstar\}$.    
    Then, the result in \cite[Theorem 4.5.1]{drusvyatskiy2017many} ensures that there exists two constants $k_1 > 0$ and $k_2 > 0$ such that
    \begin{equation} \label{eq:bound-1}
    \begin{aligned}
        \Dist^{2^d}(X,\Pstar) & \leq k_1 (\Dist(X,\mathcal{L}) + \Dist(X,\mathbb{S}^n_+)) \\
        &  \leq k_1 k_2 \lvert \langle C,X \rangle -  \langle C,\Xstar \rangle \rvert + k_2 n \max\{\lambda_{\max}(-X),0\}),  \quad \forall X \in \mathcal{P}_\epsilon,
    \end{aligned}
    \end{equation}
    where the second inequality applies the fact that $\Dist(X,\mathbb{S}^n_+) \leq n \max\{\lambda_{\max}(-X),0\})$ for any $X \in \mathbb{S}^n$ and $\Dist(X,\mathcal{L}) \leq k_2 \lvert \langle C,X \rangle -  \langle C,X^\star \rangle \rvert $ since $\mathcal{L}$ is an affine space. 
    
    We know turn to establish the relationship of $F(X)-F(\Xstar)$ and $\Dist(X,\mathcal{P}^\star)$ in \cref{eq:quadratic-growth-lemma}. Let $\alpha_0 = k_1 k_2$ and $\beta_0 = k_2 n$, and pick $\theta \in (0, \frac{1}{\max\{k_1,n\}k_2})$. Note that $0< \theta \alpha_0 < 1$ and $0< \theta \beta_0 < 1$. It will be sufficient to show that for some constant $\rho\geq \theta \beta_0$, the following inequality holds
    \begin{equation}
        \begin{aligned}
            \label{eq:abs-max-F}
         & \; \theta \alpha_0 \lvert \langle C,X \rangle -  \langle C,\Xstar \rangle \rvert + \theta \beta_0 \max\{\lambda_{\max}(-X),0\}) \\
        \leq  & \; \langle C,X \rangle -  \langle C,\Xstar \rangle  + \rho \max\{\lambda_{\max}(-X),0\}),  \quad \forall X \in \mathcal{P}_\epsilon.
        \end{aligned}
    \end{equation}   
%
Indeed, if $\lvert \langle C,X \rangle -  \langle C,\Xstar \rangle \rvert = \langle C,X \rangle -  \langle C,\Xstar \rangle $, the inequality \cref{eq:abs-max-F} holds naturally. On the other hand, if $\lvert \langle C,X \rangle -  \langle C,\Xstar \rangle \rvert =  \langle C,\Xstar \rangle - \langle C,X \rangle$, \cref{eq:abs-max-F} is equivalent to
$$
    \theta \alpha_0 (\langle C,\Xstar \rangle - \langle C,X \rangle) +  \theta \beta_0 \max\{\lambda_{\max}(-X),0\}) \leq \langle C,X \rangle -  \langle C,\Xstar \rangle  + \rho \max\{\lambda_{\max}(-X),0\})
$$
    which is the same as
    $$
         \langle C,\Xstar \rangle  \leq \langle C,X \rangle + \frac{\rho - \beta_0 \theta}{\theta \alpha_0 +1} \max\{\lambda_{\max}(-X),0\},
         \quad \forall X \in \mathcal{P}_\epsilon.
    $$
    This inequality holds for all $\rho$ such that $\frac{\rho - \beta_0 \theta}{\theta \alpha_0 +1} > \DZstar$ (or equivalently $\rho  > \DZstar + \theta (\beta_0 + \alpha_0 \DZstar)$) since the function in the left-hand side becomes an exact penalization in the form of \cref{eq:SDP-primal-penalized} (see \Cref{proposition-primal-exact-penalty}). 
    
    Therefore, upon further choosing $\theta \in (0,\min\{\frac{1}{\max\{k_1,n\}k_2} ,  \frac{\DZstar}{\beta_0 + \alpha_0 \DZstar}\}) $, we see that for all $\rho > 2 \DZstar$, the inequality \cref{eq:abs-max-F} holds.
    Combining \cref{eq:abs-max-F} with \Cref{eq:bound-1} leads to for all $ \rho >  2 \DZstar $,
    $$
        F(X)-F(\Xstar) \geq 
        \mu \cdot \dist^{2^d}(X,\Pstar), \quad \forall X \in \mathcal{P}_\epsilon,
    $$
    where $\mu = \frac{1}{ \theta} > 0$ and $d$ is the singularity degree of \cref{eq:SDP-primal} and is bounded \cite[Lemma 3.6]{sturm2000error}. Furthermore, if strict complementarity holds, the singularity degree $d$ is at most one \cite[Section 5]{sturm2000error}. 
\end{proof} 

The following technical lemma characterizes the quality of eigenvalue approximation.
\begin{lemma}[{\cite[Lemma 3.9]{ding2020revisit}}]\label{lem: importantLemmaQuadraticAccurateModel}
	 Given $X\in \mathbb{S}^n$ with the difference between the top r-th and (r+1)-th eigenvalue defined as $\delta = \lambda_{r}(X)-\lambda_{r+1}(X)$ and denote $\Lambda_{r,\dm}(X)=\max\{|\lambda_{r+1}(X)|,|\lambda_\dm (X)| \}$ and $\mathcal{C}^+_r (X):=\{VSV^{\top}\mid\Trace(S)\leq1,S\succeq0,S\in\symMat^{r}\}$, where $V \in \mathbb{R}^{n \times r}$ contains the r orthonormal eigenvectors corresponding to the largest $r$ eigenvalues of $X$. Then for any $Y\in \mathbb{S}^n $, the function $f_X: \mathbb{S}^n \to \RR$ defined as $ f_X(Y):\,=\max\{\lambda_1(Y),0\}-\max_{W\in \faceplus r (X)}\inprod{W}{Y}$ satisfies that 
 \begin{equation*} 
    \begin{aligned} 
        0 \leq f_X(Y) \leq \frac{8 \|Y-X\|^2 \Lambda_{r,n}(X)}{\delta^2} + \frac{(8+\sqrt{2}+16) \|Y-X\|^2 }{\delta}.
    \end{aligned}
\end{equation*}
\end{lemma}

We are now ready to prove \Cref{lemma:quadratic-closeness}. For convenience, we state a refined version of \Cref{lemma:quadratic-closeness} below, which shows that the lower approximation model $\Ft$ becomes quadratically close to the true penalized cost function $F$ after some finite number of iterations $T_0$. 

\begin{lemma}[Quadratic closeness] \label{lemma:quadratic-close}
    Suppose strong duality and strict complementarity hold for \cref{eq:SDP-primal} and \cref{eq:SDP-dual}. Let $r_{\mathrm{c}} \geq \rnull$. There exists a constant $\eta > 0$ (independent of $\epsilon$) such that for all iteration $t \geq T_0$, it holds that
    \begin{align}\label{eq:quadratic-close}
         \Ft (X) \leq F(X) \leq \Ft (X) + \frac{\eta}{2}\|X-\currposi\|^2,\quad \forall X \in \mathcal{X}_0.
    \end{align}    
     In particular, we can choose $\eta =4 \rho \max \biggl\{ \frac{144 \sup_{\Xstar \in \Pstar} \|\Xstar\|_{\mathrm{op}}}{\delta^2},\frac{9(8\sqrt{2}+16)}{\delta} \biggr\}$, where 
     \begin{equation*}
     \delta := \inf_{\Xstar \in \Pstar} \sup_{r \leq \rnull} \lambda_r(-\Xstar) - \lambda_{r+1}(-\Xstar)       
     \end{equation*}
     is the eigenvalue gap parameter.
\end{lemma}
\begin{proof}
    The proof is divided into two cases: unique solution and multiple solutions. We first consider the case when $\Xstar \in \Pstar$ is unique. Since the primal solution $\Xstar$ is unique, we can choose $T_0$ large enough to ensure iterate $\currposi$ is sufficiently close to $\Xstar$ in the way that $\|\currposi-\Xstar\|_{\mathrm{op}} \leq \frac{\delta}{3} $ (see \cref{subsec:estimation-T} for an estimate on the number iterations to achieve this), where $\delta$ is $ \lambda_{\rnull}(-\Xstar) - \lambda_{\rnull+1}(-\Xstar)$ in this case. By Weyl's inequality and the triangle inequality, we know
    \begin{align*}
        \lambda_{\rnull}(-\Xstar)  -  \lambda_{\rnull}(-\currposi)\leq \frac{\delta}{3} ,~\lambda_{\mathrm{\rnull}+1}(-\currposi)  -  \lambda_{\rnull+1}(-\Xstar)\leq \frac{\delta}{3} ,\text{ and } \|\currposi\|_{\mathrm{op}} - \|\Xstar\|_{\mathrm{op}} \leq \frac{\delta}{3} \leq \delta.
    \end{align*}
    Summing up the first two inequalities with the definition of $\delta$ yields
    \begin{subequations}
        \begin{equation}
            \lambda_{\rnull}(-\currposi) - \lambda_{\rnull+1}(-\currposi) \geq \frac{\delta}{3} \label{eq:consequence-T0-1} \\
        \end{equation}
        \text{and the fact $\delta \leq \|\Xstar\|_{\mathrm{op}}$ implies}
         \begin{equation}
             \|\currposi\|_{\mathrm{op}} \leq 2\|\Xstar\|_{\mathrm{op}} \label{eq:consequence-T0-2}.
         \end{equation}
    \end{subequations}
    Let $P \in \mathbb{R}^{n \times \rnull}$ denote the matrix formed by $\rnull$ orthonormal eigenvectors corresponding to the largest $\rnull$ eigenvalues of $-\currposi$, we have 
    \begin{equation*}   
    \begin{aligned} 
        F(X) - \Ft(X) & 
        = \rho \max \{\lambda_{\max}(-X),0\} - \rho \max_{\gamma \geq 0,\gamma + \Trace(S) \leq 1, S \in \mathbb{S}^{r}_+} \biggl \langle  \gamma \bar{W}_t + P_tSP_t^\tr,-X \biggr \rangle \\
        & \leq \rho \max \{\lambda_{\max}(-X),0\} - \rho \max_{\Trace(S) \leq 1, S \in \mathbb{S}^{\rnull}_+} \biggl \langle  PSP^\tr,-X \biggr \rangle \\
        & \leq \rho \left( \frac{ 9 \times 8 \|X-\currposi\|^2 \Lambda_{r,n}(-\currposi)}{\delta^2} + \frac{9 \times(8+\sqrt{2}+16) \|X-\currposi\|^2 }{\delta}\right) \quad \\
        & \leq \rho \left(\frac{144 \|X-\currposi\|^2 \|\Xstar\|_{\mathrm{op}}}{\delta^2} + \frac{9(8\sqrt{2}+16) \|X-\currposi\|^2 }{\delta}\right)\\
        & \leq 2\rho \max   \Biggl\{ \frac{144\|\Xstar\|_{\mathrm{op}}}{\delta^2},\frac{9(8\sqrt{2}+16)}{\delta} \Biggr\} \|X-\currposi\|^2,
    \end{aligned}
\end{equation*} 
where the first inequality is due to the restriction on the feasible set since we choose $\rcurrent \geq \rnull$ (recall $r = \rpast + \rcurrent$), 
the second inequality applies \cref{lem: importantLemmaQuadraticAccurateModel} and \cref{eq:consequence-T0-1}, and the third inequality is due to the fact $\Lambda_{r,n}(-\currposi) \leq \|\currposi\|_{\mathrm{op}}$ and \cref{eq:consequence-T0-2}. As a result, choosing $\eta =4 \rho \max \biggl\{ \frac{144\|\Xstar\|_{\mathrm{op}}}{\delta^2},\frac{9(8\sqrt{2}+16)}{\delta} \biggr\} $ does provide a quadratic closed upper bound in \cref{eq:quadratic-close}.

Second, we consider the case of $\Pstar$ containing multiple solutions. Similar to the case of unique solution, we choose $T_0$ large enough to ensure  $\inf_{\Xstar \in \Pstar} \|\currposi-\Xstar\|_{\mathrm{op}} \leq \frac{\delta}{3}$. Therefore, there exists a $\Xstar \in \Pstar$ and $\Bar{r} \leq \rnull$ such that $\|\currposi-\Xstar\|_{\mathrm{op}} \leq \frac{\delta}{3}$ (see \cref{subsec:estimation-T} for an estimate on the number iterations to achieve this) and $ \lambda_{\Bar{r}}(-\Xstar) - \lambda_{\Bar{r}+1}(-\Xstar) \geq \delta$, which implies $\lambda_{\Bar{r}}(-\currposi) - \lambda_{\Bar{r} +1}(-\currposi) \geq \frac{\delta}{3} $ and $\|\currposi\|_{\mathrm{op}} \leq 2 \sup_{\Xstar \in \Pstar} \|\Xstar\|_{\mathrm{op}}$. As a result, a similar argument as the case of unique solution can be made by replacing $\rnull$ and $\|\Xstar\|_{\mathrm{op}}$ with $\Bar{r}$ and $ \sup_{\Xstar \in \Pstar} \|\Xstar\|_{\mathrm{op}}$. We conclude that choosing $\eta =\BoundOnetaP$ is sufficient to ensure a quadratic closed upper bound in \cref{eq:quadratic-close}. 
\end{proof}

\Cref{lem: importantLemmaQuadraticAccurateModel,lemma:quadratic-close} will allow us to prove the linear convergence in terms of the distance to the optimal solution set (we show it in \Cref{lemma:contraction}), i.e., \cref{eq:contraction-iterate} in \Cref{lemma:linear-contraction}. To further prove the linear convergence in terms of cost value \cref{eq:contraction-cost} in \Cref{lemma:linear-contraction}, we need another two technical results: \cref{lemma:explicit-progress-proximal,lemma:cost-each-step}.
\begin{lemma}[{\cite[Lemma 7.12]{ruszczynski2011nonlinear}}] \label{lemma:explicit-progress-proximal}
    Fix a minimizer $\Xstar$ in $\Pstar$. Let $F_{\alpha}(W) := \min_{ X \in \mathcal{X}_0 } F(X) + \frac{\alpha }{2} \|X-W\|^2 $ denote the proximal mapping. For any $W \in \mathcal{X}_0\backslash \{\Xstar\} $, the following result holds\footnote{We noticed a typo in {\cite[Lemma 7.12]{ruszczynski2011nonlinear}} and we correct it here. Although {\cite[Lemma 7.12]{ruszczynski2011nonlinear}} considers an unconstrained case, the same result holds for the case with a convex constraint set here.}.
    \begin{align*}
        F_\alpha(W) \leq F(W) - \frac{\alpha}{2}\|W - \Xstar\|^2 \Phi\left(\frac{F(W)-\pstar}{\alpha\|W-\Xstar\|^2}\right),
    \end{align*}
    where 
    \begin{align*}
        \Phi(t) = 
        \begin{cases}
            t^2, & \text{if} \quad 0\leq t \leq 1,\\
            -1+2t, & \text{if} \quad t> 1. 
        \end{cases}
    \end{align*}
\end{lemma}
\begin{proof}
    This is the constrained version of the result in \cite[Lemma 7.12]{ruszczynski2011nonlinear}. The proof follows the same argument in \cite[Lemma 7.12]{ruszczynski2011nonlinear} and we omit it here.
\end{proof}

\Cref{lemma:explicit-progress-proximal} above shows that the progress made by a proximal mapping can be calculated precisely. The following result shows the linear decrease of a descent step in the objective value under the assumption of quadratic growth.


\begin{lemma}[{\cite[Lemma 4.1 and 4.2]{du2017rate}}]\label{lemma:cost-each-step}
     Supposed \Cref{assumption:linearly-independence,assumption-slater-condition} are satisfied.
     If the function further satisfies the following quadratic growth property 
     \begin{align} \label{eq:strongly-convex-secord-order}
         F(X)  -\pstar \geq c \cdot \dist^2(X,\Pstar),\quad \forall X \in \mathcal{P}_{F(\Omega_0)},
     \end{align}
     where $c > 0$ is a constant and $\mathcal{P}_{F(\Omega_0)}$ is the sublevel set defined as \cref{eq:sublevel-set-apx}, then it follows that
     \begin{align}\label{eq:proximal-improve}
         F(\currposi) - \pstar  \leq \frac{F(\currposi)- \Bigl(\Ft(\candidate )+\frac{\alpha}{2} \|\candidate -\currposi\|^2\Bigr)}{\min\{\frac{c}{2 \alpha},\frac{1}{2}\}}, \; \forall t \geq 0.
     \end{align}
     If further, a descent step occurs at iteration $t$, then the reduction in cost value gap satisfies
     $$F(\nextposi) - \pstar \leq (1- \Bar{\mu} \beta) (F(\currposi) - \pstar),$$
    where $\Bar{\mu}  = \min\{\frac{c}{2 \alpha},\frac{1}{2}\}$.
\end{lemma}

\begin{proof}
    By definition on $F_{\alpha}$, we know
    \begin{align}
        F_{\alpha}(\currposi) &  = \min_{ X \in \mathcal{X}_0 } F(X) + \frac{\alpha }{2} \|X-\currposi\|^2  \nonumber\\
        &\geq  \min_{X \in \mathcal{X}_0} \Ft(X) + \frac{\alpha }{2} \|X-\currposi\|^2 \nonumber\\
        & = \Ft(\candidate) + \frac{\alpha }{2} \|\candidate-\currposi\|^2, \label{eq:F-alpha-bound}
    \end{align}
    where $\candidate$ is the minimizer. By \cref{lemma:explicit-progress-proximal}, if $F(\currposi)- \pstar \leq \alpha \dist^2(\currposi,\Pstar) = \alpha \|\currposi-\Xstar\|^2$, where $\Xstar$ is the closest point to $\currposi$ in $\Pstar$, and $\currposi \neq \Xstar$, we have 
    \begin{align*}
        F_\alpha(\currposi) \leq F(\currposi) - \frac{(F(\currposi)-\pstar)^2}{2\alpha \|\currposi-\Xstar\|^2  }.
    \end{align*}
    Combining this inequality with \cref{eq:F-alpha-bound} leads to 
    \begin{align*}
        \frac{(F(\currposi)- \pstar )^2}{2\alpha \|\currposi-\Xstar\|^2} \leq F(\currposi) - \Ft(\candidate) - \frac{\alpha}{2}  \|\candidate-\currposi\|^2.
    \end{align*}
    By applying \cref{eq:strongly-convex-secord-order} and this inequality, we get
    \begin{align}\label{eq:cost-each-step-1}
        F(\currposi)- \pstar  \leq \frac{ F(\currposi) - \Ft(\candidate) - \frac{\alpha}{2} \|\candidate-\currposi\|^2}{\frac{c}{2 \alpha}}.
    \end{align}
    On the other hand, if $ F(\currposi)- \pstar > \alpha \dist^2(\currposi,\Pstar) =\alpha \|\currposi-\Xstar\|^2$, where $\Xstar$ is the closest point to $\currposi$, \cref{lemma:explicit-progress-proximal} ensures
    \begin{align*}
        F_\alpha(\currposi) \leq \pstar + \frac{\alpha}{2} \|\currposi - \Xstar\|^2 <   \frac{1}{2} (F(\currposi)+\pstar).
    \end{align*}
    Combining \cref{eq:F-alpha-bound} with this inequality yields
    \begin{align*}
        -\frac{1}{2}(F(\currposi)+\pstar)  < - \Ft(\candidate) - \frac{\alpha}{2} \|\candidate-\currposi\|^2.
    \end{align*}
    By adding $F(\currposi)$ on both sides, we conclude
    \begin{align*}
        F(\currposi)-\pstar \leq 2\left(F(\currposi) - \Ft(\candidate) - \frac{\alpha}{2} \|\candidate-\currposi\|^2\right).
    \end{align*}
    Combining \cref{eq:cost-each-step-1} and the above inequality completes the proof for \cref{eq:proximal-improve}.
    
    If further a descent step occurs and by a reformulation of the definition of the descent step, we have 
     $$F(\nextposi)  \leq (1-\beta) F(\currposi) +\beta \Ft(\nextposi).$$
    By \cref{eq:proximal-improve}, we have $$F(\currposi) - \pstar \leq \frac{F(\currposi)-\Ft(\nextposi)-\frac{\alpha}{2} \|\nextposi -\currposi\|^2}{\Bar{\mu} }  \leq \frac{F(\currposi)-\Ft(\nextposi)}{\Bar{\mu} }.$$
    Combining the above two inequalities leads to 
    \begin{align*}
        F(\nextposi) & \leq (1-\beta)F(\currposi) + \beta F(\currposi)- \beta\Bar{\mu}  (F(\currposi)-\pstar) \\
        & = F(\currposi) - \beta \Bar{\mu} (F(\currposi)-\pstar).
    \end{align*}
    Substracting $\pstar$ from both sides yields $$F(\nextposi) - \pstar \leq (1-\beta \Bar{\mu}) (F(\currposi)-\pstar).$$
\end{proof}
With \Cref{lem: importantLemmaQuadraticAccurateModel,lemma:quadratic-close,lemma:explicit-progress-proximal,lemma:cost-each-step}, we are ready to prove \cref{lemma:linear-contraction}, which is the key to prove  \cref{thm: linear-convergence} (as discussed in \Cref{subsection:proof-linear-convergence}). 
For convenience, we restate \cref{lemma:linear-contraction} below
\begin{lemma}[{\cref{lemma:linear-contraction} in the main text}] \label{lemma:contraction} 
Under the conditions in \cref{lemma:quadratic-close}, for any $\alpha \geq \eta$ and $t \geq {T}_0$,  \Cref{alg:(r_p-r_c)-SBM-P} with $\beta \in (0, \frac{1}{2}]$ takes only descent steps and guarantees two contractions: 
	\begin{align*}
        \dist(\nextposi,\Pstar) &\leq \sqrt{ \frac{{\alpha}/{2}}{\mu +{\alpha}/{2}} }\dist(\currposi,\Pstar), 
        \\ 
        F(\nextposi) - F(\Xstar) &\leq \left(1- \min\left\{\frac{\mu}{2 \alpha},\frac{1}{2} \right\} \beta\right) \left(F(\currposi) - F(\Xstar)\right),
 \end{align*}
 where $\mu$ is the quadratic growth constant in \cref{eq:quadratic-growth-F(x)}. 
\end{lemma}
\begin{proof}
    We first show that after $T_0$ iterations (\cref{eq:quadratic-close} holds), \Cref{alg:(r_p-r_c)-SBM-P} will only take descent steps. The proof is organized into the following two steps:
    \begin{itemize}
        \item Choosing $\alpha \geq \eta$ ensures that the descent yielded by the candidate $\candidate$ can be estimated by
        $$
           F(\currposi)  - F(\candidate) \geq   \frac{\alpha} {2}\|\candidate-\currposi\|^2.
        $$        
        \item Choosing $\beta \in (0,\frac{1}{2}]$ and $\alpha \geq \eta$ guarantees the descent condition \cref{eq:null-or-serious-step-SBM} holds.
    \end{itemize}    
    Without loss of generality, we set $ \eta = \alpha$ since we require $\eta \leq \alpha $. Since $\Ft +\frac{\alpha}{2}\|\cdot-\currposi\|^2$ is $\alpha$ strongly convex, applying the first order inequality for a strongly convex function leads to the following, $\forall X \in \mathbb{S}^n$ such that $\Amap(X) = b$,
        \begin{align} \label{eq:strongly-convex}
            & \; \Ft(X) +\frac{\alpha}{2}\|X-\currposi\|_F^2 \nonumber \\
             \geq & \; \bar{F}_t(\candidate) +\frac{\alpha}{2}\|\candidate-\currposi\|^2 + \langle g_t,X-\candidate \rangle + \frac{\alpha}{2}\|X - \candidate\|^2 \nonumber\\
             = & \; \bar{F}_t(\candidate) +\frac{\alpha}{2}\|\candidate-\currposi\|^2 + \langle \Ajmap(\ytstar) ,X-\candidate \rangle + \frac{\alpha}{2}\|X - \candidate\|^2 \nonumber \\
             = & \; \Ft(\candidate) +\frac{\alpha}{2}\|\candidate-\currposi\|^2 + \frac{\alpha}{2}\|\candidate-X\|^2,  
        \end{align}
    where $g_t$ is the gradient of $\Ft  + \frac{\alpha}{2} \|\cdot-\currposi\|^2$ evaluated at $\candidate$ and the second equality uses the optimality condition in \cref{eq:optimality-condition-Xt}.  By setting $X = \currposi$ in \cref{eq:strongly-convex} and combining the fact $F(X) \geq \Ft(X) \text{ for all } X \in \mathbb{S}^n$, we obtain
     \begin{equation*}\label{eq:strongly-convex+quadratic close}
        \begin{aligned}
            F(\currposi) \geq \Ft(\currposi) & \geq \Ft(\candidate) +\frac{\alpha}{2}\|\candidate-\currposi\|^2 + \frac{\alpha}{2}\|\candidate-\currposi\|^2,\\
            & = \bar{F}_t(\candidate) + \alpha \|\candidate-\currposi\|^2,
        \end{aligned}
    \end{equation*} 
    which implies 
    \begin{align} \label{eq:first-order-reform}
        F(\currposi) - \Ft(\candidate) &\geq \alpha \|\candidate-\currposi\|^2.
    \end{align}
    On the other hand, letting $X = \candidate$ in \cref{eq:quadratic-close} yields
    \begin{equation} \label{eq:quadratic-close-reform}
        \begin{aligned}
            \bar{F}_t(\candidate) \geq F(\candidate) - \frac{\alpha}{2}\|\candidate-\currposi\|^2  .
        \end{aligned}
    \end{equation} 
    Therefore, combining \cref{eq:first-order-reform,eq:quadratic-close-reform} leads to
    \begin{align}\label{eq:quadratic-growth}
        F(\currposi)  - F(\candidate) \geq   \frac{\alpha} {2}\|\candidate-\currposi\|^2.
    \end{align}
    
    We now move on to show that choosing $\beta \in (0,\frac{1}{2} ]$ and $\alpha \geq \eta$ yields a descent step. By the choice of $\beta$, it holds that
     \begin{equation*}
        \begin{aligned}
            \beta (F(\currposi)-\bar{F}_t(\candidate))& \leq \frac{1}{2} (F(\currposi)-\bar{F}_t(\candidate)) \\
           & \leq \frac{1}{2} (F(\currposi)-F(\candidate)) +\frac{\alpha}{4}\|\candidate-\currposi\|^2\\ 
            &\leq F(\currposi)-F(\candidate),
        \end{aligned}
    \end{equation*}
    where the second inequality applies \cref{eq:quadratic-close-reform}, and the last inequality uses \cref{eq:quadratic-growth}. This shows that a descent step indeed will occur.

    Second, we show the distance $\dist(\currposi,\Pstar)$ is shrinking with a linear rate. Upon applying \cref{lemma:quadratic-growth} with $X= \nextposi$, we obtain 
    \begin{align} \label{eq:quadratic-growth-nextposi}
        \mu \cdot \dist^{2}(\nextposi,\Pstar) \leq F(\nextposi)-F(\Xstar).
    \end{align}
    On the other hand, for any $\Xstar \in \Pstar$, with the replacement of $\candidate = \nextposi$ since a decent step happens, \cref{eq:strongly-convex} ensures
     \begin{equation}\label{eq:linear-converge-key1}
        \begin{aligned}
            F(\Xstar) +\frac{\alpha}{2}\|\Xstar-\currposi\|^2 & \geq \bar{F}_t(\nextposi) +\frac{\alpha}{2}\|\nextposi -\currposi\|^2 + \frac{\alpha}{2}\|\nextposi - \Xstar \|^2,
        \end{aligned}
    \end{equation}   
    where we also use the global lower bound property of the approximation model.
    Combining \cref{eq:linear-converge-key1} and \cref{eq:quadratic-close-reform}, we get
    \begin{equation*}
        \begin{aligned}
            F(\Xstar) +\frac{\alpha}{2}\|\Xstar-\currposi\|^2 & \geq F(\nextposi) - \frac{\alpha}{2}\|\nextposi-\currposi\|^2  +\frac{\alpha}{2}\|\nextposi -\currposi\|^2 + \frac{\alpha}{2}\|\nextposi - \Xstar \|^2 \\
            & = F(\nextposi) + \frac{\alpha}{2}\|\nextposi - \Xstar \|^2,
        \end{aligned}
    \end{equation*}
    which implies 
    \begin{align}\label{eq:relation-dist-gap}
        F(\nextposi) -  F(\Xstar) & \leq \frac{\alpha}{2}\|\Xstar-\currposi\|^2 - \frac{\alpha}{2}\|\nextposi - \Xstar \|^2.
    \end{align}
    With \cref{eq:quadratic-growth-nextposi,eq:relation-dist-gap}, we reach 
    \begin{align*}
         \mu \cdot \dist^{2}(\nextposi,\Pstar) \leq \frac{\alpha}{2}\|\Xstar-\currposi\|^2 - \frac{\alpha}{2}\|\nextposi - \Xstar \|^2.
    \end{align*}
    By setting $\Xstar$ to be the closest point to $\currposi$, it follows 
    \begin{align*}
         \mu \cdot \dist^{2}(\nextposi,\Pstar) & \leq \frac{\alpha}{2}\dist^2(\currposi,\Pstar)- \frac{\alpha}{2} \|\nextposi - \Xstar \|^2 \\
         & \leq \frac{\alpha}{2}\dist^2(\currposi,\Pstar) - \frac{\alpha}{2} \dist^2(\nextposi,\Pstar),
    \end{align*}
    which implies 
    \begin{align*}
       \dist(\nextposi,\Pstar) &  \leq  \sqrt{ \frac{\frac{\alpha}{2}}{\mu +\frac{\alpha}{2}} }\dist(\currposi,\Pstar).
    \end{align*}
    Since the contraction factor $\sqrt{ \frac{\alpha/2}{\mu +\alpha/2} }  < 1$, the distance is converging geometrically. 

    Third, we show the cost value gap $F(\currposi) - \pstar $ is shrinking geometrically. This is in fact a direct application of \Cref{lemma:quadratic-growth} and \Cref{lemma:cost-each-step}. Since we have shown above that only the descent step will occur after $T_0$ iteration, the objective value will decrease linearly as 
    \begin{align*}
        F(\nextposi) - \pstar \leq \left(1- \min\left\{\frac{\mu}{2 \alpha},\frac{1}{2} \right\} \beta\right) \left(F(\currposi) - \pstar\right).
    \end{align*}    

\end{proof}
\subsection{Bound on $T_0$} \label{subsec:estimation-T}

In the proof of \cref{lemma:quadratic-close}, we choose ${T}_0$ to be the number of iterations that ensure the iterate $\Xstar$ generated by \SBMP~is $\delta/3$ closed to the optimal solution set, i.e., $\inf_{\Xstar \in \Pstar}  \|\currposi - \Xstar\|_2 \leq \delta/3$, where $\delta := \inf_{X \in \Pstar} \sup_{r \leq \rnull} \lambda_r(-X) - \lambda_{r+1}(-X)$. 
The bound can be obtained by applying \cref{lemma-iterations-bound}.
By the fact that Frobenius norm is greater than the spectral norm and \cref{lemma:quadratic-growth}, we know
\begin{align*}
    \|\currposi - \Xstar\|_{\mathrm{op}} \leq \|\currposi - \Xstar\| = \dist(\currposi,\Pstar) \leq \sqrt{ (F(\currposi)-F(\Xstar))/\mu},
\end{align*}
where $\mu$ is the quadratic growth constant in \cref{lemma:quadratic-growth}.
Therefore, choosing $F(\currposi) - F(\Xstar) \leq \frac{\mu\delta^2}{9}$ is sufficient to ensure the iterate $\currposi$ is $\delta/3$ close to the optimal solution set $\Pstar$. Using the bound for the maximal number of iterations under the quadratic growth assumption in \cref{lemma-iterations-bound}, the bound on ${T}_0$ can be computed as 
$$\BoundOnTp,$$
where we use the fact that $F$ is $\max\{\|C\|_{\mathrm{op}}, 1\}$-Lipschitz.

    \section{Further computational details}
\label{Appendix:data-generation-SDP-SOS}
\subsection{Data generation for random SDPs}
\label{subsec:data-generation}
\begin{algorithm}[t] 
{\small
    \caption{Generate SDPs satisfying strict complementarity and constant trace property.}
    \label{alg:generate-sdp-strict}
    \begin{algorithmic}[1]
    \State Choose $r < n \in \mathbb{N}$ as the rank of primal solution $\Xstar$
    \State Randomly generate $S$ such that $S\in \mathbb{S}^n_{++}$
    \State Randomly generate the constraint matrices $A_1, \ldots, A_m \in \mathbb{S}^n$ with $\Trace(A_i) = 0, i = 1 ,\ldots,m$
    \State Compute eigen-decomposition on $S$ as $S = U_1 \Sigma_1 U_1^\tr + U_2 \Sigma_2 U_2^\tr$
    \State Randomly generate a dual solution $\ystar \in \mathbb{R}^m$
    \State Set $\Xstar = U_1 \Sigma_1 U_1^\tr$, $\Zstar = U_2 \Sigma_2 U_2^\tr$, and $C  = \Zstar + \Ajmap(\ystar)$
    \State \textbf{Return} problem data $(C,A_1,\ldots,A_m,b)$ and a pair of optimal solutions $(\Xstar,\ystar,\Zstar)$
    \end{algorithmic}
    }
\end{algorithm}
To examine the influence of $\rcurrent$ to \Cref{alg:(r_p-r_c)-SBM-P} and \Cref{alg:(r_p-r_c)-SBM-D}, we randomly generate the data that satisfy strict complementarity and have the constant trace property on the dual variable $Z$ in \cref{eq:SDP-dual}. We first randomly generate a positive definite matrix $S \in \mathbb{S}^n_{++}$ and the primal constraint matrices $A_1, \ldots, A_m \in \mathbb{S}^n$ with $\Trace(A_i) = 0, i = 1 ,\ldots,m.$ The matrix $S$ is decomposed  as 
\begin{align*}
    S =U_1 \Sigma_1 U_1^\tr + U_2 \Sigma_2 U_2^\tr,
\end{align*}
where $\Sigma_1 \in \mathbb{S}^r$ is the matrix formed by $r$ orthnormal eigenvectors correpsonding the largest $r$ eigenvalues and $r$ is a given number. We then set $\Xstar = U_1 \Sigma_1 U_1^\tr$, $\Zstar =U_2 \Sigma_2 U_2^\tr$, and $b = \Amap\left(\Xstar\right)$. To generate $C$, we randomly generate $y^\star \in \mathbb{R}^m$ and set $C = \Zstar + \Ajmap(y^\star)$. The procedure is presented in \Cref{alg:generate-sdp-strict}.

Indeed, the solution pair $(\Xstar, \ystar,\Zstar)$ generated by \Cref{alg:generate-sdp-strict} is an optimal solution.  To see this, we note that both $\Xstar$ and $\Zstar$ are positive semidefinite and satisfy the affine constraint in \cref{eq:SDP-primal} and \cref{eq:SDP-dual} respectively. The complementarity slackness is satisfied as $\langle \Xstar, \Zstar \rangle = 0$. There also exists a strictly feasible point in the primal SDP as $\Amap(\Xstar + t I) = b$ and $\Xstar + t I $ becomes positive definite as long as $t$ is sufficiently large.
The constant trace property of $Z$ can be observed from $\Trace(Z) = \Trace (C - \Ajmap(y)) = \Trace(C) - \sum_{i=1}^m y_i\Trace(A_i) = \Trace(C)$ as $\Trace(A_i) = 0, i=1,\ldots,m. $ 

\subsection{Exact penalty parameter for SDPs from sum-of-squares optimization}
\label{subsec:Exaxt-penalty-SOS}
One requirement for running \Cref{alg:(r_p-r_c)-SBM-P} is to choose a valid penalty parameter $\rho $ in \Cref{proposition-primal-exact-penalty}. In the SDPs from sum-of-squares optimization over a single sphere constraint, one valid lower bound for the parameter $\rho$ can be obtained a priori. We summarize this result in the following lemma.  
\begin{lemma}[{\cite[Lemma 5]{mai2023hierarchy}}]\label{lemma:const-trace-POP}
    Consider a $k$th-order SOS relaxation of a problem minimizing a polynomial $p_0(x)$ over one single sphere constraint, i.e.,
    \begin{equation}
    \begin{aligned}\label{eq:k-order-Lassere-relaxation-shpere}
    \max_{\gamma, \sigma_0, \psi_1} & \quad \gamma \\
     \mathrm{subject~to} & \quad  p_0(x) - \gamma -\psi_1 h(x) = \sigma_0 ,\\
                         & \quad \sigma_0 \in \Sigma[x]_{n,2k}, \;\psi_1 \in \RR[x]_{n,2(k-\lceil h \rceil)},
    \end{aligned}
\end{equation}
where $h(x) = \|x\|^2 -\bar{R}$ with $\bar{R}> 0$, $\lceil h \rceil = \lceil \mathrm{deg}(h)/2 \rceil$, $\RR[x]_{n,2(k-\lceil h \rceil)}$ denotes the real polynomial in $n$ variables and degree at most $2(k-\lceil h \rceil)$, and $\Sigma[x]_{n,2k}$ denotes the cone of SOS polynomials in $ \RR[x]_{n,2k}$. 
Then, any $\rho \geq \bigl(1+\bar{R}\bigr)^k$ is a valid exact penalty parameter that satisfies $\rho > \Trace(\Zstar)$, where $\Zstar$ is any dual optimal solution the SDP from \Cref{eq:k-order-Lassere-relaxation-shpere}.
\end{lemma}
\begin{proof}
     Let us first clarify some necessary notations: For $x \in \RR^n$ and $\alpha \in \mathbb{N}^n$, we define the monomial $x^\alpha = x_1^{\alpha_1}x_2^{\alpha_2}\cdots x_n^{\alpha_n}$, note $|\alpha| = \sum_{i=1}^n \alpha_i$, and let $\mathbb{N}^n_t = \{\alpha \in \mathbb{N}^n \;|\; |\alpha| \leq t \}$. 
     
     The proof is a simple adaption of \cite[Lemma 5]{mai2023hierarchy}.
    Given a $k\in \mathbb{N}$, let $\bigl(1+\|x\|^2\bigr)^k = \sum_{\alpha \in \mathbb{N}^n_k} \theta_{k,\alpha} x^{2 \alpha},$ where $\{\theta_{k,\alpha}\}$ is the sequence of the coefficients of polynomial $\bigl(1+\|x\|^2\bigr)^k$, and let $P_{n,k} = \mathrm{diag}\bigl(\sqrt{\theta_{k,\alpha}}\bigr)_{\alpha \in \mathbb{N}^n_k}$.  
    By the result of \cite[Lemma 5]{mai2023hierarchy}, the variable $P_{n,k} Z P_{n,k}^\tr$ has a constant trace property, i.e., $\Trace(P_{n,k} Z P_{n,k}^\tr) = \bigl(\bar{R}+1\bigr)^k$, where $Z\in \mathbb{S}^{\binom{n+k}{n}}_+$ is any feasible dual point of \cref{eq:k-order-Lassere-relaxation-shpere}. By the construction of $P_{n,k}$, the diagonal elements of $P_{n,k}$ are always greater or equal to $1$ and some of them are strictly greater than $1$. Therefore, 
    \begin{align*}
        \Trace\left(Z\right) < \Trace \left( P_{n,k} Z P_{n,k}^\tr \right)  = \left(1+\bar{R}\right)^k.
    \end{align*}
    Since this holds for any feasible $Z$, taking $Z = \Zstar$ completes the proof.
\end{proof}


\end{document}